\documentclass[10pt]{amsart}

\usepackage[headings]{fullpage}

\usepackage{amsfonts,amssymb,amsthm,amsmath,euscript,xypic,picins,float,graphicx}

\usepackage[colorlinks,citecolor=blue,linkcolor=red]{hyperref}

\newtheorem{thm}{Theorem}[section]
\newtheorem{lem}[thm]{Lemma}
\newtheorem{qst}[thm]{Question}
\newtheorem{prop}[thm]{Proposition}
\newtheorem{cor}[thm]{Corollary}
\theoremstyle{definition}
\newtheorem{df}[thm]{Definition}
\newtheorem{rk}[thm]{Remark}

\newtheorem{conv}[thm]{Convention}

\newtheorem{notation}[thm]{Notation}

\newtheorem{theor}{Theorem}

\newcommand{\R}{\mathbb R}
\newcommand{\Z}{\mathbb Z}

\newcommand{\N}{\mathbb N}

\def\epsilon{\varepsilon}
\def\phi{\varphi}

\def\hat{\widehat}

\newcommand{\Out}{\mbox{Out}}
\newcommand{\Aut}{\mbox{Aut}}

\newcommand{\ind}{\mbox{ind}}

\newcommand{\gind}{\mbox{ind}_{\rm geom}}

\newcommand{\Ind}{\mbox{ind}}

\begin{document}

\title{A train track directed random walk on $Out(F_r)$}
\author{Ilya Kapovich and Catherine Pfaff}

\address{\tt  Department of Mathematics, University of Illinois at Urbana-Champaign,
  1409 West Green Street, Urbana, IL 61801
  \newline \indent  {\url{http://www.math.uiuc.edu/~kapovich, }}} \email{\tt kapovich@math.uiuc.edu}

\address{\tt  Faculty of Mathematics, University of Bielefeld, Postfach 100131
   Universit\"atsstrasse 25, D-33501 Bielefeld, Germany
  \newline \indent {\url{http://www.math.uni-bielefeld.de/~cpfaff/}}}
  %http://www.math.uni-bielefeld.de/\~{}cpfaff/}
    \email{\tt cpfaff@math.uni-bielefeld.de}

\begin{abstract}
Several known results, by Rivin, Calegari-Maher and Sisto, show that an element $\phi_n\in\Out(F_r)$, obtained after $n$ steps of a simple random walk on $\Out(F_r)$, is fully irreducible with probability tending to 1 as $n\to\infty$. In this paper we construct a natural ``train track directed'' random walk $\mathcal W$ on $\Out(F_r)$ (where $r\ge 3$). We show that, for the element $\phi_n\in \Out(F_r)$, obtained after $n$ steps of this random walk, with asymptotically positive probability the element $\phi_n$ has the following properties: $\phi_n$ is ageometric fully irreducible, which admits a train track representative with no periodic Nielsen paths and exactly one nondegenerate illegal turn, that $\phi_n$ has ``rotationless index'' $\frac{3}{2}-r$ (so that the geometric index of the attracting tree $T_{\phi_n}$ of $\phi_n$ is $2r-3$), has index list   $\{\frac{3}{2}-r\}$ and the ideal Whitehead graph being the complete graph on $2r-1$ vertices, and that the axis bundle of $\phi_n$ in the Outer space $CV_r$ consists of a single axis. 
\end{abstract}

\thanks{The first author was partially supported by the NSF grant DMS-1405146 and by the Simons Foundation Collaboration grant no. 279836. The second author was supported first by the ARCHIMEDE Labex (ANR-11-LABX- 0033) and the A*MIDEX project (ANR-11-IDEX-0001-02) funded by the ``Investissements d'Avenir'' French government program managed by the ANR. She is secondly supported by the CRC701 grant of the DFG, supporting the projects B1 and C13 in Bielefeld. Both authors acknowledge support from U.S. National Science Foundation grants DMS 1107452, 1107263, 1107367 ``GEAR Network''.}


\date{}
\maketitle

\section{Introduction}

For an integer $r\ge 2$, an element $\phi\in\Out(F_r)$ is called \emph{fully irreducible} (sometimes also referred to as \emph{irreducible with irreducible powers}) if there is no $k\ge 1$ such that $\phi^k$ preserves the conjugacy class of a proper free factor of $F_r$. A fully irreducible $\phi\in\Out(F_r)$ is called \emph{geometric} if there exists a compact connected surface $\Sigma$ with one boundary component such that $\pi_1(\Sigma)\cong F_r$ and such that $\phi$ is induced by a pseudo-Anosov homeomorphism of $\Sigma$; fully irreducibles that are not geometric are called \emph{nongeometric}.  Bestvina and Handel proved~\cite{BH92} that a fully irreducible $\phi\in\Out(F_r)$ is nongeometric if and only if $\phi$ is \emph{atoroidal}, that is, no positive power of $\phi$ preserves the conjugacy class of a nontrivial element of $F_r$. It was later shown, as a consequence of the Bestvina-Feighn Combination Theorem~\cite{bf92}, that a fully irreducible $\phi\in\Out(F_r)$ is nongeometric if and only if the mapping torus group $F_r\rtimes_\phi \Z$ is word-hyperbolic. For this reason nongeometric fully irreducibles are also called \emph{hyperbolic}. See Section~\ref{ss:FullyIrreducibles} below for more details.

Fully irreducible elements of $\Out(F_r)$ provide a free group analog of pseudo-Anosov elements of the mapping class group $Mod(\Sigma)$ of a closed hyperbolic surface $\Sigma$. Fully irreducibles play a key role in the study of algebraic, geometric, and dynamical properties of $\Out(F_r)$. In particular, every fully irreducible $\phi\in\Out(F_r)$ admits a \emph{train track representative} (see Section~\ref{subsect:tr} below for precise definitions), and this fact was, in a sense, the starting point in the development of train track and relative train track theory for free group automorphisms.  In the structure theory of subgroups of $Out(F_r)$, subgroups containing fully irreducible elements provide basic building blocks of the theory.
For example, the Tits Alternative for $\Out(F_r)$, established in full generality in~\cite{bfh00,bfh05}, was first proved in \cite{bfh97} for subgroups of $\Out(F_r)$ containing a fully irreducible element. A result of Handel and Mosher~\cite{hm09}, with a recent different proof by Horbez~\cite{Horbez1}, shows that if $H\le \Out(F_r)$ is a finitely generated subgroup, then either $H$ contains a fully irreducible element or $H$ contains a subgroup $H_1$ of finite index in $H$ such that $H_1$ preserves the conjugacy class of a proper free factor of $F_r$. Also, fully irreducible elements are known to have particularly nice properties for the natural actions of $\Out(F_r)$ on various spaces. In particular, a fully irreducible element $\phi\in\Out(F_r)$ acts with ``North-South'' dynamics on the compactified Outer space $\overline{CV}_r$ (see \cite{ll03}) and with generalized ``North-South'' dynamics on the projectivized space of geodesic currents $\mathbb P Curr(F_r)$, \cite{m95,uyanik1,uyanik2}. For $r\ge 2$, the ``free factor complex'' $\mathcal{FF}_r$, endowed with a natural $\Out(F_r)$ action by isometries, is a free group analog of the curve complex of a finite type surface. It is known that $\mathcal{FF}_r$ is Gromov-hyperbolic, and that $\phi\in\Out(F_r)$ acts as a loxodromic isometry of $\mathcal{FF}_r$ if and only if $\phi$ is fully irreducible~\cite{bf11}.

There are several known results showing that ``random'' or ``generic'' elements of $\Out(F_r)$ are fully irreducible. The first of these results is due to Rivin~\cite{Rivin08}. He showed that if $Q=Q^{-1}$ is a finite generating set of $\Out(F_r)$ (where $r\ge 3$), then for the simple random walk $q_1,q_2,\dots $ on  $\Out(F_r)$ with respect to $Q$ (where $q_i\in Q$), the probability that $\phi_n=q_1\dots q_n\in \Out(F_n)$ is fully irreducible goes to $1$ as $n\to\infty$. Rivin later improved this result to show~\cite{Rivin10} that, with probability tending to $1$ as $n\to\infty$, the element $\phi_n$ is in fact a nongeometric fully irreducible. Rivin's approach was homological: he studied the properties of the matrices in $GL(r,\Z)$ coming from the action of $\phi_n$ on the abelianization $\Z^r$ of $F_r$. From the algebraic properties of the characteristic polynomials of these matrices, Rivin was able to derive conclusions about $\phi_n$ being a nongeometric fully irreducible with probability  tending to $1$ as $n\to\infty$. Rivin applied the same method to show~\cite{Rivin08} that ``random'' (in the same sense) elements of mapping class groups are pseudo-Anosov.

A different, geometric, approach was then explored by Maher~\cite{m11} in the context of mapping class groups (using the action of the mapping class group on the Teichmuller space), and later by Calegari and Maher~\cite{cm10} in the context of group actions of Gromov-hyperbolic spaces. 
Calegari and Maher considered the following general situation. Let $G$ be a finitely generated group acting isometrically on a Gromov-hyperbolic space $X$ and let $\mu$ be a probability measure on $G$ with finite support such that this support generates a non-elementary subgroup of $Isom(X)$. Then Calegari and Maher proved that, for the random walk on $G$ determined by $\mu$, the probability that, for a random trajectory $q_1,q_2,\dots $ of this walk, the element $g_n=q_1\dots q_n\in G$ acts as a loxodromic isometry of $X$ tends to $1$ exponentially fast as $n\to \infty$. They established this fact by showing that there exists an $L>0$ such that, in the above situation, with probability tending to $1$ exponentially fast as $n\to \infty$, the translation length of $g_n$ on $X$ is $\ge Ln$. This result applies to many natural situations, such as the action of the mapping class group (or of its ``large'' subgroup) on the curve complex, and the action of $\Out(F_r)$ (or of suitably ``large'' subgroups of $\Out(F_r)$) on the free factor complex $\mathcal{FF}_r$. Since an element of $\Out(F_r)$ acts loxodromically on $\mathcal{FF}_r$ if and only if this element is fully irreducible, the result of Calegari and Maher implies the result of Rivin if we take $Q=Q^{-1}$ to be a finite generating set of $\Out(F_r)$ and take $\mu$ to be the uniform probability measure on $Q$. Recently Mann constructed~\cite{m14} a new Gromov-hyperbolic space $\mathcal P_r$ (quasi-isometric to the main connected component of the ``intersection graph'' $I_r$ defined in~\cite{kl09}), obtained as a quotient of $\mathcal{FF}_r$ and endowed with a natural isometric action of $\Out(F_r)$ by isometries. Mann showed~\cite{m14} that $\phi\in\Out(F_r)$ acts as a loxodromic isometry of $\mathcal P_r$ if and only if $\phi$ is a nongeometric fully irreducible. The result of Calegari-Maher applies to the action of $\Out(F_r)$ on $\mathcal P_r$ and thus implies that, for a finitely supported measure $\mu$ on $\Out(F_r)$ generating a subgroup containing at least two independent nongeometric fully irreducibles, an  element $\phi_n\in\Out(F_r)$, obtained by a random walk of length $n$ defined by $\mu$, is nongeometric fully irreducible with probability tending to $1$ exponentially fast, as $n\to\infty$. Finally, Sisto~\cite{Sisto}, using a different geometric approach, introduced the notion of a ``weakly contracting element'' in a group $G$, and showed that weakly contracting elements of $\Out(F_r)$ are exactly the fully irreducibles. He showed that for any simple random walk on $\Out(F_r)$, the element $\phi_n\in\Out(F_r)$ obtained after $n$ steps is weakly contracting (and hence fully irreducible) with probability tending to $1$ exponentially fast as $n\to\infty$. 

None of the above results yield more precise structural information about ``random'' elements of $\Out(F_r)$, other than the fact that these elements are (nongeometric) fully irreducibles.

There is a considerably more detailed stratification of the set of nongeometric fully irreducibles in terms of their index, their index list, and their ideal Whitehead graph, which we discuss below. The goal of this paper is to derive such detailed structural information for ``random'' elements of $\Out(F_r)$ obtained by a certain natural random walk on $\Out(F_r)$.

The index theory for elements of $\Out(F_r)$ is motivated by surface theory. If $\phi\in Mod(\Sigma)$ is a pseudo-Anosov element (where $\Sigma$ is a closed oriented hyperbolic surface), let $\mathcal F$ be the stable measured foliation for $\phi$. Then $\mathcal F$ has singularities $p_1,\dots, p_m$, where $p_i$ is a $k_i$-prong singularity with $k_i\ge 3$. In this case it is known that the ``index sum'' $\sum_{i=1}^m (1-\frac{k_i}{2})$ equals exactly $\chi(\Sigma)$. Thus the index sum is a constant independent of $\phi$, but the ``index list'' $\{1-\frac{k_1}{2},\dots, 1-\frac{k_m}{2}\}$ is a nontrivial invariant of the conjugacy class of $\phi$ in $Mod(S)$. 

The original notion of an index for an element $\phi$ of $\Out(F_r)$, introduced in \cite{GJLL}, was formulated in terms of the dynamics of the action on the hyperbolic boundary of $F_r$. This notion of index, in general,  is not invariant under replacing $\phi$ by its positive power. Subsequently, more invariant notions of index were developed using $\R$-tree technology. We discuss the various notions of index for free group automorphisms in Section~\ref{sect:index} below.

If $\phi\in \Out(F_r)$ (where $r\ge 2$) is fully irreducible, there is a naturally associated ``attracting $\R$-tree,'' endowed with a natural isometric action of $F_r$ (this tree is similar in spirit to the ``dual tree'' obtained by lifting the stable measured foliation of a pseudo-Anosov element of $Mod(\Sigma)$  to the universal cover $\tilde\Sigma$ and then collapsing the leaves). See Section~\ref{sect:index} for the explanation of the construction of $T_\phi$ from a train track representative of $\phi$.
If $\phi$  is a nongeometric fully irreducible, the action of $F_r$ on $T_\phi$ is free but highly non-discrete (in fact, every $F_r$-orbit is dense in $T_\phi$). However, it is known that every branch point in $T_\phi$ has finite degree, and that there are only finitely many $F_r$-orbits of branch points in $T_\phi$. Thus one can still informally view the quotient $T_\phi/F_r$ as a ``graph'' and, using a formula for what the Euler characteristic of this graph should be, define the notion of a ``geometric index'' $\gind(T_\phi)=\sum_{[P]} (\deg(P)-2)$ of $T_\phi$, where the summation is taken over $F_r$-orbits $[P]$ of branch-points in $T_\phi$;  see Definition~\ref{d:gind} below. If $\phi\in\Out(F_r)$ is a geometric fully irreducible, the action of $F_r$ on $T_\phi$ is not free, but there is a natural definition of $\gind(T_\phi)$ in this case too.   Unlike in the surface case,  $\gind(T_\phi)$ is not a constant in terms of $r$ and does depend on the choice of a fully irreducible $\phi$.  For a fully irreducible $\phi\in\Out(F_r)$, the attracting tree $T_\phi$ depends only on the conjugacy class of $\phi$ in $\Out(F_r)$, and in fact $T_{\phi^k}=T_\phi$ for all $k\ge 1$. Hence $\gind(T_\phi)$ is an invariant of the conjugacy class of $\phi$ in $\Out(F_r)$, which is also preserved by taking positive powers of $\phi$.
As a consequence of more general results, it is known that, for a fully irreducible $\phi\in\Out(F_r)$, one has $1\le \gind(T_\phi)\le 2r-2$ and that, for a geometric fully irreducible $\phi\in\Out(F_r)$, one has $\gind(T_\phi)=2r-2$. Surprisingly, it turns out that for $r\ge 3$ there exist nongeometric fully irreducibles $\phi\in\Out(F_r)$ with $\gind(T_\phi)=2r-2$~\cite{bf94,bf95,GJLL,g05,hm07,JL08}; such $\phi$ are called \emph{parageometric}. A nongeometric fully irreducible $\phi\in\Out(F_r)$ with $\gind(T_\phi)<2r-2$ is said to be \emph{ageometric}. 

As we have seen, for a nongeometric fully irreducible $\phi\in\Out(F_r)$, the geometric index $\gind(T_\phi)$ arises from an ``index sum'' over representatives of $F_r$-orbits of branch points in $T_\phi$. The terms of this sum provide an ``index list,'' which is also an invariant of the conjugacy class of $\phi$,  preserved by taking positive powers. In \cite{hm11}, Handel and Mosher formalized this fact by introducing the notion of an \emph{index list} and of \emph{rotationless index} $i(\phi)$ (the latter is called ``index sum'' in \cite{hm11}) for a nongeometric fully irreducible $\phi$. The most invariant definition of these notions involves  looking at the structure of branch-points of $T_\phi$, which also shows that $i(\phi)=-\frac{1}{2}\gind(T_\phi)$ for every non geometric fully irreducible $\phi$. Handel and Mosher also gave an equivalent description of the index list and rotationless index in terms of a train track representative of $\phi$. We give this description in Definition~\ref{d:RotationlessIndex} below. 

For a nongeometric fully irreducible $\phi$, Handel and Mosher also introduced another combinatorial object, called the \emph{ideal Whitehead graph} $\mathcal{IW}(\phi)$ of $\phi$, which encodes further, more detailed, information than the index list in a single finite graph. They also provided an equivalent description of $\mathcal{IW}(\phi)$ in terms of a train track representative of $\phi$; see Definition~\ref{d:IWG} below. For a pseudo-Anosov, the component of the ideal Whitehead graph coming from a foliation singularity is a polygon with edges corresponding to the lamination leaf lifts bounding a principal region in the universal cover \cite{n86}.
Since the number of vertices of each polygonal ideal Whitehead graph component is determined by the number of prongs of the singularity, the index list and the ideal Whitehead graph record the same data. In the $Out(F_r)$ setting, not only is the ideal Whitehead graph $\mathcal{IW}(\phi)$ a finer invariant (c.f. \cite{pII,pIII}), but it provides further information about the behavior of lamination leaves at a singularity. It is again an invariant of the conjugacy class of $\phi$, also invariant under taking positive powers of $\phi$. Moreover, while $\mathcal{IW}(\phi)$ is a more detailed structural invariant than $i(\phi)$ or the index list of $\phi$, both of these invariants can be ``read-off'' from $\mathcal{IW}(\phi)$.

We will now describe the main result of the present paper. Let $r\ge 3$ and let the free group $F_r=F(a_1,\dots, a_r)$ be equipped with a fixed free basis $A=\{a_1,\dots, a_r\}$. 
We denote by $R_r$ the \emph{$r$-rose}, which is a wedge of $r$ directed loop-edges, wedged at a single vertex $v$ and labelled $a_1,\dots, a_r$. Thus we have a natural identification $F_r=F(a_1,\dots, a_r)=\pi_1(R_r,v)$.

An \emph{elementary Nielsen automorphism} of $F_r$ is an element $\theta\in\Aut(F_r)$ such that there exist $x,y\in A^{\pm 1}$, $y\ne x^{\pm1}$, with the property that $\theta(x)=yx$, $\theta(x^{-1})=x^{-1}y^{-1}$, and $\theta(z)=z$ for each $z\in A^{\pm 1}-\{x,x^{-1}\}$. We denote such $\theta$ by $\theta=[x\mapsto yx]$. We say that an ordered pair $(\theta=[x\mapsto yx], \theta'=[x'\mapsto y'x'])$ is \emph{admissible} if either $x' = x$ and $y' \neq y^{-1}$ or $y' = x$ and $x' \neq y^{-1}$. A sequence $\theta_1,\dots, \theta_n$ (where $n\ge 1$) of standard Nielsen automorphisms of $F_r$ is called \emph{admissible} if, for each $1\le i<n$, the pair $(\theta_i,\theta_{i+1})$ is admissible. A sequence $\theta_1,\dots, \theta_n$ of standard Nielsen automorphisms of $F_r$ is called \emph{cyclically admissible} if it is admissible and if the pair $(\theta_n,\theta_1)$ is also admissible. We denote by $S$ the set of all elementary Nielsen automorphisms of $F_r$ (so that $S$ is a finite set with exactly $4r(r-1)$ elements, see Section~\ref{sect:ttdrw}); we also verify in Lemma~\ref{lem:4r-6} that for every $\theta\in S$ there are exactly $4r-6$ elements $\theta'\in S$ such that the pair $(\theta,\theta')$ is admissible. It is well-known that $S$ generates a subgroup of finite index in $\Out(F_r)$.

We define a finite-state Markov chain with the state set $S$ as follows.  For $\theta,\theta'$ we set the transition probability $P(\theta'|\theta)$ from $\theta$ to $\theta'$ to be $1/(4r-6)$ if the pair $(\theta,\theta')$ is admissible and $0$ otherwise. We show in Lemma~\ref{lem:irr-aper} that this is an irreducible aperiodic finite state Markov chain and that the uniform distribution $\mu_r$ on $S$ is stationary for this chain.
We then consider a random process $\mathcal W$ defined by this chain starting with the uniform distribution $\mu_r$ on $S$. Thus $\mathcal W$ can be viewed as a random walk, where we first choose an element $\theta_1\in S$ uniformly at random and then, if at step $n\ge 1$ we have chosen $\theta_n\in S$, we choose $\theta_{n+1}\in S$ according to the distribution $P(-|\theta_n)$ defined above. The sample space of $\mathcal W$ is the set $S^\N$ of all sequences $\theta_1,\theta_2,\dots $ of elements of $S$ and the random walk $\mathcal W$ defines a probability measure $\mu_\mathcal W$ on $S^\N$ whose support consists of all infinite admissible sequences of $S$. 
To each trajectory $\omega=\theta_1,\theta_2,\dots $ of $\mathcal W$ we associate a sequence $\phi_n\in\Out(F_r)$, where $\phi_n=\theta_n\circ \dots \circ \theta_1$.

The random walk $\mathcal W$ can be viewed  as an $\Out(F_r)$ version of the simple non-backtracking random walk on the free group itself. The reason is the following crucial property of admissible sequences: if $\theta_1,\dots,\theta_n$ is an admissible sequence of elements of $S$, then, for every letter $a\in A^{\pm 1}$, computing the image $(\theta_n\circ\dots \circ \theta_1)(a)$ by performing letter-wise substitutions produces a freely reduced word in $A^{\pm 1}$. This fact, established in Lemma~\ref{L:reg} below, implies that for any cyclically admissible sequence $\theta_1,\dots,\theta_n$, the element $\phi_n=\theta_n\circ\dots \circ \theta_1\in \Out(F_r)$ admits a train track representative $g_n:R_r\to R_r$ on the rose $R_r$, and, moreover, this train track map has exactly one nondegenerate illegal turn; see Theorem~\ref{T:tt}. That is why we also think of $\mathcal W$ as a ``train track directed'' random walk on $\Out(F_r)$.  

In addition, we show in Theorem~\ref{t:decomp} that for each train track map $g \colon R_r\to R_r$ with exactly one nondegenerate illegal turn with $g_\#=\phi\in \Out(F_r)$, for some positive power $g^p$ of $g$ there exists a cyclically admissible sequence $\theta_1,\dots,\theta_n$ such that $\phi^p=\theta_n\circ\cdots \circ \theta_1$, and so that our walk $\mathcal W$ reaches $\phi^p$ (and, moreover, $p$ only depends on $r$).

\begin{df}[Property $(\mathcal G)$]
Let $r\ge 3$ be an integer. We say that $\phi\in\Out(F_r)$ has \emph{property $(\mathcal G)$} if all of the following hold:
\begin{enumerate}
\item The outer automorphism $\phi$ is ageometric fully irreducible;
\item We have $i(\phi)=\frac{3}{2}-r$ (so that $\gind(T_\phi)=2r-3$), and $\phi$ has single-element index list $\{\frac{3}{2}-r\}$.
\item There exists a train track representative $f:R_r\to R_r$ of $\phi$ such that $f$ has no pINPs and such that $f$ has exactly one nondegenerate illegal turn.
\item  The ideal Whitehead graph $\mathcal{IW}(\phi)$ of $\phi$  is the complete graph on $2r-1$ vertices.
\item The axis bundle for $\phi$ in $CV_r$ consists of a single axis. 
\end{enumerate}

\end{df}

(The terms appearing in this definition that have not yet been defined are explained later in the paper).

Our main result (c.f. Theorem~\ref{T:main} below) is:

\begin{theor}\label{T:A}
Let $r\ge 3$.  For $n\ge 1$ let $E_n$ be the event that for a trajectory $\omega=\theta_1,\theta_2,\dots $ of $\mathcal W$ the sequence $\theta_1,\dots,\theta_n$  is cyclically admissible.
Also, for $n\ge 1$ let  $B_n$ be the event that for a trajectory $\omega=\theta_1,\theta_2,\dots $ of $\mathcal W$ the outer automorphism  $\phi_n=\theta_n \circ \cdots \circ \theta_1\in\Out(F_r)$ has property $(\mathcal G)$.

Then the following hold:

\begin{enumerate} 
\item For the conditional probability $Pr(B_n|E_n)$ we have \[\lim_{n\to\infty} Pr(B_n|E_n)=1.\]
\item We have $Pr(E_n)\to_{n\to\infty} \frac{2r-3}{2r(r-1)}$ and $\liminf_{n\to\infty} Pr(B_n)\ge \frac{2r-3}{2r(r-1)}>0$.
\item For $\mu_\mathcal W$-a.e. trajectory $\omega=\theta_1,\theta_2,\dots $ of $\mathcal W$, there exists an $n_\omega\ge 1$ such that for every $n\ge n_\omega$ such that $\mathfrak t_n=\theta_1,\dots, \theta_{n}$ is cyclically admissible, we have that the outer automorphism $\phi_n=\theta_n \circ \cdots \circ \theta_1\in\Out(F_r)$ has property $(\mathcal G)$.
\end{enumerate}
\end{theor}

We then project the random walk $\mathcal W$ to a random walk on $SL(r,\Z)$ by sending each $\theta\in S$ to its transition matrix in $SL(r,\Z)$, when $\theta$ is viewed as a graph map $R_r\to R_r$. We analyze the spectral properties of this projected walk and show that it has positive first Lyapunov exponent, see Proposition~\ref{P:pos}. We then conclude that for $\mu_\mathcal W$-a.e. trajectory $\theta_1,\theta_2,\dots, $, the stretch factor $\lambda(\theta_n\circ \dots \circ\theta_1)$ grows exponentially in $n$ for any increasing sequence of indices $n$ such that  $\theta_1,\theta_2,\dots, \theta_n$ is cyclically admissible. See Theorem~\ref{T:growth} below for the precise statement, and see Section~\ref{sect:stretch} for the definition and properties of stretch factor for an element of $\Out(F_r)$.

As a consequence, we show that our random walk $\mathcal W$ has positive linear rate of escape with respect to the word metric defined by any finite generating set of $\Out(F_r)$ (c.f. Theorem~\ref{T:LRE}):

\begin{theor}\label{T:B}
Let $r\ge 3$ and let $Q$ be a finite generating set of $\Out(F_r)$ such that $Q=Q^{-1}$. Then there exists a constant $c>0$ such that, for $\mu_\mathcal W$-a.e. trajectory $\omega=\theta_1,\theta_2,\dots $ of $\mathcal W$,
\[
\lim_{n \to \infty} \frac{1}{n} |\theta_n \dots \theta_1|_Q =c.
\]
\end{theor}

Here for $\phi\in \Out(F_r)$, $|\phi|_Q$ denotes the distance from $1$ to $\phi$ in $\Out(F_r)$ with respect to the word metric on $\Out(F_r)$ corresponding to $Q$.

Note that our random walk $\mathcal W$ is a ``left'' random walk on $\Out(F_r)$, since with a random trajectory $\theta_1,\theta_2,\dots $ of $\mathcal W$ we associate the sequence $\phi_n=\theta_n \circ \cdots \circ  \theta_1\in\Out(F_n)$ (rather than $\theta_1 \circ \cdots \circ \theta_n$). We explain in Remark~\ref{rk:left-right} how one can convert our random walk into a more traditional ``right'' random walk on $\Out(F_r)$, although after such a conversion the statements of our main results become less natural.

The proof of Theorem~\ref{T:A} is based on completely different methods from all the previous results about the properties of ``random'' elements of $\Out(F_r)$ (see above the discussion of the work of Rivin, Calegari-Maher, and Sisto). Instead of using the action of $\Out(F_r)$ on the free factor complex or on the abelianization of $F_r$, we analyze the properties of train track representatives of elements $\phi_n\in\Out(F_r)$ obtained by our walk $\mathcal W$. The main payoff is that, apart from concluding that $\phi_n$ is fully irreducible, we obtain a great deal of extra detailed structural information about the properties of $\phi_n$, where such information does not seem to be obtainable by prior methods.
A key tool in establishing that $\phi_n$ is fully irreducible is the train track criterion of full irreducibility obtained in \cite{pII} (see Proposition~\ref{P:FIC} below); we also discuss a related criterion obtained in \cite{K14} (see Proposition~\ref{P:K14} below). We substantially rely on ideas and results of \cite{pI,pII,pIII}, although the exposition given in the present paper is almost completely self-contained.

Finally, we pose several open problems naturally arising from our work:

\begin{qst}
In the context of our Theorem~\ref{T:A}, for a $\mu_\mathcal W$-a.e. trajectory $\theta_1,\theta_2,\dots$ and $n>>1$ such that $\theta_1,\dots,\theta_n$ is cyclically admissible, what can be said about the rotationless index, index list, and Ideal Whitehead graph of $(\theta_n\circ \dots \circ\theta_1)^{-1}$?  
\end{qst}
In \cite{JL08}, J\"aeger and Lustig, for each $r\ge 3$, constructed a positive automorphism $\phi$ such that $\phi$ is ageometric fully irreducible with $i(\phi)=\frac{3}{2}-r$ and such that $i(\phi^{-1})=1-r$, so that $\phi^{-1}$ is parageometric. In their construction $\phi$ arises as a rather special composition of positive elementary Nielsen automorphisms, where this composition is cyclically admissible in our sense. However, experimental evidence appears to indicate that for $\phi\in \Out(F_r)$ produced by our walk $\mathcal W$ for long ``random'' cyclically admissible compositions, the absolute value of $i(\phi^{-1})$ is much smaller than the maximum value of $r-1$ achieved by parageometrics. 

\begin{qst}\label{q:2}
Again in the context of Theorem~\ref{T:A}, is it true that for $\mu_\mathcal W$-a.e. trajectory $\theta_1,\theta_2,\dots$ of $\mathcal W$, projecting this trajectory to the free factor complex $\mathcal{FF}_r$ as $\theta_1\dots \theta_n p$, where $p$ is a vertex of $\mathcal{FF}_r$ (or perhaps as $\theta_1^{-1}\dots \theta_n^{-1} p$), gives a sequence that converges to a point of the hyperbolic boundary $\partial \mathcal{FF}_r$?
\end{qst}

Note that by the recent work of Horbez~\cite{Horbez2} on describing the Poisson boundary of $\Out(F_r)$, the answer to the similar question for a simple random walk on $\Out(F_r)$ is positive. In several personal conversations, Camille Horbez indicated to the second author a plausible approach for getting a positive answer to Question~\ref{q:2}. 

\begin{qst}
Let $r\ge 3$ and let $Q=Q^{-1}$ be a finite generating set of $\Out(F_r)$. If $q_1,q_2,\dots $ is a random trajectory of the simple random walk on $\Out(F_r)$, what can be said about the properties of $\phi_n=q_1\dots q_n\in \Out(F_r)$, apart from the fact that, with probability tending to $1$ as $n\to\infty$, the automorphism $\phi_n$ is a nongeometric fully irreducible? In particular, is $\phi_n$ ageometric? What can be said about $i(\phi_n)=-\frac{1}{2}\gind(T_{\phi_n})$, and about the index list and the Ideal Whitehead graph of $\phi_n$?
\end{qst}

\begin{qst}
Let $\Sigma$ be a closed oriented hyperbolic surface. What can be said about the index/singularity list for the stable foliation of a ``random'' element $\phi_n\in Mod(S)$ obtained by a simple random walk of length $n$ on $Mod(S)$? (Note that by the results of Rivin, Maher, and Calegari-Maher, discussed above, we do know that $\phi_n$ is pseudo-Anosov with probability tending to $1$ as $n\to\infty$).
\end{qst}

It would also be interesting to understand the index properties of generic automorphisms $\phi_n\in \Out(F_r)$ (where $r\ge 3$) produced by a simple random walk on $Out(F_r)$ with respect to some finite generating set of $Out(F_r)$. As noted above, it is already known that in this situation $\phi_n$ is atoroidal and fully irreducible with probability tending to $1$ as $n\to\infty$. Computer experiments, conducted by us using Thierry Coulbois' computer package for free group automorphisms\footnote{The package is available at {\url{http://www.cmi.univ-mrs.fr/~coulbois/train-track/ }}} appear to indicate that generically both $\phi_n$ and $\phi_n^{-1}$  are ageometric fully irreducible, with a very small value of $|i(\phi_n)|$ (in contrast with an almost maximal value $|i(\phi_n)|=r-\frac{2}{3}$ in Theorem~\ref{T:A}).  These experiments also appear to indicate that several possible index lists for $\phi_n$ occur with asymptotically positive probability each, with the single-entry list $\{-\frac{1}{2}\}$ occurring with the highest probability. However, the maximal values of the length $n$ of a simple random walk on $Out(F_r)$ (with $r=3,4,5,6$), that our experiments were able to handle, were around $n\approx 80-85$, and longer experiments are needed to get more conclusive empirical data.

A plausible conjecture here would be that all singularities of the stable foliation of a random $\phi_n$ are 3-prong singularities. Note that a result of Eskin, Mirzakhani, and Rafi~\cite{emr12} shows that for ``most'' (in a different sense) closed geodesics in the moduli space of $\Sigma$, the pseudo-Anosov element of $Mod(\Sigma)$ corresponding to such a closed geodesic has all singularities of its stable foliation being  3-prong.

The first author thanks Terence Tao for supplying a proof of Proposition~\ref{P:Tao} and to Jayadev Athreya, Vadim Kaimanovich, Camille Horbez and Igor Rivin for helpful discussions about Lyapunov exponents and random walks.  
We are also grateful to Lee Mosher for useful conversations regarding the proof of Theorem~\ref{t:decomp}. 

\section{Preliminaries}\label{s:prelims}

\subsection{Graphs, paths and graph maps}\label{ss:graphmaps}

\begin{df}[Graphs]\label{d:graphs}
A \emph{graph} $\Gamma$ is a 1-dimensional cell-complex. We call the 0-cells of $\Gamma$ \emph{vertices} and denote the set of all vertices of $\Gamma$ by $V\Gamma$. We refer to open 1-cells of $\Gamma$ as \emph{topological edges} of $\Gamma$ and denote the set of all topological edges of $\Gamma$ by $E_{top}\Gamma$.

Each topological edge $\Gamma$ is homeomorphic to the open interval $(0,1)$ and thus, when viewed as a 1-manifold, admits two possible orientations.  An \emph{oriented edge} of $\Gamma$ is a topological edge with a choice of an orientation on it. We denote by $E\Gamma$ the set of all oriented edges of $\Gamma$. If $e\in E\Gamma$ is an oriented edge, we denote by $\overline e$ the same underlying edge with the opposite orientation. Note that for each $e\in E\Gamma$ we have $\overline{\overline e} =e$ and $\overline e\ne e$; thus $-: E\Gamma\to E\Gamma$ is a fixed-point-free involution. 

Since $\Gamma$ is a cell-complex, every oriented edge $e$ of $\Gamma$ comes equipped with the orientation-preserving attaching map $j_e:[0,1]\to \Gamma$ such that $j_e$ maps $(0,1)$ homeomorphically to $e$ and such that $j_e(0),j_e(1)\in V\Gamma$.  By convention we choose the attaching maps so that, for each $e\in E\Gamma$ and each $s\in (0,1)$, we have $(j_{\overline e}^{-1}\circ j_e)(s)=1-s$.
For $e\in E\Gamma$ we call $j_e(0)$ the \emph{initial vertex} of $e$, denoted $o(e)$, and we call $j_e(1)$ the \emph{terminal vertex} of $e$, denoted $t(e)$. Thus, by definition, $o(\overline e)=t(e)$ and $t(\overline e)=o(e)$.

If $\Gamma$ is a graph and $v\in V\Gamma$, a \emph{direction at $v$} in $\Gamma$ is an edge $e\in E\Gamma$ such that $o(e)=v$. We denote the set of all directions at $v$ in $\Gamma$ by $Lk\Gamma(v)$ and call it the \emph{link} of $v$ in $\Gamma$. Then the \emph{degree} of $v$ in $\Gamma$, denoted $\deg(v)$ or $\deg_\Gamma(v)$, is the cardinality of the set $Lk\Gamma(v)$.

An \emph{orientation} on a graph $\Gamma$ is a partition $E\Gamma=E_+\Gamma\sqcup E_-\Gamma$ such that for an edge $e\in E\Gamma$ we have $e\in E_+\Gamma$ if and only if $\overline e\in E_-\Gamma$.
\end{df}

Note that both topological edges and oriented edges are, by definition, open subsets of $\Gamma$ and they don't contain their endpoints.

\begin{df}[Combinatorial and topological paths]\label{d:ctpaths}
 A \emph{combinatorial edge-path} $\gamma$ of \emph{length} $n\ge 1$ is a sequence $\gamma=e_1,\dots, e_n$ such that $e_i\in E\Gamma$ for $i=1,\dots, n$ and such that $t(e_i)=o(e_{i+1})$ for all $1\le i<n$. We put $o(\gamma):=o(e_1)$, $t(\gamma):=t(e_n)$, and $\gamma^{-1}:=\overline{e_n},\dots, \overline{e_1}$. Thus $\gamma^{-1}$ is again a combinatorial edge-path of length $n$. For $v\in V\Gamma$ we also view $\gamma=v$ as a combinatorial edge-path of length $0$ with $o(\gamma)=t(\gamma)=v$ and $\gamma^{-1}=\gamma$.
For a combinatorial edge-path $\gamma$ of length $n\ge 0$ we denote $|\gamma|=n$.

A combinatorial edge-path $\gamma$ is \emph{reduced} or \emph{tight} if $\gamma$ does not contain subpaths of the form $e, \overline e$, where $e\in E\Gamma$.

A \emph{topological edge-path} in $\Gamma$ is a continuous map $f \colon [a,b]\to \Gamma$ such that either $a=b$ and $f(a)=f(b)\in V\Gamma$ or $a<b$ and there exists a subdivision $a=a_0<a_1<\dots < a_n=b$ and a combinatorial edge-path $\gamma=e_1,\dots, e_n$ in $\Gamma$ such that:

(1) We have $f(a_i)\in V\Gamma$ for $i=0,\dots, n$.

(2) We have $f(a_i)=o(e_i)$ for $i=0,\dots, n-1$ and $f(a_i)=t(e_{i-1})$ for $i=1,\dots, n$.

(3) $f|_{(a_{i-1},a_{i})}$ is an orientation-preserving homeomorphism mapping $(a_{i-1},a_{i})$ onto $e_i$. 

\noindent Sometimes we drop the commas and just write $\gamma=e_1\dots e_n$.

Note that, for a topological edge-path  $f:[a,b]\to \Gamma$, where $a<b$, the combinatorial edge-path $\gamma=e_1,\dots, e_n$ with the above properties is unique; we say that $\gamma$ is the combinatorial edge-path \emph{associated} to $f$; we also call $a=a_0<\dots <a_n=b$ the \emph{associated subdivision} for $f$.
If $a=b$ and $f(a)=f(b)=v\in V\Gamma$, we say the path $\gamma=v$ is associated to $f$. 

Let $f:[a,b]\to \Gamma$ be a topological edge-path (where $a<b$), let $\gamma=e_1,\dots, e_n$ be the associated combinatorial edge-path, and let $a=a_0<a_1<\dots < a_n=b$ be the corresponding subdivision. We say that a topological edge-path $f$ is \emph{tame} if for every $i=1,\dots, n$ the map $j_{e_i}^{-1}\circ f: (a_{i-1},a_i)\to (0,1)$ is a (necessarily unique) orientation-preserving affine homeomorphism.
By convention, if $f:[a,b]\to \Gamma$ is a topological edge-path with $a=b$, we also consider $f$ to be tame.

A \emph{topological path} $f:[a,b]\to \Gamma$ (where $a<b$) is defined similarly to as in the definition of a topological edge-path above, except that we no longer require $f(a)=o(e_1)$ and $f(b)=t(e_n)$, but instead allow $f(a)\in e_1\cup o(e_1)$ and $f(b)\in e_n\cup t(e_n)$.  For $i=1$ and $i=n$, condition (3) in the above definition is relaxed accordingly.
For $a=b$ we view any map $f:\{a\}\to \Gamma$ as a topological path in $\Gamma$.

For a topological path $f:[a,b]\to \Gamma$ with $a<b$ there is still a canonically \emph{associated} combinatorial edge-path $\gamma=e_1,\dots, e_n$ and a canonically \emph{associated} subdivision $a=a_0<a_1<\dots < a_n=b$.

We define what it means for a topological path $f:[a,b]\to \Gamma$ to be \emph{tame}, similarly to the notion of a tame topological edge-path above, by requiring all the maps $j_{e_i}^{-1}\circ f|_{(a_{i-1},a_i)}$ to be injective \textbf{affine} orientation-preserving maps from $(a_{i-1},a_i)$ to subintervals of $(0,1)$. For $1<i<n$ it is still the case that $j_{e_i}^{-1}\circ f\left((a_{i-1},a_i)\right)=(0,1)$. However, we now allow for the possibility that $j_{e_1}^{-1}\circ f\left((a_{0},a_1)\right)=(s,1)$ with $s>0$ (in the case where $f(a)\in e_1$ rather than $f(a)=o(e_1)$) and that  $j_{e_n}^{-1}\circ f\left((a_{n-1},a_n)\right)=(0,s)$ with $s<1$ (in the case where $f(a)\in e_n$ rather than $f(a)=t(e_n)$). Also, if $a=b$, we consider any map   $f:\{a\}\to \Gamma$ to be a tame path in $\Gamma$.
\end{df}

Note that if $f:[a,b]\to \Gamma$ is a topological path (respectively, tame topological path), then for any $a\le a'\le b'\le b$ the restriction $f|_{[a',b']}:[a',b']\to\Gamma$ is again a topological path (respectively, tame topological path) in $\Gamma$.

Also notice that, if $n\ge 1$ and $\gamma=e_1,\dots, e_n$ is a combinatorial edge-path and $a\in \mathbb R$, then there exists a unique tame topological edge-path $f:[a,a+n]\to \Gamma$ with associated  combinatorial path $\gamma$ and associated subdivision $a_i=a+i$, $i=0,\dots, n$. By contrast, given $\gamma=e_1,\dots, e_n$ and $a\in \mathbb R$, there exist uncountably many topological edge-paths $f:[a,b]\to \Gamma$ with associated  combinatorial path $\gamma$ and associated subdivision $a_i=a+i$, $i=0,\dots, n$.  The distinction between topological edge-paths and tame topological edge-paths is often ignored in the literature, but this distinction is important when considering fixed points and dynamics of graph maps, as we will see later.

\begin{df}[Paths]\label{d:paths}
Let $\Gamma$ be a graph. By a \emph{path} in $\Gamma$ we will mean a tame topological path $f:[a,b]\to\Gamma$.  A  path $f:[a,b]\to \Gamma$ is \emph{trivial} if $a=b$ and \emph{nontrivial} if $a<b$.
A path $f$ is \emph{tight} or \emph{reduced} if the map $f:[a,b]\to \Gamma$ is locally injective. Thus a trivial path is always tight, and a nontrivial path is tight if and only if the combinatorial edge-path associated to $f$ is reduced.
\end{df}

\subsection{Graph maps}

\begin{df}[Graph maps]
Let $\Gamma$ and $\Delta$ be graphs. A \emph{topological graph map} $f:\Gamma\to\Delta$ is a continuous map such that $f(V\Gamma)\subseteq V\Delta$ and such that the restriction of $f$ to each edge of $\Gamma$ is a topological edge-path in $\Delta$. More precisely, this means that for each $e\in E\Gamma$ the map $f\circ j_e: [0,1]\to \Delta$ is a topological edge-path in $\Delta$.

A \emph{graph map} is a topological graph map $f:\Gamma\to\Delta$ such that the restriction of $f$ to each edge of $\Gamma$ is a path in $\Gamma$ (in the sense of Definition~\ref{d:paths}), that is, such that for every $e\in E\Gamma$ the map $f\circ j_e: [0,1]\to \Delta$ is a tame topological edge-path in $\Delta$.
\end{df}

\begin{conv}
By convention, if $f:[a,b]\to \Gamma$ is a tame topological edge-path with the associated combinatorial edge-path $\gamma=e_1,\dots, e_n$, we will usually suppress the distinction between $f$ and $\gamma$.  In particular, if $f:\Gamma\to\Delta$ is a graph map and $e\in E\Gamma$, we will usually suppress the distinction between the tame topological path $f\circ j_e: [0,1]\to \Delta$ and the associated combinatorial edge-path $\gamma=e_1,\dots, e_n$ in $\Delta$. Moreover, in this situation we will often write $f(e)=e_1,\dots, e_n$ or even $f(e)=e_1\dots e_n$.
\end{conv}

Note that our definition implies that if $f:\Gamma\to\Delta$ is a topological graph map, then for each edge $e\in E\Gamma$ we have $f(e)=e_1,\dots, e_n$ with $n\ge 1$.  It is sometimes useful to allow a topological graph map to send an edge to a vertex (rather than to an edge-path of positive combinatorial length), but we will not need this level of generality in the present paper.

A topological graph-map $f:\Gamma\to\Delta$ is said to be \emph{expanding} if for each edge $e\in E\Gamma$, we have $|f^n(e)|\to\infty$ as $n\to\infty$.

\begin{rk}
The distinction between the notions of a graph map and of a topological graph map is important when considering the fixed points and the dynamics of a (topological) graph map $f:\Gamma\to\Gamma$. Indeed, suppose that $f:\Gamma\to\Gamma$ is a topological graph map such that for some edge $e\in E\Gamma$ the combinatorial edge-path associated with $f\circ j_e:[0,1]\to\Gamma$ is $e_1,\dots, e_n$, such that $n\ge 3$, and such that for some $2\le i_0\le n-1$ we have $e_{i_0}=e$. Then there exists a subdivision $0=a_0<a_1< \dots <a_n$ such that $f\circ j_i$ maps $(a_{i-1},a_i)$ homeomorphically and preserving orientation to $e_i$, for $i=1,\dots, n$. Denote $x_i=j_e(a_i)$ and denote by $(x_{i-1},x_i)$ the open segment in $e$ between $x_{i-1}$ and $x_i$; so that  $(x_{i-1},x_i)=j_e\left((a_{i-1},a_i)\right)$. Thus, for $i=1,\dots, n$, $f$ maps the open segment $(x_{i-1},x_i)$ homeomorphically and preserving orientation to $e_i$.
Our assumption that $e_{i_0}=e$ with $1<i_0<n$ implies that the map $h:=j_{e}^{-1}\circ f \circ j_e$ maps the subsegment $[a_{i_0-1},a_{i_0}]$ of $[0,1]$ by an orientation preserving homeomorphism to the interval $[0,1]$. The intermediate value theorem then implies that there exists $a_{i_0-1}<s<a_i$ such that the point $x=j_e(s)\in (x_{i_0-1},x_{i_0})$ is fixed by $f$, that it satisfies $f(x)=x$. However, the orientation-preserving homeomorphism $h:[a_{i_0-1},a_{i_0}]\to [0,1]$ can, in principle, have uncountably many fixed points; e.g. $h$ could coincide with the identity map on some nondegenerate subsegment of $[a_{i_0-1},a_{i_0}]$. Thus, $f$ may have uncountably many fixed points in the interval $(x_{i_0-1},x_{i_0})$ of $e$. 
On the other hand, if in the above situation $f$ is a graph map (so that the path $f\circ j_e:[0,1]\to\Gamma$  is tame), then $h:=j_{e}^{-1}\circ f \circ j_e$ maps the subinterval $[a_{i_0-1},a_{i_0}]$ of $[0,1]$ to the interval $[0,1]$ by an orientation preserving \textbf{affine} homeomorphism. It then follows that there exists a \textbf{unique} $x\in  (x_{i_0-1},x_{i_0})$ such that $f(x)=x$.

Thus, if $\Gamma$ is a finite graph and $f:\Gamma\to\Gamma$ is an expanding (in the combinatorial sense defined above) topological graph-map, then $f$ may have uncountably many fixed points in $\Gamma$. By contrast, if $\Gamma$ is finite and $f:\Gamma\to\Gamma$ is an expanding graph-map, then $f$ has only finitely many fixed points and only countably many periodic points in $\Gamma$.
 
Allowing $f:\Gamma\to\Gamma$ to be a topological graph map, rather than a graph map, may result in some additional pathologies of the dynamics of $f$ under iterations; e.g. an expanding  topological graph-map $f:\Gamma\to\Gamma$ may turn out to act as a ``contraction'' on a nondegenerate subsegment of an edge of $\Gamma$. Restricting our consideration to graph maps in this paper rules out these kinds of pathologies. 
\end{rk}

For a square matrix $M$ with real coefficients we denote by $\lambda(M)$  the spectral radius of the matrix $M$, that is, the maximum of $|\lambda_i|$ where $\lambda_i\in \mathbb C$ varies over all eigenvalues of $M$.

\begin{df}[Transition matrix of a graph map]\label{d:transition-matrix}
Let $\Gamma$ be a finite graph with $m=\#(E_{top}\Gamma)\ge 1$ topological edges. Choose an orientation $E\Gamma=E_+\Gamma\sqcup E_-\Gamma$ and an ordering $E_+\Gamma=\{e_1,\dots, e_m\}$. 
Let $f:\Gamma\to\Gamma$ be a graph map. The \emph{transition matrix} $M(f)$ of $f$ is an $m\times m$ matrix with nonnegative integer entries, where, for $1\le i,j\le m$, the entry $m_{ij}$ in the position $ij$ in $M$ is equal to the number of times $e_i$ and $\overline{e_i}$ appear in the combinatorial edge-path $f(e_j)$. 

We denote $\lambda(f):=\lambda(M(f))$, the spectral radius of the matrix $M(f)$.
\end{df}

It is not hard to see for the above definition that if $f,g:\Gamma\to\Gamma$ are graph-maps, then $M(g\circ f)=M(g)M(f)$. In particular, we have $M(f^k)=[M(f)]^k$ for each integer $k\ge 1$.

\begin{lem}\label{L:>0}
Let $\Gamma$ be a finite connected graph and let $f,g:\Gamma\to\Gamma$ be such that $M(f)>0$ and $g:\Gamma\to\Gamma$ is surjective.
Then $M(g\circ f)>0$.
\end{lem}
\begin{proof}
Let $e\in E\Gamma$ be arbitrary. Since $M(f)>0$, the path $f(e)$ passes through every topological edge of $\Gamma$. Since $g\colon\Gamma\to\Gamma$ is surjective, it follows that the path $g(f(e))$ also passes through every topological edge of $\Gamma$. Hence $M(g\circ f)>0$, as required.
\end{proof}

\begin{df}[Regular map]
A graph map $f\colon\Gamma\to\Delta$ is \emph{regular} if for each $e\in E\Gamma$ the combinatorial edge-path $f(e)=e_1,\dots, e_n$ is reduced. Note that $f$ is reduced if and only if the path $f\circ j_e\colon [0,1]\to \Delta$ is locally injective.
\end{df}

Note that, if $f\colon \Gamma\to\Gamma$ is a graph map, then for each $k\ge 1$, we have that $f^k\colon \Gamma\to\Gamma$ is also a graph map. However, if $f\colon \Gamma\to\Gamma$ is a regular graph map, then, in general, the map $f^k\colon \Gamma\to\Gamma$ may fail to be regular for some $k\ge 1$.

\subsection{Perron-Frobenius theory}

We say that a $d\times d$ matrix $M$ with real coefficients is \emph{nonnegative}, denoted $M\ge 0$, if all coefficients of $M$ are $\ge 0$.
Recall that a nonnegative $d\times d$ matrix $M$ is called \emph{irreducible} if for each $1\le i,j\le d$ there exists a $k\ge 1$ such that $(M^k)_{ij}>0$.
It is not hard to check that, in the context of Definition~\ref{d:transition-matrix}, the matrix $M(f)$ is irreducible if and only if for each $e,e'\in E\Gamma$ there exists a $k\ge 1$ such that the path $f^k(e)$ contains an occurrence of either $e'$ or of $\overline{e'}$.

For a $d\times d$ matrix $M=(m_{ij})_{i,j=1}^m$ we write $M>0$ if $m_{ij}>0$ for all $1\le i,j\le d$. Note that if $M>0$, then $M$ is irreducible. 

Recall that  for a square matrix $M$ with real coefficients we denote by $\lambda(M)$  the spectral radius of the matrix $M$.

A key basic result of Perron-Frobenius theory says that if $M\ge 0$ is a $d\times d$ irreducible matrix then $\lambda(M)>0$ and, moreover, $\lambda(M)$ is an eigenvalue for $M$, called the \emph{Perron-Frobenius (PF) eigenvalue}. 
Moreover, in this case there exists an eigenvector $v\in \mathbb R^d$ with $Mv=\lambda(M)v$ such that all coefficients of $v$ are $> 0$. See \cite{Seneta} for background on Perron-Frobenius theory.

\subsection{Train track maps}

\begin{df}[Train track map]
Let $\Gamma$ be a finite connected graph without degree-1 or degree-2 vertices.

A graph-map $f\colon\Gamma\to\Gamma$ is called a \emph{train track map} if the following hold:
\begin{enumerate}
\item $f$ is a homotopy equivalence and
\item for each $k\ge 1$ the graph-map $f^k\colon\Gamma\to\Gamma$ is regular (that, is, for every $k\ge 1$ and every $e\in E\Gamma$ the edge-path $f^k(e)$ is reduced).
\end{enumerate}
\end{df}

The definition above implies that if $f\colon\Gamma\to\Gamma$ is a train track map, then for each $n\ge 1$ $f^n\colon\Gamma\to\Gamma$ is also a train track map.

A train track map $f\colon\Gamma\to\Gamma$ is said to be \emph{irreducible} if its transition matrix $M(f)$ is irreducible.
A train track map $f\colon\Gamma\to\Gamma$ is said to be \emph{expanding} if for every edge $e\in E\Gamma$ we have $|f^n(e)|\to\infty$ as $n\to\infty$. Thus a train track map $f$ is expanding if and only if for each $e\in E\Gamma$ there exist $n\ge 1$ such that $|f^n(e)|\ge 2$.

\begin{df}[Derivative map]
Let $f\colon\Gamma\to\Gamma$ be a graph map. The \emph{derivative map} $Df: E\Gamma\to E\Gamma$ is defined as follows. 
For an edge $e\in E\Gamma$ with $f(e)=e_1,\dots, e_n$ we have $Df(e):=e_1$.
\end{df}

Note that the derivative map $Df$ is well-defined, even without the assumption that the graph-map $f$ be regular.
Note also that if $f,g\colon\Gamma\to\Gamma$ are graph-maps, then $D(f\circ g)=Df \circ Dg$. In particular, for each $k\ge 1$, we have that $D(f^k)=(Df)^k$.

An edge $e\in E\Gamma$ is called \emph{$f$-periodic} if for some $k\ge 1$ we have $(Df)^k(e)=e$, that is, if for some $k\ge 1$ the edge-path $f^k(e)$ starts with $e$.
A vertex $v\in V\Gamma$ is \emph{$f$-periodic} if for some $k\ge 1$ we have $f^k(v)=v$.

\begin{df}[Turns]
Let $\Gamma$ be a graph. For a vertex $v\in V\Gamma$ a \emph{turn} in $\Gamma$ at $v$ is an unordered pair $e,e'$ of (not necessarily distinct) oriented edges of $\Gamma$ such that $o(e)=o(e')=v$. A turn $e,e'$ is called \emph{degenerate} if $e=e'$ and is called \emph{non-degenerate} if $e\ne e'$.
For a graph $\Gamma$ we denote by $\mathcal T(\Gamma)$ the set of all turns in $\Gamma$ and denote by $\mathcal T_\times(\Gamma)$ the set of all non-degenerate turns in $\Gamma$.

For an edge-path $\gamma=e_1,\dots, e_n$ in $\Gamma$ we say that a turn $e,e'$ \emph{occurs} in $\gamma$ if there exists an $i$ such that  $\{e,e'\}=\{e_i^{-1},e_{i+1}\}$. We denote the set of all turns that occur in $\gamma$ by $\mathcal T(\gamma)$. Note that, by definition, $\mathcal T(\gamma)=\mathcal T(\gamma^{-1})$. Similarly, if $\alpha$ is a non-degenerate path in $\Gamma$ with associated combinatorial edge-path $\gamma$, we set $\mathcal T(\alpha):=\mathcal T(\gamma)$. If $\alpha$ is a degenerate path in $\Gamma$, we set $T(\alpha):=\emptyset$.

Note that if $f\colon\Gamma\to\Gamma$ is a graph-map, then the derivative map $Df\colon E\Gamma\to E\Gamma$ naturally extends to the map $Df\colon\mathcal T(\Gamma)\to \mathcal T(\Gamma)$ defined as $D(\{e,e'\}):=\{Df(e), Df(e')\}$, where $\{e,e'\}\in \mathcal T(\Gamma)$.
\end{df}

\begin{df}[Taken and legal turns]\label{d:T}
Let $f:\Gamma\to\Gamma$ be a graph-map.

We denote $\mathcal T(f):=\cup_{e\in E\Gamma} \mathcal T(f(e))$ and $\mathcal T_\infty(f):=\cup_{k\ge 1} \mathcal T(f^k)$. We refer to elements of $\mathcal T(f)$ as turns \emph{immediately taken by $f$} and to elements of $\mathcal T_\infty(f)$ as turns \emph{eventually taken by $f$}.
 
We also say that a non-degenerate turn $\{e,e'\}$ is \emph{$f$-legal} if the turn $\{Df^k(e), Df^k(e')\}$ is non-degenerate for each $k\ge 1$.
A turn $\{e,e'\}$ is  \emph{$f$-illegal} if it is not illegal. In particular, degenerate turns are always illegal.
\end{df}

We collect some basic elementary facts regarding turns and train track maps in the following proposition, whose proof is left to the reader:

\begin{prop}
Let $\Gamma$ be a finite connected graph without degree-1 and degree-2 vertices and let $f:\Gamma\to\Gamma$ be a graph-map which is a homotopy equivalence.

Then:
\begin{enumerate}
\item We have $Df\left(\mathcal T_\infty(f)\right)=\mathcal T_\infty(f)$.
\item If $f$ is a train track map, then every eventually taken turn by $f$ is legal.
\item The map $f$ is a train track map if and only if $f$ is regular and every turn in $\mathcal T(f)$ is legal.
\end{enumerate}
\end{prop}

Recall that if $v\in\Gamma$, we denote by $Lk_\Gamma(v)$ the set of all $e\in E\Gamma$ with $o(e)=v$ and refer to elements of $Lk_\Gamma(v)$ as \emph{directions} at $v$ in $\Gamma$.

\begin{df}[Local and limited Whitehead graphs]\label{d:W}
Let $f:\Gamma\to\Gamma$ be a graph-map. For a vertex $v\in V\Gamma$ we define the \emph{limited Whitehead graph} of $f$ at $v$, denoted $Wh_L(f,v)$, to be a graph with vertex set  $Lk_\Gamma(v)$ and with the set of topological edges defined as follows. To every turn $\{e,e'\}\in \mathcal T(f)$ such that $e,e'\in Lk_\Gamma(v)$, we associate a topological edge in  $Wh_L(f,v)$ with endpoints $e,e'\in Lk(v)$.

For a vertex $v\in V\Gamma$, define the \emph{local Whitehead graph} of $f$ at $v$ (also sometimes called the \emph{Whitehead graph} of $f$ at $v$), denoted $Wh(f,v)$, to be a graph with vertex set  $Lk_\Gamma(v)$ and with a topological edge with endpoints $e,e'\in Lk_\Gamma(v)$ whenever $\{e,e'\}\in \mathcal T_\infty(f)$. %(Note that $Wh(v,f)$ is referred to as the "local Whitehead graph for $f$ at $v$" in~\cite{HM11,CP}).
\end{df}

Thus, by definition, $Wh_L(f,v)$ is a subgraph of $Wh(f,v)$ and these graphs have the same vertex set, namely $Lk_\Gamma(v)$.

Note that if $f\colon\Gamma\to\Gamma$ is a regular graph-map, then $\mathcal T(f)$ contains no degenerate turns and hence $Wh_L(f,v)$ has no loop-edges.
In particular, if $f\colon\Gamma\to\Gamma$ is a train track map, then $Wh(f,v)$ has no loop-edges.

\begin{df}[Legal paths]
Let $f\colon\Gamma\to\Gamma$ be a train track map. A combinatorial edge-path $\gamma$ in $\Gamma$ is \emph{legal} if every turn in $\mathcal T(\gamma)$ is legal.
Similarly, a path $\alpha$ in $\Gamma$ is \emph{legal} if every turn in $\mathcal T(\alpha)$ is legal.
\end{df}

\subsection{Topological representatives}\label{subsect:tr}

For an integer $r\ge 2$ we fix a free basis $A=\{a_1,\dots, a_r\}$ of $F_r$. Let $R_r$ be the \emph{$r$-rose}, that is, a graph with a single vertex $v$ and $r$ topological loop-edges at $v$.
We choose an orientation on $R_r$ and an ordering $E_+R_r=\{e_1,\dots, e_r\}$ of $E_+R_r$. We identify $F_r=F(a_1,\dots, a_r)$ with $\pi_1(R_r,v)$ by sending $a_i\in F_r$ to the loop $e_i\in \pi_1(R_r,v)$.

\begin{df}[Marking]
Let $r\ge 2$ be an integer. A \emph{marking} for $F_r$ is a graph map $h \colon R_r\to \Gamma$, such that $\Gamma$ is a finite connected graph without degree-one and degree-two vertices, and such that $h \colon R_r\to \Gamma$ is a homotopy equivalence. 
\end{df} 

Note that if $h \colon R_r\to \Gamma$ is a marking, then $h$ naturally determines an isomorphism $h_\ast\colon\pi_1(R_r,v)\to \pi_1(\Gamma, h(v))$, which we can use to identify $F_r=\pi_1(R_r,v)$ with $\pi_1(\Gamma, h(v))$.

\begin{df}[Topological representative]
Let $\phi\in\Out(F_r)$, where $r\ge 2$. A \emph{topological representative} of $\phi$ is a marking $h \colon R_r\to\Gamma$ together with a graph-map $f:\Gamma\to\Gamma$ such that $f$ is a homotopy equivalence and such that the outer automorphism of $\pi_1(\Gamma)$, induced by $f$, is equal to $\phi$, modulo the identification of $F_r=\pi_1(R_r)$ and $\pi_1(\Gamma)$ via $h_\ast$.
More precisely, denote $v_0=h(v)\in V\Gamma$ and choose a path $\alpha$ in $\Gamma$ from $v_0$ to $f(v_0)$. Define $f_\ast\colon \pi_1(\Gamma,v_0)\to \pi_1(\Gamma,v_0)$ by sending (the homotopy class of) a closed path $\gamma$ at $v_0$ to the (homotopy class of) the closed path $\alpha f(\gamma) \alpha^{-1}$ at $v_0$. The fact that $f\colon\Gamma\to\Gamma$ is a homotopy equivalence implies that $f_\ast\colon \pi_1(\Gamma,v_0)\to \pi_1(\Gamma,v_0)$ is an isomorphism. Changing the choice of $\alpha$ results in modifying $f_\ast$ by a composition with an inner automorphism of $\pi_1(\Gamma,v_0)$, so that $f_\ast$ is well-defined as an outer automorphism of $\pi_1(\Gamma,v_0)$. Saying that the automorphism of $\pi_1(\Gamma)$, induced by $f$, is equal to $\phi$ modulo the identification of $F_r=\pi_1(R_r)$ and $\pi_1(\Gamma)$ via $h_\ast$ means that $h_\ast^{-1} \circ f_\ast \circ h_\ast\colon \pi_1(R_r,v)\to\pi_1(R_r,v)$ is an automorphism whose outer automorphism class is $\phi$.
\end{df}

Although in the above definition a topological representative of $\phi\in\Out(F_r)$ consists of a marking $h\colon R_r \to \Gamma$ and a graph-map $f\colon\Gamma\to\Gamma$, we usually will suppress the mention of the marking when talking about topological representatives and will refer to $f\colon\Gamma\to\Gamma$ as a topological representative of $\phi$. In the applications considered in this paper we will always work with the markings $h\colon R_r \to \Gamma$ where $\Gamma=R_r$ and $h=Id_{R_r}$, which makes explicitly mentioning the marking particularly redundant. 
If $f\colon \Gamma\to\Gamma$ is a topological representative of some $\phi\in\Out(F_r)$, then for any $\Phi\in\Aut(F_r)$ whose outer automorphism class is $\phi$ we also say that  $f\colon\Gamma\to\Gamma$  is a \emph{topological representative} of $\Phi$.

For $\phi\in\Out(F_r)$ a \emph{train track representative} of $\phi$ is a topological representative $f \colon \Gamma\to\Gamma$ such that $f$ is a train track map.

\begin{df}[Standard representative]\label{d:standard-rep}
Let $\Phi\in\Aut(F_r)$ and $\Phi(a_i)=x_{i,1}\dots x_{i,n_i}$ be a freely reduced word over $A^{\pm 1}$ of length $n_i\ge 1$, for $i=1,\dots, r$.

The \emph{standard representative} $g_\Phi$ of $\Phi$ is then defined as follows:
Use $\Gamma=R_r$ and $h=Id_{R_r}$ as the marking, so that $g_\Phi:R_r\to R_r$.  For each $i=1,\dots, r$, at the combinatorial edge-path level, we have $f(e_i)=e_{i,1}\dots e_{i,n_i}$, where $e_{i,k}\in ER_r$ is the edge corresponding to $x_{i,k}\in A^{\pm 1}$ under the identification $F(a_1,\dots,a_r)=\pi_1(R_r,v)$.
The subdivision of $[0,1]$ corresponding to the path $f\circ j_{e_i}:[0,1]\to R_r$ is chosen so that each subdivision interval of $[0,1]$ mapping to the edge $e_{i,k}$, $k=1,\dots, n_i$, has length $1/n_i$.
\end{df}

Note that if $\Phi,\Psi\in \Aut(F_r)$ are arbitrary, then $g_\Psi\circ g_\Phi\colon R_r\to R_r$ satisfies all the requirements of being a topological representative of $\Psi\circ \Phi$ except that the map $g_\Psi\circ g_\Phi$ may, in general, fail to be regular, since for an edge $e\in R_r$ the path $g_\Psi(g_\Phi(e))$ may fail to be reduced.
 
If for every $e\in ER_r$ the path $g_\Psi(g_\Phi(e))$  is reduced (that is, if the graph map $g_\Psi\circ g_\Phi \colon R_r\to R_r$ is regular), then $g_\Psi\circ g_\Phi\colon R_r\to R_r$ is indeed a topological representative of $\Psi\circ \Phi$. Moreover, in this case  $g_\Psi\circ g_\Phi$ and $g_{\Psi\circ \Phi}$ are isotopic rel $VR_r=\{v\}$.

\subsection{Stretch factors}\label{sect:stretch}

Let $A$ be a free basis of $F_r$, where $r\ge 2$.  For $w\in F_r$ we denote by $||w||_A$ the cyclically reduced length of $w$ with respect to $A$.

\begin{df}[Stretch factors]\label{D:stretch}
For $\phi\in\Out(F_r)$  and $w\in F_r$ put
\[
\lambda_A(\phi,w):=\limsup_{n\to\infty} \sqrt[n]{||\phi^n(w)||_A}.
\]
It is known that the actual limit in the above formula always exists, and  it is also known that $\lambda_A(\phi,w)$ depends only on $\phi$ and $w$, but not on the choice of a free basis $A$ of $F_r$. 
For this reason we denote $\lambda(\phi,w):=\lambda_A(\phi,w)$ where $A$ is any free basis of $F_r$. 
Now put $\lambda(\phi):=\sup_{w\in F_r\setminus\{1\}} \lambda(\phi,w)$. We call $\lambda(\phi)$ the \emph{stretch factor} or the \emph{growth rate}  of $\phi$.
\end{df}

It is known that for every $\phi\in \Out(F_r)$ there exists $w\in F_r$ with $\lambda(\phi)=\lambda(\phi,w)$, and moreover, that $\lambda(\phi)$ can be ``read-off'' from a relative train track representative of $\phi$. See \cite{Levitt} for details.
We will need only the simplest case of this fact here:

\begin{prop}\cite{BH92,Levitt}\label{P:stretch}
Let $r\ge 2$ and let $\phi\in \Out(F_r)$ be such that $\phi$ can be represented by an expanding train track map $f\colon\Gamma\to\Gamma$ with $M(f)$ irreducible.
Then
\[
\lambda(\phi)=\lambda(M(f)) >1.
\]
\end{prop}

\subsection{Nielsen paths}\label{ss:NielsenPaths}

\begin{df}[Nielsen paths]\label{d:NielsenPaths}
Let $f:\Gamma\to\Gamma$ be a train track representative of some $\phi\in\Out(F_r)$.

A \emph{Nielsen path} for $f$ is a nondegenerate tight path $\gamma$ in $\Gamma$ with endpoints $x,y\in \Gamma$ (where $x,y$ are not required to be vertices) such that $f(x)=x$, $f(y)=y$, and $f(\gamma)$ is homotopic to $\gamma$ rel endpoints. A \emph{periodic Nielsen path} for $f$ is a nondegenerate path $\gamma$ in $\Gamma$ such that for some $n\ge 1$ the path $\gamma$ is  Nielsen for $f^n$.

An \emph{indivisible Nielsen path} for $f$, abbreviated as INP, is a Nielsen path $\gamma$ for $f$ such that $\gamma$ cannot be written as a concatenation $\gamma=\gamma'\gamma''$, where $\gamma'$ and $\gamma'$ are Nielsen paths for $f$. Similarly, an \emph{indivisible periodic Nielsen path} for $f$, abbreviated as pINP, is a  periodic Nielsen path $\gamma$ for $f$ such that $\gamma$ cannot be written as a concatenation $\gamma=\gamma'\gamma''$, where $\gamma'$ and $\gamma'$ are periodic Nielsen paths for $f$.
\end{df}

It is known that pINPs have a specific structure, see \cite{BH92} Lemma 3.4:

\begin{prop}{\label{P:PINP}}
Let $f:\Gamma\to\Gamma$ be an expanding irreducible train track map. Then every pINP $\eta$ in $\Gamma$ has the form $\eta=\rho_1^{-1}\rho_2$, where $\rho_1$ and $\rho_2$ are nondegenerate legal paths with $o(\rho_1)=o(\rho_2)=v\in V\Gamma$ and such that the turn at $v$ between $\rho_1$ and $\rho_2$ is an illegal nondegenerate turn for $f$.  
\end{prop}

Note that in the context of Proposition~\ref{P:PINP}, there exist $k\ge 1$ such that $f^k$ fixes the points $t(\rho_1)$, $t(\rho_2)$. Note also that the points $t(\rho_1)$, $t(\rho_2)$ need not be vertices of $\Gamma$.

\subsection{Fully irreducible outer automorphisms}{\label{ss:FullyIrreducibles}}

If $G$ is a group and $w\in G$ is a group element, we denote by $[w]$ the conjugacy class of $w$ in $G$. Similarly, if $H\le G$ is a subgroup of $G$, we denote by $[H]$ the conjugacy class of $H$ in $G$. 

\begin{df}{\label{d:FullyIrreducibles}}
Let $\phi\in \Out(F_r)$ where $r\ge 3$. The element $\phi$ of $\Out(F_r)$ is said to be \emph{fully irreducible} if there does not exist an integer $k\ge 1$ and a nontrivial proper free factor $B$ of $F_r$, such that $\phi^k([B])=[B]$.
\end{df}

Recall also that an element $\phi\in \Out(F_r)$ is called \emph{atoroidal} if there does not exist $1\ne w\in F_r$ and an integer $k\ge 1$ such that $\phi^k([w])=[w]$. For $\phi\in\Out(F_r)$, a conjugacy class $[w]$, where $w\in F_r, w\ne 1$, is called \emph{$\phi$-periodic} if there exists a $k\ge 1$ such that $\phi^k([w])=[w]$. Thus $\phi\in \Out(F_r)$ is atoroidal if and only if $\phi$ has no periodic conjugacy classes.

A special case of an important general result of Bestvina and Handel \cite{BH92} shows that if $\phi\in\Out(F_r)$ (where $r\ge 2$) is fully irreducible, then there exists an expanding irreducible train track representative $f \colon \Gamma\to\Gamma$ of $\phi$.

Another key result of Bestvina and Handel \cite{BH92} provides a complete characterization of non-atoroidal fully irreducible elements of $\Out(F_r)$:

\begin{prop}\label{p:BH92}\cite{BH92}
Let $\phi\in\Out(F_r)$, where $r\ge 2$ and suppose that $\phi$ is non-atoroidal. Then $\phi$ is fully irreducible if and only if there exists a compact connected surface $\Sigma$ with a single boundary component such that $b_1(\Sigma)=r$ (so that the fundamental group of $\Sigma$ is free of rank $r$), an identification $\pi_1(\Sigma)=F_r$, and a pseudo-Anosov homeomorphism $g:\Sigma\to \Sigma$ such that the outer automorphism of $\pi_1(\Sigma)$ induced by $g$ is equal to $\phi$.
\end{prop}

In view of the above result of Bestvina-Handel, fully irreducible elements of $\Out(F_r)$ are divided into two main classes: non-atoroidal fully irreducible elements of $\Out(F_r)$ are said to be \emph{geometric} and atoroidal fully irreducible elements of $\Out(F_r)$ are said to be \emph{nongeometric}.

It is well-known that $\Out(F_2)$ contains no atoroidal elements. Therefore, if $\phi\in\Out(F_r)$ is a nongeometric (i.e. atoroidal) fully irreducible, then $r\ge 3$.

A recent result of Kapovich~\cite{K14} (see also \cite{DKL}), building on the work of Pfaff~\cite{pII}, gives a train track characterization of nongeometric fully irreducibles. For completeness, and to give context, we state the result in the form it is stated in in \cite{K14}:

\begin{prop}\label{P:K14}\cite{K14}
Let $r\ge 3$ and let $\phi\in \Out(F_r)$ be an arbitrary atoroidal element.

Then the following are equivalent:

\begin{enumerate}
\item The automorphism $\phi$ is fully irreducible.

\item There exists a train track representative $f\colon\Gamma\to\Gamma$ of $\phi$ such that the matrix $M(f)$ is irreducible and such that for each $v\in V\Gamma$ the local Whitehead graph $Wh(f,v)$ is connected.

\item There exists a train track representative $f\colon\Gamma\to\Gamma$ of $\phi$ such that for some $k\ge 1$,  $M(f^k)>0$ and such that for each $v\in V\Gamma$, the local Whitehead graph $Wh(f,v)$ is connected.

\item For each train track representative $f\colon\Gamma\to\Gamma$ of $\phi$ there exists a $k\ge 1$ such that $M(f^k)>0$  and the local Whitehead graph $Wh(f,v)$ is connected for each $v\in V\Gamma$.

\end{enumerate}

\end{prop}

However, the form we in fact use it in is that of \cite{pII}:

\begin{prop}\label{P:FIC}\cite{pII} 
Let $g:\Gamma\to\Gamma$ be a train track representative of an outer automorphism $\phi \in Out(F_r)$ such that
~\\
\vspace{-6mm}
\begin{itemize}
\item [(I)] $g$ has no periodic Nielsen paths,
\item [(II)] the transition matrix for $g$ is Perron-Frobenius, and 
\item [(III)] all local Whitehead graphs $Wh(g,v)$ (where $v$ varies over the vertices of $\Gamma$) for $g$ are connected.
\end{itemize}
\indent Then $\phi$ is fully irreducible. Moreover, Proposition~\ref{p:BF94} implies that $\phi$ is ageometric (see the definition of ageometric fully irreducibles in Section~\ref{sect:index} below).
\end{prop}

\subsection{Index and geometric index}\label{sect:index}

The notion of an \emph{index} $\Ind(\phi)$  of an element of $\phi\in\Out(F_r)$ was originally introduced in~\cite{GJLL} and formulated in terms of the dynamics of the action of representatives $\Phi\in \Aut(F_r)$ of $\phi$ on the hyperbolic boundary $\partial F_r$ of $F_r$.  Since we are not going to work with the $\Ind(\phi)$ directly, we omit the precise definition here and refer the reader to \cite{GJLL,ch12} for details. Note, however, that the index $\Ind(\phi)$  is not, in general, invariant under taking positive powers of $\phi$. There is a natural notion (again see \cite{GJLL,ch12}) of a \emph{rotationless} element $\phi\in\Out(F_r)$, also defined in terms of the action on $\partial F_r$.  It is known that every element of $\Out(F_r)$ has a positive rotationless power and that if $\phi\in\Out(F_r)$ is rotationless, then $\Ind(\phi)=\Ind(\phi^k)$ for all $k\ge 1$.

Recall that the \emph{unprojectivized Culler-Vogtmann Outer space} $cv_r$ consists of all minimal free discrete isometric actions on $F_r$ on $\R$-trees, considered up to an $F_r$-equivariant isometry. Points of $cv_r$ can also be described in terms of ``marked metric graph structures'' on $F_r$. There is a natural action of $\R_{>0}$ on $cv_r$ by multiplying the metric on $T\in cv_r$ by a positive scalar. The quotient space $cv_r/\R_{>0}$ is called the \emph{projectivized Culler-Vogmtann Outer space} and denoted $CV_r$. The space $CV_r$ can also be canonically identified with the set of all $T\in cv_r$ such that the quotient metric graph $T/F_r$ has volume 1. 

It is known~\cite{bf94,CL95,g00} that the closure $\overline{cv}_r$ of $cv_r$, with respect to equivariant Gromov-Hausdorff convergence topology, consists of all ``very small'' isometric actions of $F_r$ on $\R$-trees, again considered up to an $F_r$-equivariant isometry.
The projectivization $\overline{CV}_r:=\overline{cv}_r/\R_{>0}$ of  $\overline{cv}_r$ is compact and provides the standard compactification of $CV_r$. 

For an $\R$-tree $T$ and a point $P\in T$, we denote by $\deg_T(P)$ the number of connected components of $T-\{P\}$. A point $P\in T$ is called a \emph{branch point} if $\deg_T(P)\ge 3$.

In \cite{GL}, Gaboriau and Levitt introduced the notion of a ``geometric
index'' or ``branching index'' for any $T\in \overline{cv}_r$. In the same paper they proved that for any free $F_r$-tree $T\in \overline{cv}_r$ the number of $F_r$-orbits of branch-points of $T$ is finite and is bounded above by $2r-2$. For a point $P\in T$ we will denote by $[P]$ the
$F_r$-orbit of $P$. For simplicity, we will only define the geometric index for free $F_r$-trees (as noted below, for every nongeometric fully irreducible $\phi\in\Out(F_r)$, the action of $F_r$ on the ``attracting tree'' $T_\phi$ is free). See \cite{GL,ch12} for the definition in the case of an arbitrary $T\in \overline{cv}_r$.

\begin{df}[Geometric index]\label{d:gind}
Let $T\in \overline{cv}_r$ be a free $F_r$-tree. Define the \emph{geometric
  index} $\gind(T)$ as
\[
\gind(T):=\sum_{[P]: \deg(P)\ge 3} [\deg_T(P)-2].
\]
\end{df}

In particular, Gaboriau and Levitt proved in \cite{GL} that for
any $T\in \overline{cv}_r$ we have $\gind(T)\le 2r-2$.

For every fully irreducible $\phi\in F_r$ there is an associated \emph{attracting tree} $T_\phi\in  \overline{cv}_r$, which is unique up to projectivization, that is up to multiplying the metric on $T_\phi$ by a positive scalar. We recall an explicit construction of $T_\phi$ in terms of train tracks; see \cite{GJLL} for details. Let $\phi\in\Out(F_r)$ be fully irreducible and let $f:\Gamma\to\Gamma$ be a train track representative of $\phi$. 
Let $\lambda:=\lambda(f)>1$ be the Perron-Frobenius eigenvalue of $M(f)$. We consider the universal cover $\tilde\Gamma$ of $\Gamma$ with the simplicial metric (where every edge has length $1$) and with the free discrete isometric action of $F_r$ by covering transformations. Choose a lift $\tilde f: \tilde\Gamma\to\tilde\Gamma$ of $f$. For any points $x,y\in \tilde \Gamma$ put
\[
d_\infty(x,y)=\lim_{n\to\infty} \frac{d((\tilde f)^n(x), (\tilde f)^n(y))}{\lambda^n}
\]
(it is know that this limit exists). Then $d_\infty$ is a pseudo-metric on $\tilde\Gamma$ and we put $T_\phi:=\tilde \Gamma/\sim$, where for $x,y\in \tilde\Gamma$ we have $x\sim y$ whenever $d_\infty(x,y)=0$. The pseudo-metric $d_\infty$ descends to a metric, still denoted $d_\infty$, on $T_\phi$. Equipped with the metric $d_\infty$, the set $T_\phi$ is an $\R$-tree, which also inherits a natural action of $F_r$ by isometries (coming from the action of $F_r$ on $\widetilde\Gamma$). The $\R$-tree $T_\phi$, endowed with this action of $F_r$, is called the \emph{attracting tree} of $\phi$. It is known that for every fully irreducible $\phi\in \Out(F_r)$ we have $T_\phi\in  \overline{cv}_r$, and, moreover, that in this case the action of $F_r$ on $T_\phi$ is free if and only if $\phi$ is nongeometric. Also, it is easy to see that $T_{\phi}=T_{\phi^k}$, for each $k\ge 1$. In fact, it is known~\cite{ll03} that for any fully irreducible $\phi\in\Out(F_r)$ the projective class $[T_\phi]$ is the unique attracting fixed point for the action of $\phi$ on the projectivization  $\overline{CV}_r$ of  $\overline{cv}_r$, that $[T_{\phi^{-1}}]$ is the unique repelling fixed point for the action of $\phi$ on $\overline{CV}_r$, and that $\phi$ acts on $\overline{CV}_r$  with ``uniform North-South'' dynamics.

The following proposition summarizes key known facts about the relationship between the index of a fully irreducible $\phi\in \Out(F_r)$ and the geometric index of the tree $T_\phi$.
Most of these facts are originally proved in earlier work by various authors~\cite{GJLL,GL,bf94,g05,hm07}, and others. All parts of Proposition~\ref{p:class} are explicitly stated in \cite{ch12}, and we refer the reader to \cite{ch12} for more detailed background information and references.

\begin{prop}\label{p:class}
Let $\phi\in\Out(F_r)$ be fully irreducible.
Then:
\begin{enumerate}
\item We have $\gind(T_\phi)\le 2r-2$.
\item We have $\gind(T_\phi)= \gind(T_{\phi^{-1}})=2r-2$ if and only if $\phi$ is geometric.
\item If $\phi^{-1}$ is rotationless, then $2\Ind(\phi)=\gind(T_\phi)$.  
\item The tree $T_\phi$ is ``geometric'' in the sense of~\cite{bf94,LP97} if and only if $\gind(T_\phi)=2r-2$.
\end{enumerate}
\end{prop}

\begin{df}[Parageometric and ageometric fully irreducibles]
Let $\phi\in\Out(F_r)$ be a nongeometric fully irreducible. We say that $\phi$ is \emph{parageometric} if $\gind(T_\phi)=2r-2$ (which by Proposition~\ref{p:class} implies that $\gind(T_{\phi^{-1}})<2r-2$). We say that $\phi$ is \emph{ageometric} if $\gind(T_\phi)<2r-2$.
\end{df}

\cite[Theorem 3.2]{bf94} shows that for a fully irreducible $\phi\in\Out(F_r)$ the tree $T_\phi$ is geometric (again in the sense of~\cite{bf94,LP97}) if and only if the ``stable'' train track representative of $\phi$ contains a pINP. Since a train track with no pINPs is stable, this implies the following well-known fact: 
\begin{prop}\label{p:BF94}
Let $\phi\in\Out(F_r)$ be fully irreducible and such that $\phi$ admits a train track representative $f\colon\Gamma\to\Gamma$ with no pINPs.
Then $\phi$ is ageometric.
\end{prop}

There is a more recent notion of an index, namely ``$\mathcal Q$-index'' $\ind_\mathcal Q(T)$ defined for a tree $T\in\overline{cv_r}$ with dense $F_r$-orbits (such as the attracting tree $T_\phi$ of a fully irreducible $\phi\in\Out(F_r)$, see \cite{ch12,ch14} for details). In particular, it is known~\cite[Theorem 5.1]{ch12} that for every fully irreducible $\phi\in\Out(F_r)$ one has $\ind_\mathcal Q(T_\phi)=\gind(T_{\phi^{-1}})$. 

\subsection{Ideal Whitehead graphs and the rotationless index}{\label{ss:IWGs}}

Note that the index $\Ind(\phi)$ of a fully irreducible $\phi\in\Out(F_r)$, discussed above, in general is not invariant under taking positive powers of $\phi$. The geometric index $\gind(T_\phi)$ is invariant under taking positive powers, but the definition of $\gind(T_\phi)$ in terms of branch points in $T_\phi$ makes it unclear how to actually compute $\gind(T_\phi)$.

In \cite{hm11} Handel and Mosher introduced the notion of a rotationless index (there just called the index sum) $i(\phi)$ of a fully irreducible $\phi\in\Out(F_r)$, which coincides with $-\frac{1}{2}\gind(T_\phi)$. The rotationless index is defined in terms of train track representatives of $\phi$, which makes it easy to compute in practice. Proposition~\ref{p:class} thus implies that  for a fully irreducible $\phi \in Out(F_r)$ the rotationless index $i(\phi$) satisfies the inequality $0 > i(\phi) \geq 1-r$. 

To define the rotationless index, we first need to define the ideal Whitehead graph (introduced by Handel and Mosher in ~\cite{hm11}) for a nongeometric fully irreducible.

\begin{df}[Ideal Whitehead graph]{\label{d:IWG}}
Let $g\colon \Gamma \to \Gamma$ be a train track representative of a nongeometric fully irreducible $\phi \in Out(F_r)$. A point $v \in \Gamma$ is called a \emph{singularity} if $v$ is either the endpoint of a periodic Nielsen path or has at least three gates. The \emph{local stable Whitehead graph} $\mathcal{SW}(g; v)$ for $g$ at a singularity $v$ has:

(1) a vertex for each periodic direction $d \in \mathcal{D}(v)$ and

(2) edges connecting vertices for $d_1, d_2 \in \mathcal{D}(v)$ when some $g^k(e)$, with $e \in \mathcal{E}(\Gamma)$, traverses $\{d_1, d_2 \}$.

For a pINP-free $g$, the ideal Whitehead graph $\mathcal{IW}(\phi)$ of $\phi$ is defined as 
$$\bigsqcup\limits_{\text{singularities v} \in \Gamma} \mathcal{SW}(g;v).$$
In general, one needs to make the following additional identifications. For each pINP $\rho$ for $g$, one needs to identify the vertex for the initial direction of $\rho$ with the vertex for the initial direction of $\bar{\rho}$.
\end{df}

\begin{df}[Rotationless index and index list]{\label{d:RotationlessIndex}}
Let $\phi\in\Out(F_r)$ be a nongeometric fully irreducible outer automorphism and let $C_1, \dots, C_l$ be the connected components of the ideal Whitehead graph $\mathcal{IW}(\phi)$. For each $j$, let $k_j$ denote the number of vertices of $C_j$. The \emph{index list} for $\phi$ is

\begin{equation}{\label{E:IndexList}}
\{i_1, \dots, i_j, \dots, i_l\} = \{1-\frac{k_1}{2}, \dots, 1-\frac{k_j}{2}, \dots, 1-\frac{k_l}{2}\}.
\end{equation} 

\noindent The \emph{rotationless index} is then the sum $i(\phi) = \sum\limits_{j=1}^{l} i_j$.
\end{df}

From this definition one observes that it is possible to obtain the index list (hence index sum) directly from any pINP-free train track representative $g \colon \Gamma \to \Gamma$. The $k_i$ in Equation \ref{E:IndexList} are replaced by the number of gates $k_i$ at the singular vertices $v_i \in \Gamma$. The rotationless index is then computed as follows (where singularities here will precisely mean vertices with at least three gates):

\begin{equation} 
i(\phi) = \sum\limits_{\text{singularities v}}(1-\frac{\#(\text{gates at v})}{2}).
\end{equation}

If there are pINPs, the situation is somewhat more complicated. However, \cite{p13} provides a method for computing the index list directly from any train track representative $g\colon \Gamma \to \Gamma$ of a nongeometric fully irreducible. Let $v_1, \dots, v_k$ be the periodic vertices for $g$ and, for each $1 \leq i \leq k$, let $n_i$ denote the number of gates at the vertex $v_i$. We define an equivalence relation on the set of all periodic points by $x_i \sim x_j$ when there exists a pINP with endpoints $x_i$ and $x_j$ and call an equivalence class a \emph{Nielsen class}. Given a Nielsen class $N_i = \{x_1, \dots, x_n \}$, we let $g_j$ denote the number of gates at $x_j$. Then, letting 
\begin{center}
$n_i = (\sum g_i) - \# \{\text{iPNPs}$ $\rho$ $\text{such that both endpoints of}$ $\rho$ $\text{ are in } N_i \},$
\end{center}
the index list becomes 
$$\{1-\frac{n_1}{2}, \dots, 1-\frac{n_t}{2} \},$$
where we only include nonzero entries. The rotationless index is thus the sum $i(\phi) = \sum\limits_{j=1}^{t} 1-\frac{n_j}{2}$.

\begin{rk}
An explanation of why there are only finitely many nonzero entries and how this computation is finite can be found in \cite{p13}
\end{rk}

The following key fact relates the rotationless index with the other notions of index described above (while the conclusion of Proposition~\ref{p:hm11} does not appear to have been stated by Handel and Mosher in \cite{hm11} explicitly, it follows directly from the definitions of $\gind$ and $i(\phi)$ and from Lemma~3.4 in \cite{hm11}, which provides an identification between $F_r$-orbits of branch points in $T_\phi$ and components of the ideal Whitehead graph of $\phi$):
\begin{prop}\label{p:hm11}\cite{hm11}
Let $\phi\in\Out(F_r)$ be a nongeometric fully irreducible.
Then $i(\phi)=-\frac{1}{2}\gind(T_\phi)$.
\end{prop}
Moreover, Handel and Mosher (again see \cite[Lemma 3.4]{hm11}) also show that the index list of a nongeometric fully irreducible $\phi$ can be interpreted directly in terms of the tree $T_\phi$. Therefore, the index list (and not just the index sum $\phi$) depends only on $\phi$ and not on the choice of a train track representative of $\phi$.

In view of Proposition~\ref{p:hm11}, we immediately obtain:

\begin{cor}
Let $\phi\in \Out(F_r)$ be a nongeometric fully irreducible. Then:
\begin{enumerate}
\item $\phi$ is parageometric if and only if $i(\phi)=1-r$.
\item $\phi$ is ageometric if and only if $0>i(\phi)>1-r$.
\end{enumerate} 
\end{cor}

%\begin{rk}\label{r:ageometrics}\footnote{Do we still need this remark?}
%This division of the set of nongeometric fully irreducibles into ageometric and parageometric outer automorphisms could be considered to have evolved out of a series of papers. In \cite{GL} Gaboriau and Levitt define the geometric index $\gind(T)$ for an $\R$-tree equipped with a minimal, small $F_r$-action. They prove that the index satisfies the inequality $\frac{1}{2} \leq \gind(T) \leq r-1$, with the equality $\gind(T)=r-1$ realized precisely by ``geometric trees.'' In \cite{GJLL} it is proved that, after replacing a fully irreducible $\phi$ by a suitable positive power, one has $\gind(T_{\phi}) = 2\Ind(\phi)$. While it had been previously known that geometric fully irreducibles have geometric attracting tree, Levitt proved in \cite{l93} that even some nongeometric fully irreducible outer automorphisms have geometric attracting tree, hence creating a natural division of nongeometric fully irreducible outer automorphisms by their index sum. Much of this history is outlined in~\cite[Section 6.4]{GJLL}.
%\end{rk}

\subsection{The axis bundle for a fully irreducible}\label{s:AxisBundle}

We call a point $\Gamma\in CV_r$ in Outer space a \emph{train track graph}  for $\phi$ if there exists an \emph{affine} train track representative $g \colon \Gamma \to \Gamma$, i.e. a train track representative on $\Gamma$ such that each open interval in the interior of each edge is stretched by a constant factor equal to the dilitation $\lambda(\phi)$ of $\phi$.

In \cite{hm11}, Handel and Mosher define the axis bundle for a nongeometric fully irreducible to answer the question posed by Vogtmann as to whether the set of train tracks for an irreducible automorphism contractible. The axis bundle $\mathcal{A}_{\phi}$ is a closed subset of $CV_r$, proper homotopy equivalent to a line, invariant under $\phi$, and such that the two ends limit on the repeller and attractor of the source-sink action of $\phi$ on $\overline{CV_r}$.

\cite{hm11} gives three equivalent definitions of the axis bundle, one of which we include here: 

\begin{df}[Axis bundle $\mathcal{A}_{\phi}$]{\label{d:AxisBundle}}
For a nongeometric fully irreducible $\phi \in Out(F_r)$, the \emph{axis bundle} $\mathcal{A}_{\phi}$ for $\phi$ is defined as $\mathcal{A}_{\phi} = \overline{\cup_{k=1}^{\infty} TT(\phi^k)}$, where $TT(\phi^k)$ is the set of all train track graphs for $\phi^k$, where $k\ge 1$. 
\end{df}

If $\phi\in\Out(F_r)$ is a nongeometric fully irreducible and $f \colon \Gamma\to\Gamma$ is a train track representative of $\phi$, we can equip $\Gamma$ with a volume-1 ``eigenmetric'' (see \cite{DKL} for a detailed explanation),  so that, viewed as a marked metric graph, $\Gamma$ becomes a train track graph for $\phi$ in the above sense. Then taking an ``isometric'' folding path, determined by $f$, from $\Gamma$ to $\Gamma\cdot \phi$ in $CV_r$, and translating this path by all integer powers of $\phi$, gives a $\phi$-invariant bi-infinite folding line $A_f\subseteq CV_r$.  This line $A_f$ is contained in the axis bundle $\mathcal{A}_{\phi}$ for $\phi$ and is called an \emph{axis} for $\phi$. Moreover, the line $A_f$ is a geodesic in $CV_r$ with respect to the asymmetric Lipschtz metric on $CV_r$, see \cite{fm11}.

While the axis bundle generally contains more than a single axis, Mosher and Pfaff prove in \cite{mp13}:

\begin{thm}{\label{t:AxisBundle}}The axis bundle of an ageometric, fully irreducible outer automorphism $\phi \in Out(F_r)$ is a unique axis precisely if both of the following two conditions hold:
\begin{enumerate}
\item the index sum satisfies $i(\phi) = \frac{3}{2}-r$ and 
\item no component of the ideal Whitehead graph $\mathcal{IW}(\phi)$ has a cut vertex.
\end{enumerate}
\end{thm}

It can be noted that $\mathcal{A}_{\phi}$ and $\mathcal{A}_{\psi}$ differ by a translation by an element of $Out(F_r)$ on $CV_r$ if and only if there exist integers $k,l \geq 1$ such that $\phi^k$ and $\psi^l$ are conjugate in $Out(F_r)$. Also, \cite{mp13} provides a method for computing the axis bundle in the case of Theorem \ref{t:AxisBundle}. Thus, identifying when two fully irreducible outer automorphisms satisfy the conditions of Theorem \ref{t:AxisBundle}, allows one to identify if they fall into a setting where one can ``by hand'' compute whether they have conjugate powers.

\section{Admissible compositions of Nielsen automorphisms}\label{sect:adm}

\begin{conv}
Recall that, for $r\ge 2$, for the free group $F_r$ we have chosen a distinguished free basis $A=\{a_1,\dots, a_r\}$ for $F_r$. The $r$-rose $R_r$ was defined in Section~\ref{subsect:tr} using as a wedge of $r$ loop-edges $e_1,\dots, e_r$ corresponding to $a_1,\dots, a_n$ at a single vertex $v$, giving an identification $F_r=F(a_1,\dots,a_r)=\pi_1(R_r,v)$.
Thus, $ER_r=\{e_1,\dots, e_r, \overline{e_1},\dots, \overline{e_r}\}$, with the identification $F(a_1,\dots,a_r)=\pi_1(R_r,v)$ sending $a_i$ to $e_i$ and, correspondingly, sending $a_i^{-1}$ to $\overline{e_i}$.

Note that in this case (the case where $\Gamma = R_r$ and $v$ is the vertex of $\Gamma$), for the set of directions $Lk(v)$ at $v$, we have $Lk(v)=ER_r$. By convention, using the above identification $ER_r=A^{\pm 1}$, we will sometimes use the identification $ER_r=Lk(v)=A^{\pm 1}$ and view elements of $A^{\pm 1}$ as directions at $v$ in $R_r$. In particular, we will use this convention when working with local Whitehead graphs and limited Whitehead graphs of graph-maps $R_r\to R_r$ and with turns taken by such maps.

Also, since $R_r$ has a single vertex $v$, when dealing with local Whitehead graphs and limited Whitehead graphs of graph-maps $g \colon R_r\to R_r$, we will usually use the abbreviated notations $Wh(g):=Wh(g,v)$ and $Wh_L(g):=Wh_L(g,v)$.

\end{conv}

\subsection{Standard Nielsen automorphisms and admissible sequences.}

\begin{df}[Standard Nielsen automorphism]
Let $r\ge 2$. By a \emph{standard Nielsen automorphism}, we will mean an automorphism $\theta$ of $F_r$ such that there exist $x,y\in A^{\pm 1}$ with $\theta(x)=yx$ and $\theta(z)=z$ for each $z\in A^{\pm 1}$ with $z\ne x^{\pm 1}$.
In this case we say that the ordered pair $(x,y)$ is the \emph{characteristic pair} for $\theta$ and we specify such $\theta$ using notation $\theta=[x\mapsto yx]$.
\end{df}
Note that if $\theta=[x\mapsto yx]$, then the fact that $\theta$ is an automorphism of $F_r$ implies that $y\ne x^{\pm 1}$.

The following notion is based on the work of Pfaff~\cite{pI,pII},
although the terminology that we use here is slightly different.

\begin{df}[Admissible compositions]

Let $\theta=[x\mapsto yx]$ and $\theta'=[x'\mapsto y'x']$ be standard Nielsen automorphisms of $F_r$.
The ordered pair $(\theta,\theta')$ is called \emph{admissible} if either $x' = x$ and $y' \neq y^{-1}$ or $y' = x$ and $x' \neq y^{-1}$.

A sequence $\theta_1,\dots, \theta_n$ (where $n\ge 1$) of standard Nielsen automorphisms of $F_r$ is called \emph{admissible} if for each $1\le i<n$ the pair $(\theta_i,\theta_{i+1})$ is admissible. 
In this case we also say that the composition $\theta_n\circ \dots \circ \theta_1$ is \emph{admissible}.

A sequence $\theta_1,\dots, \theta_n$ of standard Nielsen automorphisms of $F_r$ is called \emph{cyclically admissible}
if it is admissible and if the pair $(\theta_n,\theta_1)$ is also admissible (that is, if $\theta_{n-1}\circ \dots \circ \theta_1\circ \theta_n$ is also admissible).  In this case we also say that the corresponding composition $\theta_n\circ \dots \circ \theta_1$ is \emph{cyclically admissible}.

If $\mathfrak t=\theta_1,\dots, \theta_n$ is a sequence of  standard
Nielsen automorphisms of $F_r$, and if $k\ge 1$ is an integer, we
denote by $\mathfrak t^k$ the sequence \[
\underbrace{\theta_1,\dots, \theta_n,
\theta_1,\dots, \theta_n, \dots, \theta_1,\dots, \theta_n}_{k \text{
  copies of } \mathfrak t}.
\]
Note that if $\mathfrak t$ is cyclically admissible then for every
$k\ge 1$ the sequence $\mathfrak t^k$ is also cyclically admissible
(and, in particular, admissible).
\end{df}

Recall that in Definition~\ref{d:standard-rep} to every $\Phi\in\Aut(F_r)$ we have associated its standard topological representative $g_\Phi:R_r\to R_r$.
Given $\Phi \in \Aut(F_r)$, denote by $\mathcal T(\Phi):=\mathcal T(g_\Phi)$ the set of all turns in $R_r$  that occur in $g_\Phi(e_i)$, where $i=1,\dots, r$ (see Definition~\ref{d:T}).

The following basic lemma is a direct corollary of the definitions:

\begin{lem}\label{L:basic}
Let $\theta=[x\mapsto yx]$ be a standard Nielsen automorphism of $F_r$ and let $g_\theta:R_r\to R_r$ be its topological representative.

Then:
\begin{enumerate}
\item The set $\mathcal T(g_\theta)$ consists of a single turn $\{y^{-1},x\}$.
\item $g_\theta:R_r\to R_r$ is a train track map with exactly one nondegenerate illegal turn in $R_r$, namely the turn $\{x,y\}$.
\item We have $Dg_\theta(x)=Dg_\theta(y)=y$, and we have $Dg_\theta(a)=a$ for every $a\in A^{\pm 1}$, $a\ne x$.
\item We have $Dg_\theta(A^{\pm 1})=A^{\pm 1}-\{x\}$. 
\end{enumerate}
\end{lem}
Because of part (3) of Lemma~\ref{L:basic} we refer to $x$ as \emph{the unachieved direction} for $\theta=[x\mapsto yx]$.

Let $\mathfrak t=\theta_1,\dots, \theta_n$ be a sequence of standard Nielsen automorphisms of $F_r$, where $g_{\theta_i}:R_r\to R_r$ is the standard representative of $\theta_i$ (see Definition~\ref{d:standard-rep}.)
In this case we denote $g_{\mathfrak t}:=g_{\theta_ n}\circ \dots\circ g_{\theta_1}:R_r\to R_r$.
Note that the map $g_{\mathfrak t}$ may, in general, fail to be a regular graph-map, since for some edge $e_k\in ER_r$ the path $g_{\theta_ n}\circ \dots\circ g_{\theta_1}(e_k)$ may fail to be reduced. However, if $g_{\mathfrak t}$ is regular, then $g_{\mathfrak t}$  is a topological representative of $\Phi=\theta_n\circ \dots \circ \theta_1\in \Aut(F_r)$ and in this case $g_{\mathfrak t}$ is isotopic to $g_\Phi$ rel $VR_r=\{v\}$. 
We will see below that if the sequence $\mathfrak t=\theta_1,\dots, \theta_n$ is admissible, then indeed  $g_{\mathfrak t}$ is regular and, moreover, $g_{\mathfrak t}$ is a train track representative of $\Phi$ with some additional nice properties.

\begin{conv}\label{conv:main}
Unless specified otherwise, for the remainder of Section~\ref{sect:adm} we fix an admissible sequence 
\[
\mathfrak t=\theta_1,\dots, \theta_n \tag{$\dag$}
\]
of standard Nielsen automorphisms $\theta_i=[x_i\mapsto y_ix_i]$ of $F_r$ (where $n\ge 1$ and $i=1,\dots, n$) and fix the corresponding composition automorphism $\Phi=\theta_n\circ \dots \circ \theta_1\in \Aut(F_r)$.

For $1\le k\le m\le n$ we denote $\mathfrak t_{k,m}=\theta_k,\dots, \theta_m$, $g_{m,k}:=g_{\mathfrak t_{k,m}}=g_{\theta_m}\circ \dots \circ g_{\theta_k} \colon R_r\to R_r$ and $\Phi_{m,k}=\theta_m\circ \dots  \circ  \theta_k\in \Aut(F_r)$. 
Also, for $1\le m\le n$ denote $g_m:=g_{m,1}$ and $\Phi_m:=\Phi_{m,1}$. Thus $\Phi=\Phi_n=\Phi_{n,1}$ and $g_\mathfrak t=g_n=g_{n,1}$.
\end{conv}

Note that since $\mathfrak t_n=\theta_1,\dots, \theta_n$ is an admissible sequence, for every $1\le k\le m\le n$ the sequence $\mathfrak t_{k,m}=\theta_k,\dots, \theta_m$ is also admissible.

\subsection{Properties of admissible sequences}

\begin{lem}\label{L:T}
For $1\le k \le m\le n$ we have 
\[
\mathcal{T}(g_{m,k}) \subseteq \{D(g_{m,j+1})(\{y_j^{-1}, x_j\}) \mid j=k,\dots, m\},
\]
where in the case $k=m$ we interpret $g_{m,m+1}$ as the identity map $Id \colon R_r\to R_r$.
\end{lem}

\begin{proof}
We argue by induction on $m-k$.

If $m-k=0$ and $m=k$ then $g_{m,k}=\theta_m=[x_m\to y_mx_m]$ and the statement of the lemma holds.
Suppose now that $m-k\ge 1$ and that the conclusion of the lemma has been established for all smaller values of $m-k$.

Let $\{a,b\}$ (where $a,b\in A^{\pm 1}$) be a turn in $\mathcal{T}(g_{m,k})$. Then the turn $\{a,b\}$ occurs in $g_{m,k}(c)$ for some $c\in A^{\pm 1}$.

Suppose first that $c\ne x_k^{\pm 1}$. Then $\theta_k(c)=c$ and 
\[
g_{m,k}(c)=g_{\theta_m}\circ \dots  \circ g_{\theta_k}(c)=g_{\theta_m}\circ \dots \circ g_{\theta_{k+1}}(c)=g_{m,k+1}(c).
\]
By the inductive hypothesis applied to $g_{m,k+1}$ we have 
\[
\mathcal T(g_{m,k+1})\subseteq   \{D(g_{m,j+1})(\{y_j^{-1}, x_j\})| j=k+1,\dots, m\}
\]
and hence $\{a,b\}\in  \{D(g_{m,j+1})(\{y_j^{-1}, x_j\}) \mid j=k,\dots, m\}$.

Suppose now that $c=x_k^{\pm 1}$. Since $\mathcal T (g_{m,k}(c))=\mathcal T (g_{m,k}(c^{-1}))$, without loss of generality we may assume that $c=x_k$.
Then $\theta_k(c)=y_kx_k$ and hence $g_{m,k}(c)=g_{m,k+1}(y_k)g_{m,k+1}(x_k)$.

Since the turn $\{a,b\}$ occurs in $g_{m,k}(c)$, then one of the following happens:

\begin{enumerate}
\item the turn $\{a,b\}$ occurs in $g_{m,k+1}(y_k)$; 
\item the turn $\{a,b\}$ occurs in $g_{m,k+1}(x_k)$; 
\item $\{a,b\}=Dg_{m,k+1}(\{y_k^{-1}, x_k\})$.
\end{enumerate}

If (3) happens, then by using $j=k$ we see that \[\{a,b\}\in   \{D(g_{m,j+1})(\{y_j^{-1}, x_j\})\} _{j=k+1}^m  \subseteq \{D(g_{m,j+1})(\{y_j^{-1}, x_j\})\}_{j=k}^m,\] as required. If (1) or (2) happens, we have $\{a,b\}\in \mathcal T (g_{m,k+1})$ and  by the inductive hypothesis applied to $g_{m,k+1}$ it follows that
\[
\{a,b\}\in \mathcal T(g_{m,k+1})\subseteq \{D(g_{m,j+1})(\{y_j^{-1}, x_j\})\} _{j=k+1}^m  \subseteq \{D(g_{m,j+1})(\{y_j^{-1}, x_j\})\}_{j=k}^m.
\]

Thus in all cases we have $\{a,b\}\in \{D(g_{m,j+1})(\{y_j^{-1}, x_j\})\}_{j=k}^m$. Since $\{a,b\}\in \mathcal{T}(g_{m,k})$ was arbitrary, it follows that \[
\mathcal T(g_{m,k+1})\subseteq \{D(g_{m,j+1})(\{y_j^{-1}, x_j\}) \mid j=k,\dots, m\}.
\]
This completes the proof of the inductive step.
\end{proof}

\begin{rm}
One can also show that the inclusion in the statement of Lemma~\ref{L:T} is actually an equality, but we will not need this fact here.
\end{rm}

\begin{lem}\label{L:T'}
For any $1\le m\le n$ we have
\[
\mathcal T(g_m)=\mathcal T(\theta_{m})\cup D\theta_{m}(\mathcal T(g_{m-1})),
\]
where for $m=1$ we interpret $g_{m-1}=g_0$ as the identity map of $R_r$.
\end{lem}
\begin{proof}

We argue by induction on $m$.
For $m=1$ we have $g_1=g_{\theta_1}$ and  $\mathcal T(g_{0})=\emptyset$. The conclusion of the lemma clearly holds in this case.

Suppose now that $m\ge 2$ and that the statement of the lemma has been proved for $g_{m-1}$.

If $e$ is an edge of $R_r$ and $g_{m-1}(e)=e_1\dots e_q$ then $\mathcal T(g_{m-1})=\{ \{ \overline{e_1},e_2\}, \dots, \{\overline{e_{q-1}},e_q\}\}$ and
$g_{m}(e)=g_{\theta_{m}}(g_{m-1}(e))=g_{\theta_{m}}(e_1)\dots g_{\theta_{m}}(e_q)$.

Therefore every turn in $\mathcal T(g_{m})$ arises either as a turn contained in the image of an edge under $g_{\theta_{m}}$ (so that it belongs to $\mathcal T(\theta_{m})$) or as the image under $D\theta_{m}=Dg_{\theta_{m}}$ of a turn in $\mathcal T(g_{m-1})$. The map $g_{m-1}:R_r\to R_r$ is a homotopy equivalence and hence is surjective. Thus every element of $\mathcal T(\theta_{m})$ will in fact occur in $\mathcal T(g_{m})$.

Therefore $\mathcal T(g_{m})=\mathcal T(\theta_{m})\cup D\theta_{m}(\mathcal T(g_{m-1}))$, as claimed.
\end{proof}

Recall that the limited Whitehead graph (see Definition~\ref{d:W}) of a graph-map $f:R_r\to R_r$ is a graph $Wh_L(f)$ with the vertex set
$Lk(v)=A^{\pm 1}$ and a topological edge joining vertices $d$ and $d'$
whenever $\{d,d'\}\in \mathcal T(f)$. In particular, if the map $f$ is
regular, then $Wh_L(f)$ has no loop-edges.

\begin{notation}\label{n:Upsilon}
Denote by $\Delta_r$ the graph with the vertex set $A^{\pm 1}$ where for every unordered pair $a,b$ of (possibly equal) elements of $A^{\pm 1}$ there is a topological edge with endpoints $a,b$. Thus $\Delta_r$ is the complete graph on the vertex set $A^{\pm 1}$ together with a loop-edge at each vertex.

If $\theta=[x\mapsto yx]$ is an standard Nielsen automorphism, we can extend $D\theta$ to a graph-map $\hat D\theta: \Delta_r\to\Delta_r$ defined as follows. For each vertex $a$ of $\Delta_r$, put $\hat D\theta(a):=D\theta(a)$. For an edge $e$ of $\Delta_r$ joining vertices $a,b$ the map $\hat D\theta$ sends $e$ to the edge joining the vertices $D\theta(a)$ and $D\theta(b)$ in $\Delta_r$. 

For $r\ge 3$  denote by $\Upsilon_r$ the complete graph on $2r-1$
vertices together with an edge joining a vertex of that graph to one
new vertex. (So that $\Upsilon_r$ is a connected graph with
$2r$ vertices and no loop-edges).

Also, for $r\ge 3$ and $x,y\in A^{\pm 1}$ such that $x\ne y^{\pm 1}$, denote by $\Upsilon_r[x,y]$ the complete graph
on vertices $A^{\pm 1}\setminus \{x\}$ together with a single edge joining $x$
with the vertex $y^{-1}$. 

Thus $\Upsilon_r[x,y]$ is a
connected graph with $2r$ vertices and with no loop-edges, and
$\Upsilon_r[x,y]$ is a subgraph of $\Delta_r$. 
\end{notation}

\begin{cor}\label{C:Wh}
The following hold:
\begin{enumerate}
\item For any $2\le m\le n$, the edge-set of the graph $Wh_L(g_m)$ consists of the edge joining the vertices $y_m^{-1}$ and $x_m$ and of the edges of the graph $\hat D\theta_m\left(Wh_L(g_{m-1},v)\right)$.
\item For $2\le m\le n$, if the graph $Wh_L(g_{m-1})$ is connected, then the graph $Wh_L(g_m,v)$ is also connected.
%\item If for some $2\le m\le n$ the graph $Wh_L(g_{m-1})$ is
%connected, then the graph $Wh_L(g_n)$ is also connected.
\item If for some $1\le i\le n$ the graph $Wh_L(g_{i})$ is isomorphic
  to $\Upsilon_r$ as an unlabelled graph, then
  $Wh_L(g_{i})=\Upsilon_r[x_i,y_i]$. (Recall that
  $\theta_i=[x_i\mapsto y_ix_i]$).
\item For $2\le m\le n$, if the graph $Wh_L(g_{m-1})$ is equal to $\Upsilon_r[x_{m-1},y_{m-1}]$, then the graph $Wh_L(g_m,v)$ is equal to $\Upsilon_r[x_m,y_m]$.
\end{enumerate}
\end{cor}

\begin{proof}

Part (1)  follows directly from the definitions and Lemma~\ref{L:T'}.

For (2), assume that the graph $Wh_L(g_{m-1})$ is connected. Then every one of the $2r$ elements of $A^{\pm 1}$ occurs as as endpoint of an edge of $Wh_L(g_{m-1})$.
Since the image of the set of directions $A^{\pm 1}$ under $D\theta_m$ is $A^{\pm 1}-\{x_m\}$, it follows that the connected graph $\hat D\theta_m(Wh_L(g_{m-1}))$ has as its vertex set the set $A^{\pm 1}-\{x_m\}$.  By (1), we have that the edge-set of $Wh_L(g_m)$  consists of the edges of $\hat D\theta_m(Wh_L(g_{m-1}))$ and of the edge $\tilde e$ joining $y_m^{-1}$ and $x_m$. Since $y_m^{-1}\ne x_m$, it follows that $y_m^{-1}$ is a vertex of the connected graph $\hat D\theta_m(Wh_L(g_{m-1}))$. Thus the edge $\tilde e$ joins the vertex $x_m$ to a vertex of the connected graph $\hat D\theta_m(Wh_L(g_{m-1}))$ whose vertex set is $A^{\pm 1}-\{x_m\}$.  Therefore the graph $Wh_L(g_m)$ is connected, as claimed. Thus (2) is verified. 

For part (3), by Lemma~\ref{L:T'} we have 
\[\mathcal T(g_i)=\mathcal
T(\theta_{i})\cup D\theta_{i}(\mathcal T(g_{i-1})).
\] Note that we have $i\ge 2$ since $Wh_L(g_{1})=Wh_L(\theta_{1})$ is not graph-isomorphic to $\Upsilon_r$. The direction $x_i$ does not belong to the image of the map $D\theta_{i}$. Therefore, from the above formula for $T(\theta_{i})$ and since $\mathcal
T(\theta_{i})=\left\{\{y_i^{-1},x_i\}\right\}$, the only edge incident to vertex $x_i$ in $Wh_L(g_i)$ is the edge joining $x_i$ and $y_i^{-1}$. The assumption that $Wh_L(g_{i})$ is isomorphic
  to $\Upsilon_r$ as an unlabelled graph now implies that
  $Wh_L(g_{i})=\Upsilon_r[x_i,y_i]$. Thus, part (3) is verified.

The proof of part (4) is similar to the proof of part (2), although it requires slightly more detailed analysis. Since the pair $(\theta_{m-1}=[x_{m-1}\mapsto y_{m-1}x_{m-1}], \theta_m=[x_m\mapsto y_mx_m])$ is admissible, we have either $x_m=x_{m-1}$, $y_m\ne y_{m-1}^{-1}$, or else $y_m=x_{m-1}, x_m\ne y_{m-1}^{-1}$. 

We assume that $y_m=x_{m-1}, x_m\ne y_{m-1}^{-1}$, as the other case is similar. Note that $x_m\ne x_{m-1}$ since otherwise $\theta_m$ would not be an automorphism of $F_r$.

By part (1) we know that $Wh_L(g_m)$ consists of the edge
joining the vertices $y_m^{-1}$ and $x_m$ and of the edges of the
graph $\hat D\theta_m Wh_L(g_{m-1})$. Thus we need to
show that $\hat D\theta_m Wh_L(g_{m-1})$ is the complete
graph on the vertex set $A^{\pm 1}\setminus \{x_m\}$. 

Since $D\theta_m(x)=x$ for every $x\in A^{\pm 1}, x\ne x_m$, the map $\hat D\theta_m$ fixes all vertices of $Wh_L(g_{m-1})$ different from $x_m$ and it fixes all
edges of $Wh_L(g_{m-1})$ that are incident to neither $x_{m-1}=y_m$
nor to $x_m$. Every edge of $Wh_L(g_{m-1})$ joining $x_m$ to some
$x\in A^{\pm 1}\setminus\{x_{m-1},x_m\}$ is mapped by $\hat
D\theta_m$ to an edge joining $y_m=x_{m-1}$ to $x$. Since
$Wh_L(g_{m-1})$ is equal to $\Upsilon_r[x_{m-1},y_{m-1}]$, it follows
that $\hat D\theta_m Wh_L(g_{m-1}$ contains all the edges between
distinct elements of $A^{\pm 1}\setminus \{x_m\}$. By Lemma~\ref{L:T'}
it follows that $Wh_L(g_m,v)=\Upsilon_r[x_m,y_m]$, as required.

\end{proof}

\subsection{Admissible sequences and train track maps}

\begin{lem}\label{L:reg}
Let $1\le m\le n$. Then:
\begin{enumerate}
\item We have $Dg_m(A^{\pm 1})=A^{\pm 1}-\{x_m\}$.
\item The map $g_m:R_r\to R_r$ is a regular graph-map (that is, the image of every edge is a reduced edge-path). 
\end{enumerate}
\end{lem}

\begin{proof}

Recall that $\theta_i=[x_i\mapsto y_ix_i]$.

We establish (1) by induction on $m$. For $m=1$, (1) follows from Lemma~\ref{L:basic}.
Thus assume that $2\le m\le n$ and that (1) has been established for $g_{m-1}$.

Since $Dg_m(A^{\pm 1})$ is contained in $D\theta_m(A^{\pm 1})=A^{\pm 1}-\{x_m\}$ and since the set  $A^{\pm 1}-\{x_m\}$ has cardinality $2r-1$, to establish (1) it suffices to show that the set $Dg_m(A^{\pm 1})$ has cardinality $2r-1$.

By the inductive hypothesis we have $Dg_{m-1}(A^{\pm 1})=A^{\pm 1}-\{x_{m-1}\}$.

By Lemma~\ref{L:basic},  the restriction of $D\theta_{m}$ to $A^{\pm 1} - \{x_{m-1} x_{m}\}$ is the identity. Recall that since $(\theta_{m-1},\theta_m)$ is an admissible pair, either $x_{m-1}= y_{m}$ or $x_{m-1} = x_{m}$. If $x_{m-1} = y_{m}$, then $y_{m} \notin A^{\pm 1} - \{x_{m-1}, x_{m}\}$, and specifically $y_{m}$ is not in the image of $Dg_{m-1}$. So $D\theta_{m}$ acts as a bijection of the image $A^{\pm 1} - \{x_{m-1}\}$ of $Dg_{m-1}$ onto $A^{\pm 1} - \{x_{m}\}$ by sending $x_{m}$ to $x_{m-1} = y_{m}$ and fixing all other directions. If $x_{m-1} = x_{m}$, then $D\theta_{m}$ acts as the identity on $A^{\pm 1} - \{x_{m-1}\}$. So $D\theta_{m}$ acts as the identity on the image of $Dg_{m-1}$. Thus, in either case, the image of $Dg_m$ also has precisely $2r-1$ directions in its image.
Hence $Dg_m(A^{\pm 1})=A^{\pm 1}-\{x_m\}$, as required.
This completes the inductive step, so that (1) is verified.

We prove (2) also by induction on $m$.
For $m=1$ the statement is obvious. Thus we assume that  $2\le m\le n$ and that (2) has been established for all admissible compositions of $\le m-1$ standard Nielsen automorphisms.

To show that $g_m$ is regular we need to verify that $\mathcal T(g_m)$ contains no degenerate turns.
Let $\{a',b'\}$ be a turn in $\mathcal T(g_m)$, where $a',b'\in A^{\pm 1}$. 
By Lemma~\ref{L:T'} we have
\[
\mathcal T(g_m)=\mathcal T(\theta_{m})\cup D\theta_{m}(\mathcal T(g_{m-1})).
\]
The map $g_{\theta_m}$ is regular by definition so that $\mathcal T(\theta_{m})$ contains no degenerate turns. Thus, if $\{a',b'\}\in \mathcal T(\theta_{m})$ then $\{a',b'\}$ is a nondegenerate turn, as required.

Suppose now that  $\{a',b'\}\in D\theta_{m}(\mathcal T(g_{m-1}))$, so that $\{a',b'\}=D\theta_m(\{a,b\})$ for some turn $\{a,b\}\in \mathcal T(g_{m-1})$.
By the inductive hypothesis applied to $g_{m-1}$, the map $g_{m-1}$ is regular and hence the turn $\{a,b\}$ is nondegenerate. 
The only nondegenerate illegal turn for $g_{\theta_m}$ is $\{x_m,y_m\}$. Thus, to conclude that $D\theta_m(\{a,b\})$ is a nondegenerate turn, it suffices to establish:

{\bf Claim.} We have $\{a,b\}\ne \{x_m,y_m\}$.

By Lemma~\ref{L:T} we have
\[
\{a,b\}\in \mathcal T(g_{m-1})\subseteq \{D(g_{m-1,j+1})(\{y_j^{-1}, x_j\})\}_{j=1}^m. 
\]

Hence 
\[
\{a,b\}=D(g_{m-1,j+1})(\{y_j^{-1}, x_j\}) \tag{$\ast$}
\]
for some $1\le j\le m-1$.

Consider first the case that $(\ast)$ happens for $j=m-1$. Thus $\{a,b\}= D(g_{m,m-1})(\{y_{m-1}^{-1}, x_{m-1}\})$. Recall that by convention $g_{m,m-1}=Id_{R_r}$ and hence $\{a,b\}=\{y_{m-1}^{-1}, x_{m-1}\}$.
Assume for the sake of contradiction that $\{a,b\}= \{x_m,y_m\}$. Then $\{a,b\}= \{x_m,y_m\}=\{y_{m-1}^{-1},x_{m-1}\}$. 
Since the pair $(\theta_{m-1},\theta_m)$ is admissible, we have either $x_{m-1}=x_m$ and $y_m\ne y_{m-1}^{-1}$, or $x_{m-1}=y_m$ and $x_m\ne y_{m-1}^{-1}$, each of which yields a contradiction.

If $x_{m-1}=x_m$ and $y_m\ne y_{m-1}^{-1}$ then $\{x_{m-1},y_{m-1}^{-1}\}\ne \{x_m,y_m\}$, yielding a contradiction.
If $x_{m-1}=y_m$ and $x_m\ne y_{m-1}^{-1}$ then $\{x_{m-1},y_{m-1}^{-1}\}\ne \{y_m,x_m\}$, again yielding a contradiction.
Thus, for $j=m-1$ we get $\{a,b\}\ne  \{x_m,y_m\}$, as required.

Consider now the case where $(\ast)$ happens for $1\le j\le m-2$. Then $j+1\le m-1$ and $g_{m-1,j+1}=g_{m-1}\circ \dots \circ g_{j+1}$.
Since $g_{m-1,j+1}=g_{m-1}\circ \dots \circ g_{j+1}$, the image of the set of directions $A^{\pm 1}$ under $D(g_{j+1,m-1})$ is contained in $Dg_{m-1}(A^{\pm 1})=A^{\pm 1}-\{x_{m-1}\}$.

Therefore, neither of $a$ nor $b$ is equal to $x_{m-1}$. 
Since the pair $(\theta_{m-1},\theta_m)$ is admissible, we have either $x_{m-1}=x_m$ and $y_m\ne y_{m-1}^{-1}$ or $x_{m-1}=y_m$ and $x_m\ne y_{m-1}^{-1}$.

Suppose, for the sake of contradiction, that in fact  $\{a,b\}=\{x_m,y_m\}$, so that either $a=x_m$ and $b=y_m$ or $a=y_m$ and $b=x_m$. Assume that $a=x_m$ and $b=y_m$, as the other case is symmetric.
If $x_{m-1}=x_m$, then $a=x_m=x_{m-1}$, contradicting the fact that neither $a$ nor $b$ equals $x_{m-1}$. 

If $x_{m-1}=y_m$, then $b=y_m=x_{m-1}$, again contradicting the fact that neither $a$ nor $b$ equals $x_{m-1}$. 

Thus $\{a,b\}\ne \{x_m,y_m\}$, and the Claim is verified.

As noted above, this implies that the turn $\{a',b'\}=D\theta_m(\{a,b\})$ is nondegenerate.

We have shown that every turn in $\mathcal T(g_m)$ is nondegenerate, and hence $g_m$ is a regular graph map. This completes the inductive step, so that (2) is established.

\end{proof}

Lemma~\ref{L:reg}  has the following important consequence:

\begin{thm}\label{T:tt}
Let $n\ge 1$ and let $\mathfrak t=\theta_1,\dots,\theta_n$ be a cyclically admissible sequence of standard Nielsen automorphisms $\theta_i=[x_i\mapsto y_ix_i]$ of $F_r$. 
Then $g_\mathfrak t=g_{\theta_n}\circ \dots \circ g_{\theta_1}:R_r\to R_r$ is a train track map with exactly one nondegenerate illegal turn, namely the turn $\{x_1, y_1\}$. 
\end{thm}

\begin{proof}
Since $\theta_n\circ\dots\dots \circ \theta_1$ is a cyclically admissible composition, it follows that, for each $k\ge 1$, the composition  $(\theta_n\circ\dots\dots \circ \theta_1)^k$ is admissible.
Hence, by Lemma~\ref{L:reg} applied to $(\theta_n\circ\dots\dots \circ \theta_1)^k$, it follows that for each $k\ge 1$ the map $g_\mathfrak t^k:R_r\to R_r$ is regular.
Thus for every edge $e\in ER_r$ the path $g_\mathfrak t^k(e)$ is reduced. Hence, $g_\mathfrak t$ is a train track map, as required.

We have $Dg_\mathfrak t=Dg_{\theta_n}\circ \dots \circ D_{g_{\theta_1}}$. Since $Dg_{\theta_1}(x_1)=Dg_{\theta_1}(y_1)=y_1$, it follows that  $Dg_\mathfrak t(x_1)=Dg_\mathfrak t(y_1)$. Hence the turn $\{x_1,y_1\}$ is illegal for $g_\mathfrak t$.

Suppose that $g_\mathfrak t$ had $\ge 2$ nondegenerate illegal turns. It would follow that for some $k\ge 1$ the image of the set of directions $A^{\pm 1}$ under $Dg_\mathfrak t^k$ has $\le 2r-2$ elements. However, by part (1) of Lemma~\ref{L:reg} applied to $g_\mathfrak t^k$ we know that the set $Dg_\mathfrak t^k(A^{\pm 1})$ has exactly $2r-1$ elements, yielding a contradiction.
Thus, the train track map $g:R_r\to R_r$  has exactly one illegal turn, namely $\{x_1, y_1\}$.

\end{proof}

\begin{lem}\label{L:complete}
Let $n\ge 1$ and let $\mathfrak t=\theta_1,\dots,\theta_n$ be a
cyclically admissible sequence of standard Nielsen automorphisms
$\theta_i=[x_i\mapsto y_ix_i]$ of $F_r$. Suppose that $Wh_L(g_\mathfrak
t)$ is isomorphic, as an unlabelled graph, to $\Upsilon_r$. Then: 
\begin{enumerate}
\item We have $Wh_L(g_\mathfrak
t)=Wh(g_\mathfrak
t)=\Upsilon[x_n,y_n]$.
\item For every integer $p\ge 1$ we have $Wh_L(g_\mathfrak
t^p)=Wh(g_\mathfrak
t^p)=\Upsilon[x_n,y_n]$.
\end{enumerate}
\end{lem}
\begin{proof}
The fact that $Wh_L(g_\mathfrak
t)=\Upsilon[x_n,y_n]$ follows from part (4) of Lemma~\ref{C:Wh}. There
exists some $k\ge 1$ such that $Wh(g_\mathfrak
t)=Wh_L(g_\mathfrak
t^k)$. We have $g_\mathfrak
t^k=g_{\mathfrak t^k}$, and $\mathfrak t^k$ is a cyclically admissible
sequence ending in $\theta_n$. Iteratively applying part (4) of
Lemma~\ref{C:Wh}, we see that $Wh_L(g_{\mathfrak
t^k})=\Upsilon[x_n,y_n]$. Thus $Wh_L(g_\mathfrak
t)=Wh(g_\mathfrak
t)=\Upsilon[x_n,y_n]$, as claimed, and part (1) is verified.

If $p\ge 1$ is an integer, then $\mathfrak t^p$ is a cyclically admissible
sequence with initial segment $\mathfrak t$, ending with $\theta_n$ and having $Wh_L(g_{\mathfrak t})=\Upsilon[x_n,y_n]$. Hence, by part (4) of  Lemma~\ref{C:Wh}, we get $Wh_L(g_\mathfrak
t^p)=\Upsilon[x_n,y_n]$.  Now part (1) of the present lemma implies that $Wh_L(g_\mathfrak
t^p)=Wh(g_\mathfrak
t^p)=\Upsilon[x_n,y_n]$.
\end{proof}

\section{Periodic Nielsen path prevention}

Recall that $r\ge 2$ and that $F_r=F(A)$ where $A=\{a_1,\dots, a_r\}$ is a fixed free basis of $F_r$.

\begin{df}\label{D:block}
Let $r \geq 3$. A \emph{periodic indivisible Nielsen path prevention
  sequence} or \emph{pINP prevention sequence} is an admissible
sequence $\mathfrak p = \theta_1,\dots,\theta_k$ of standard Nielsen
automorphisms of $F_r$ such that, whenever we have a cyclically
admissible sequence $\mathfrak t = \theta_1',\dots,\theta_n'$ such
that $n \geq k$ and $\theta_i = \theta_i'$ for each $1 \leq i \leq k$
and such that  $g_\mathfrak t=g_{\theta_n'}\circ \dots \circ
g_{\theta_1'}:R_r\to R_r$ is a n expanding irreducible train track map, then $g_\mathfrak t$ has no pINP's. 
\end{df}

\begin{lem}\label{L:NPKilling} Let $r\ge 4$ and let $\{x,w,y,z\} \subset A^{\pm 1}$ be a subset of four distinct elements no two of which are inverses of each other. Let $\mathfrak p = \theta_1,\dots,\theta_6$, where $\theta_1=[z \mapsto xz], \theta_2=[w \mapsto zw], \theta_3=[y \mapsto y \bar w], \theta_4=[y \mapsto y \bar x], \theta_5=[y \mapsto y \bar w], \theta_6=[y \mapsto y \bar x]$. Then $\mathfrak p$ is a pINP prevention sequence. 
\end{lem}

\begin{proof}
Let $\mathfrak t$ be a cyclically admissible sequence starting with $\mathfrak p = \theta_1,\dots,\theta_6$. 

We need to show that $g_{\mathfrak t}$ has no pINPs. This conclusion
follows from~\cite[Lemma 5.5]{pI}. The only difference between the
terminology used here and that used in ~\cite{pI}, is that in the
terminology of ~\cite{pI}  a pINP prevention sequence  is only
required to prevent pINPs when an admissible sequence starting with  a
pINP prevention sequence composes to give a rotationless expanding irreducible train track map.  Lemma~5.5 of \cite{pI} shows that $\mathfrak p = \theta_1,\dots,\theta_6$ is a pINP prevention sequence in this sense.

However, if $\mathfrak t$ is a cyclically admissible sequence beginning with $\mathfrak p = \theta_1,\dots,\theta_6$, then every power of $\mathfrak t$ is an admissible sequence and there exist $k\ge 1$ such that $g_{\mathfrak t^k}$ is rotationless. Since $\mathfrak t^k$ starts with $\mathfrak p$, Lemma 5.5  of \cite{pI} applies to $\mathfrak t^k$ and implies that $g_{\mathfrak t^k}=(g_{\mathfrak t})^k$  has no pINPs. However, by definition, a path in $R_r$ is a pINP for $g_\mathfrak t$ if and only if this path is a pINP for $(g_{\mathfrak t})^k$. Hence $g_\mathfrak t$ has no pINPs as required.
\end{proof}

\begin{rk}
The idea of the proof of \cite[Lemma 5.5]{pI} is as follows. Suppose that $\mathfrak t = \theta_1, \dots, \theta_n$ is a cyclically admissible sequence starting with $\mathfrak p = \theta_1,\dots,\theta_6$ (where $\mathfrak p$ is as in the statement of Lemma~\ref{L:NPKilling}) such that $g_\mathfrak t$ is a rotationless expanding train track map. By replacing $\mathfrak t$ by its power, we may assume that, if any pINPs exist for $g_\mathfrak t$, then they have period 1, and so they are in fact INPs. 
We then need to show that $g_\mathfrak t$ does not in fact have any INPs. 
Suppose, on the contrary, that $g_\mathfrak t$ has an INP.  Then, by Proposition~\ref{P:PINP}, this INP has the form $\alpha=\rho_1^{-1}\rho_2$, where $\rho_1,\rho_2$ are legal paths with the common initial vertex such that the turn between  $\rho_1,\rho_2$ is a nondegenerate illegal turn for $g_\mathfrak t$.  By Theorem~\ref{T:tt} we know that $g_\mathfrak t$ has only one illegal turn, namely the turn $\{x,z\}$ (since $\mathfrak t$ starts with $\theta_1=[z \mapsto xz]$). Thus $\rho_1$ starts with an initial segment of $x$ and $\rho_2$ starts with an initial segment of $z$ (or the other way around). Note that $t(\rho_1),t(\rho_2)$ are fixed points of $g_\mathfrak t$ but they need not be vertices. The fact that $\alpha=\rho_1^{-1}\rho_2$ is an INP for $g_\mathfrak t$ means that for every $p \ge 1$ the path $g_\mathfrak t^p(\rho_1^{-1}) g_\mathfrak t^p(\rho_2)$ reduces to $\rho_1^{-1}\rho_2$. For this reason, for each $0 \leq k \leq n$, the tightened path $g_{k,1}(\alpha)$ cannot be taken by $g_{n, k+1}(\alpha)$ to a legal path for $g_\mathfrak t$.

Roughly speaking, the proof of Lemma~5.5 in \cite{pI} proceeds by showing, using admissibility of $\mathfrak t$ and the specific combinatorics of $\mathfrak p$, that, in fact, for some $k\ge 1$, the path $g_{\theta_1, \dots, \theta_k}(\alpha)$ reduces to a nondegenerate path that is taken by $g_{n, k+1}$ to a legal path for $g_\mathfrak t$.  This contradicts the fact that $\alpha=\rho_1^{-1}\rho_2$ is an INP for $g_\mathfrak t$.
This argument is illustrated in more detail in Lemma~\ref{L:NPKilling3} below.
\end{rk}

In \cite{pII}, the case of rank $3$ is handled separately. While \cite{pII} does produce a pINP prevention sequence for $r=3$, this fact is not stated there explicitly and therefore we provide a sketch of the proof here, following the procedure of \cite[Section 5]{pII}.

\begin{lem}\label{L:NPKilling3} Let $r\ge 3$ and let $\{a,b,c\} \subset A^{\pm 1}$ be a subset of three distinct elements, no two of which are inverses of each other. Let
$$
\theta_1 = [a \mapsto ca],
\theta_2 = [\bar{b} \mapsto a \bar{b}],
\theta_3 = [\bar{b} \mapsto \bar{c} \bar{b}],
\theta_4 = [a \mapsto \bar{b} a],$$
$$\theta_5 = [a \mapsto ca], 
\theta_6 = [a \mapsto ba],
\theta_7 = [a \mapsto ca],
\theta_8 = [a \mapsto ca].
$$ 
Then $\mathfrak p=\theta_1,\dots,\theta_8$ is a pINP prevention sequence.
\end{lem}

\begin{proof}[Sketch of proof] The complete verification process is rather long, and so we just show the longer of the two cases. The other case is similar and also proceeds as in \cite[Section 5]{pII}.

Note that a pINP for $g_{\mathfrak t}$, where $\mathfrak t$ is a cyclically admissible sequence beginning with $\mathfrak p$, is an INP for for $g_{\mathfrak t^m}=(g_\mathfrak t)^m$ for some $m\ge 1$.

Thus it suffices to show that if $\mathfrak t$ is a cyclically admissible sequence beginning with $\mathfrak p$ such that $g_\mathfrak t$ is an expanding irreducible train track map, then $g_\mathfrak t$ has no INPs. Suppose, on the contrary, that $\mathfrak t=\theta_1', \dots,\theta_n'$ is a cyclically admissible sequence beginning with $\mathfrak p$ such that $g_\mathfrak t$ is an expanding irreducible train track map and such that $g_\mathfrak t$ possesses an INP $\rho$. 

Then $\rho$ would have to contain the illegal turn $\{a,c\}$ for $g_{\mathfrak t}$ and (possibly after reversing its orientation) could be written as $\rho=\rho_1^{-1}\rho_2$ where $\rho_1$, $\rho_2$ are nondegenerate legal paths, with the initial direction of $\rho_1$ being $c$ and the initial direction of $\rho_2$ being $a$. Note that the terminal points of one or both of $\rho_1$, $\rho_2$ may be contained in the interiors of edges of the 3-rose $R_3$.  However, there exist legal edge-paths $\rho_1'$, $\rho_2'$ such that $\rho_1'$ begins with $\rho_1$, and $\rho_2'$ begins with $\rho_2$. Thus
$\rho_1'=c\dots$ and $\rho_2'=a\dots$ are $g_\mathfrak t$-legal edge-paths, and $g_\mathfrak t((\rho_1')^{-1}\rho_2')$ tightens to a path $(\rho_1'')^{-1}\rho_2''$, where $\rho_1,\rho_2$ are legal edge-paths, with $\rho_1''$ starting with $a$ and $\rho_2''$ starting with $c$. 

In working with $g_\mathfrak t$ we will use the notations $g_{k,m}$ introduced in Convention~\ref{conv:main}.

We now make the following crucial observation. 

{\bf Claim.} Suppose that for some $1\le k<n$ the tightened form of the path $g_{k,1}((\rho_1')^{-1}\rho_2')$ is $\alpha^{-1}\beta$ where $\alpha$ is a terminal segment of $g_{k,1}(\rho_1')$ and where $\beta$ is a terminal segment of $g_{k,1}(\rho_2')$.  Then both $\alpha$ and $\beta$ are nontrivial (i.e. containing at least one edge each) edge-paths and the turn $\tau$ between them satisfies $\tau=\{x_{k+1},y_{k+1}\}$. 

First note that if one of $\alpha$, $\beta$ is trivial, then $g_{n,k+1}(\alpha^{-1}\beta)$ is contained in either $g_\mathfrak t\left((\rho_1')^{-1}\right)$ or $g_\mathfrak t(\rho_2')$, and hence is a $g_\mathfrak t$-legal path, contrary to the assumption that the tightened form of $g_\mathfrak t((\rho_1')^{-1}\rho_2')$ is the path $(\rho_1'')^{-1}\rho_2''$ containing a $g_\mathfrak t$-illegal turn $\{a,c\}$. Thus $\alpha$ and $\beta$ are nontrivial edge-paths. Since, by definition, the path $\alpha^{-1}\beta$ is tight, the turn $\tau$ between $\alpha$ and $\beta$ is nondegenerate. 

Suppose that $\tau\ne \{x_{k+1},y_{k+1}\}$. We know, by Lemma~\ref{L:reg},  that $Dg_{n,k+1}$ identifies the directions $x_{k+1},y_{k+1}$ and that $Dg_{n,k+1}(A^{\pm 1})=A^{\pm 1}\setminus \{x_n\}$. Thus, if $\tau\ne \{x_{k+1},y_{k+1}\}$ then the $Dg_{n,k+1}$-images of the directions comprising $\tau$ are distinct, so that $Dg_{n,k+1}(\tau)$ is a nondegenerate turn. Moreover, since $x_n\not\in Dg_{n,k+1}(A^{\pm 1})$, and since the pair $(\theta_n,\theta_1)$ is admissible (so that $x_1=x_n$ or $y_1=x_n$), the turn $Dg_{n,k+1}(\tau)$ is not equal to $\{x_1,y_1\}$. This means that  $g_{n,k+1}(\alpha^{-1}\beta)$ is a tight path which is legal for $g_\mathfrak t$.  This contradicts the fact that $g_{n,k+1}(\alpha^{-1}\beta)$ must reduce to $(\rho_1'')^{-1}\rho_2''$. Thus, the claim is verified. 

Now, $g_{\theta_1}(c)=c$ and $g_{\theta_1}(a)=ca$. So $\rho_1'$ has to be of the form $\rho_1'=c e_2...$ for some additional edge $e_2$. Also, since the illegal turn for $g_{\theta_2}$ is $\{a,\bar{b}\}$ and $a$ is not in the image of $Dg_{\theta_1}$, we have that $Dg_{\theta_1}(e_2)=\bar{b}$. So $e_2=\bar{b}$.

Since $g_{2,1}(c\bar{b})=ca\bar{b}$ and $g_{2,1}(a)=ca$, we know that $\rho_2'$ has to be of the form $\rho_2'=ae_2'$  for some  additional edge $e_2'$ with $Dg_{2,1}(e_2')=\bar c$ (since $\{\bar{b},\bar{c}\}$ is the illegal turn for $g_{\theta_3}$ and $\bar{b}$ is not in the image of $Dg_{2,1}$). The only option is $e_2'=\bar{c}$. Since $g_{3,1}(a\bar{c})=ca\bar{c}$ and $g_{3,1}(c\bar b)=ca\bar{c}\bar{b}$, we must have $\rho_2'=a e_2'e_3'..$ for an additional edge $e_3'$  satisfying that $Dg_{3,1}(e_3')=a$ (since the illegal turn for $g_{\theta_4}$ is $\{a,\bar{b}\}$ and $\bar{b}$ is not in the image of $Dg_{3,1}$). The only option is $e_3'=\bar{b}$.

Since $g_{4,1}(a\bar{c}\bar{b})=c\bar{b}a\bar{c}\bar{b}a\bar{c}\bar{b}$ and $g_{4,1}(c\bar b)=c\bar{b}a\bar{c}\bar{b}$, we must have $\rho_1'=c e_2 e_3...$ for an additional edge $e_3$  satisfying  $Dg_{4,1}(e_3)=c$ (since the illegal turn for $g_{\theta_5}$ is $\{a,c\}$ and $a$ is not in the image of $Dg_{4,1}$). So either $e_3=a$ or $e_3=c$. We analyze here the case where $e_3=c$ and leave the case of $e_3=a$ to the reader.

Since $g_{5,1}(a\bar{c}\bar{b}) = c\bar{b}ca\bar{c}\bar{b}c a\bar{c}\bar{b}$ and $g_{5,1}(c\bar bc)=c\bar{b}ca\bar{c}\bar{b}c$, we must have $\rho_1'=c e_2 e_3 e_4...$ for an additional edge $e_4$  satisfying that $Dg_{5,1}(e_4)=b$ (since the illegal turn for $g_{\theta_6}$ is $\{a,b\}$ and $a$ is not in the image of $Dg_{5,1}$). So $e_4=b$.

We have $g_{6,1}(a\bar{c}\bar{b}) = c\bar{b}cba\bar{c}\bar{b}c ba\bar{c}\bar{b}$ and 
$g_{6,1}(c\bar bcb) =c\bar{b}cba\bar{c}\bar{b}cbc\bar{a}\bar{b}\bar{c}b$.
After cancellation, we are left with the turn $\{a,c\}$, which is illegal for $g_{\theta_7}$ and so we can proceed by applying $g_{\theta_7}$.
 
Since $g_{7,1}(a\bar{c}\bar{b}) = 
c\bar{b}cbca\bar{c}\bar{b}c bca\bar{c}\bar{b}$ and 
$g_{7,1}(c\bar bcb) =c\bar{b}cbca\bar{c}\bar{b}cbc\bar{a}\bar{c}\bar{b}\bar{c}b$,
cancellation ends with the turn $\{a,\bar{a}\}$. This is not the illegal turn for $g_{\theta_8}$. Therefore, by the claim above,  we could not have $\rho_1=c\bar bcb\dots$ and $\rho_2=a\bar{c}\bar{b}\dots$.

The remaining case, where $e_3=a$, yields a similar situation, and we conclude that $\mathfrak p$ is a pINP prevention sequence. 
 \qedhere
\end{proof}

Note that, by definition, any admissible sequence that starts with a pINP prevention sequence is also itself a pINP prevention sequence.

We will need the following important fact which is essentially a restatement of the main result of Pfaff~\cite{pII} (it can be noted that a power should have been taken of the map constructed for the main theorem of Pfaff~\cite{pII}, but that the result is otherwise correct): 

\begin{prop}\label{P:CP}
Let $r\ge 3$. Then there exists a cyclically admissible sequence \[\mathfrak s=\theta_1',\dots, \theta_q'\] of standard Nielsen automorphisms of $F_r$ such that for $\Psi=\theta_q'\circ \dots \circ \theta_1'\in \Aut(F_r)$ and for $g_\mathfrak s=g_{\theta_q'}\circ \dots \circ g_{\theta_1}':R_r\to R_r$ the following hold:

\begin{enumerate}
\item The map $g_\mathfrak s:R_r\to R_r$ is an expanding irreducible train track map with no pINPs.
\item We have $M(g_\mathfrak s)>0$.
\item We have $Wh(g_\mathfrak s)=Wh_L(g_\mathfrak
  s)=\Upsilon_r[x_q',y_q']$, where $\theta_q'=[x_q'\mapsto y_q'x_q']$ (and in particular $Wh_L(g_\mathfrak s)$ is
  connected). 
\item The sequence $\mathfrak s$ starts with a pINP prevention sequence $\mathfrak p$, provided by  Lemma~\ref{L:NPKilling} in the case $r\ge 4$ and provided by Lemma~\ref{L:NPKilling3} in the case $r=3$. In particular,  $\mathfrak s$ is a pINP prevention sequence.
\item The  element of $\Out(F_r)$ represented by  $g_\mathfrak s:R_r\to R_r$ is ageometric fully irreducible (and in particular, it is hyperbolic).
\end{enumerate}
\end{prop}

In fact, the cyclically admissible sequence $\mathfrak s$ constructed
in \cite{pI} satisfies (1), (3), (4), (5) above, has $M(g_\mathfrak
s)$ irreducible, and has $Wh(g_\mathfrak s)$ graph-isomorphic to
$\Upsilon_r$. We have $Wh(g_\mathfrak s)=Wh(g_\mathfrak s^k)$ for
every $k\ge 1$, and there is some $k\ge 1$ such that $Wh(g_\mathfrak
s)=Wh(g_\mathfrak s^k)=Wh_L(g_\mathfrak s^k)$. Since $\mathfrak s^k$
is a cyclically reduced admissible sequence, Lemma~\ref{L:complete}
implies that $Wh_L(g_\mathfrak s^k)=\Upsilon_r[x_q',y_q']$. Thus by replacing this $\mathfrak s$ with $\mathfrak s^k$ we obtain a cyclically admissible sequence satisfying (1)-(5) above (the power is for (2)).

We can now prove the main technical result of this paper:

\begin{thm}\label{T:tech}
Let $r\ge 3$ and let $\mathfrak s$ be provided by Proposition~\ref{P:CP}. 
Let $\mathfrak t=\theta_1,\dots, \theta_n$ be a cyclically admissible sequence of standard Nielsen automorphisms $\theta_i=[x_i\mapsto y_ix_i]$ of $F_r$, such that $\mathfrak s$ is an initial segment   of $\mathfrak t$. Then for $g_\mathfrak t:R_r\to R_r$ and for the  element $\phi\in\Out(F_r)$ represented by  $g_\mathfrak t:R_r\to R_r$ the following hold:
\begin{enumerate}
\item The map $g_\mathfrak t:R_r\to R_r$ is a train track map with exactly one nondegenerate illegal turn in $R_r$.
\item We have $M(g_\mathfrak t)>0$ (and hence $M(g_\mathfrak t)$ is irreducible).
\item We have $Wh_L(g_\mathfrak t)=Wh(g_\mathfrak t)=\Upsilon_r[x_n,y_n]$ (and, in particular, $Wh(g_\mathfrak t)$ is connected).
\item The map $g_\mathfrak t:R_r\to R_r$ has no pINP's. 
\item The element $\phi\in\Out(F_r)$  is ageometric fully irreducible.

\item The ideal Whitehead graph $\mathcal{IW}(\phi)$  is the complete graph on $2r-1$ vertices. The element $\phi\in\Out(F_r)$ has $i(\phi)=\frac{3}{2}-r$ and  index list $\{\frac{3}{2}-r\}$. 

\item The axis bundle for $\phi$ in $CV_r$ consists of a single axis. 
\end{enumerate}
\end{thm}

\begin{proof}
Part (1) follows from Theorem~\ref{T:tt}.
Part (2) follows from Proposition~\ref{P:CP} and Lemma~\ref{L:>0}.
Part (3) follows from Proposition~\ref{P:CP}  and Corollary~\ref{C:Wh}.
Part (4) holds since $\mathfrak s$ is a pINP prevention sequence.
Part (5) follows from Proposition~\ref{P:CP} .

For (6), note that since $g_\mathfrak t$ has exactly one nondegenerate illegal turn (namely the turn $\{x_1,y_1\}$), there are exactly $2r-1$ gates at the vertex $v$ of $R_r$: the gate $\{x_1,y_1\}$ and the gates $\{z\}$, where $z$ varies over $A^{\pm 1}-\{x_1,y_1\}$. Since every gate contains exactly one periodic direction, it follows that there are exactly $2r-1$ periodic directions at $v$.  Corollary~\ref{C:Wh} and Lemma~\ref{L:complete} imply that $Wh(g_\mathfrak t,v)=Wh_L(g_\mathfrak t,v)=\Upsilon_r[x_n,y_n]$ (recall that the definition of $\Upsilon_r[x,y]$ is given in Notation~\ref{n:Upsilon}).

Since the direction $x_n$ does not belong to the image of the derivative map $D\theta_n$, it follows that the direction $x_n$ is not in the image of $Dg_\mathfrak t$ and hence not in the image of $Dg_{\mathfrak t^k}$ for any $k\ge 1$. Thus $x_n$ is not a periodic direction for $g_\mathfrak t$. Each of the gates at $v$ contains exactly one periodic direction for $g_\mathfrak t$. Thus exactly one direction in the gate $\{x_1,y_1\}$ is periodic.

Since $\mathfrak t$ is cyclically admissible, the pair $(\theta_n,\theta_1)$ is admissible. Thus either $x_1=x_n$ or $y_1=x_n$.
If $x_1=x_n$, then, since $x_n$ is not a periodic direction, it follows that $y_1$ is a periodic direction. Since $y_1\in A^{\pm 1}-\{x_n\}$ and $Wh(g_\mathfrak t,v)=\Upsilon_r[x_n,y_n]$, Definition~\ref{d:IWG} implies that $\mathcal{SW}(g,v)$ is a complete graph on the $2r-1$ vertices $A^{\pm 1}-\{x_n\}$.
If $y_1=x_n$, then again, since $x_n$ is not a periodic direction, it follows that $x_1$ is a periodic direction. Since $x_1\in A^{\pm 1}-\{x_n\}$ and $Wh(g_\mathfrak t,v)=\Upsilon_r[x_n,y_n]$, it again follows that $\mathcal{SW}(g,v)$ is a complete graph on the $2r-1$ vertices $A^{\pm 1}-\{x_n\}$.
Thus we see that in either case $\mathcal{SW}(g,v)$ is a complete graph on the $2r-1$ vertices.
Since by (4) we know that $g_\mathfrak t$ has no pINPs, by Definition~\ref{d:IWG} it follows that $\mathcal{IW}(\phi)=\mathcal{SW}(g,v)$ is a complete graph on $2r-1$ vertices (and in particular is connected). Therefore, by Definition~\ref{d:RotationlessIndex},  we have $i(\phi)=1-\frac{2r-1}{2}=\frac{3}{2}-r$ and the index list for $\phi$ is $\{\frac{3}{2}-r\}$.  Thus (6) is verified.

Finally,  (7) follows from parts (1)-(6) by Theorem~\ref{t:AxisBundle}.
\end{proof}

\begin{rk}\label{R:permut}
Let  $\mathfrak t$ be any cyclically admissible sequence that contains $\mathfrak s$ as a sub-block (rather than necessarily starts with $\mathfrak s$). Then some cyclic permutation $\mathfrak t'$ of $\mathfrak t$ is a cyclically admissible sequence which begins with $\mathfrak s$.  Thus Theorem~\ref{T:tech}  applies to $\mathfrak t'$. The outer automorphism $\phi'\in \Out(F_r)$  represented by $g_{\mathfrak t'}$ is conjugate to $\phi$ in $\Out(F_r)$, and since, by Theorem~\ref{T:tech},  parts (5),(6),(7) hold for $\phi'$, they  also hold for $\phi$.  Moreover, $g_{\mathfrak t'}$ can be used as a topological representative for $\phi$, except that we need to change the marking on $R_r$ from the identity map to the map corresponding to the initial segment of $\mathfrak t$ that moved to the end to obtain $\mathfrak t'$ as a cyclic permutation of $\mathfrak t$. Thus $g_{\mathfrak t'}$, with a modified marking, is a topological representative of $\phi$, and conclusions (1)-(4) hold for $g_{\mathfrak t'}$. 

Regarding $\mathfrak t$ itself, in this case we do know, by Theorem~\ref{T:tt} and Lemma~\ref{L:>0}, that $g_\mathfrak t$ is a  train track map with exactly one nondegenerate illegal turn and with $M(g_\mathfrak t)>0$. Moreover, we also know, by Lemma~3.1 of \cite{pI} that  $g_\mathfrak t$ has no pINPs.  The fact that $g_\mathfrak t$ is a topological representative of a fully irreducible atoroidal element implies, by Proposition~\ref{P:K14}, that $Wh(g_\mathfrak t)$ is connected. However it is not clear if one can claim that $Wh(g_\mathfrak t)$ is graph-isomorphic to $\Upsilon_r$. 
\end{rk}

\section{Train track directed random walk}\label{sect:ttdrw}

Recall that we set for the free group $F_r=F(A)=F(a_1,\dots, a_r)$ (where $r\ge 2$) a distinguished free basis $A=\{a_1,\dots, a_r\}$.
Let $S$  be the set of all standard Nielsen automorphisms of $F_r$ (with respect to the basis $A$). 

Recall that each $\theta\in S$ has the form $\theta=[x\mapsto yx]$ where $x,y\in A^{\pm 1}$ are arbitrary elements such that $y\ne x^{\pm 1}$.  Hence, $\#(S)=2r(2r-2)=4r(r-1)$.

For $\theta\in S$, let $S_+(\theta)$ be the set of all $\theta'\in S$ such that the pair $(\theta,\theta')$ is admissible. 
Similarly, for $\theta\in S$, let $S_-(\theta)$ be the set of all $\theta'\in S$ such that the pair $(\theta',\theta)$ is admissible. 

\begin{lem}\label{lem:4r-6}
Let $r\ge 2$. Then for each $\theta\in S$ we have $\#(S_+(\theta))=\#(S_-(\theta))=4r-6$.
\end{lem}
\begin{proof}
Let $\theta=[x\mapsto yx]\in S$. We will show that $\#(S_+(\theta))=4r-6$. The argument that $\#(S_-(\theta))=4r-6$ is similar.

By definition, $\theta'=[x'\mapsto y'x']$ belongs to $S_+(\theta)$ if and only if the pair $(\theta,\theta')$ is admissible, that is, if and only if either $x = x'$ and $y' \neq y^{-1}$, or $x = y'$ and $x' \neq y^{-1}$.

We first count the number $n_1$ of $\theta'=[x'\mapsto y'x']\in S$ such that $x = x'$ and $y' \neq y^{-1}$.  The choice of $x'=x$ is uniquely determined by the condition $x=x'$. We can choose $y'$ to be any element of the set $A^{\pm 1}-\{x^{\pm 1}, y^{-1}\}$. Thus there are $2r-3$ choices of $y'$, and so $n_1=2r-3$.

We next count the number $n_2$ of $\theta'=[x'\mapsto y'x']\in S$ such that $x = y'$ and $x' \neq y^{-1}$. The choice of $y'$ is uniquely determined by the condition $y'=x$. We can then choose $x'$ to be an arbitrary element of $A^{\pm 1}-\{x^{\pm 1}, y^{-1}\}$.  Thus there are $2r-3$ choices for $x'$, so that $n_2=2r-3$.

Since $y\ne x^{\pm 1}$, the case where $x = x', y' \neq y^{-1}$ and the case where $x = y', x' \neq y^{-1}$ are mutually disjoint. Hence, $\#(S_+(\theta))=n_1+n_2=4r-6$, as claimed.
\end{proof}

\subsection{A train track directed Markov chain}

\begin{df}
Let $r\ge 3$. Consider the finite state Markov chain $\mathcal Y$ defined as follows.
The state set of $\mathcal Y$ is $S_r$. For any states $\theta,\theta'\in S_r$, the transition probability $P_\mathcal Y(\theta'|\theta)$ from $\theta$ to $\theta'$ is
\[
P_\mathcal Y(\theta'|\theta):=\begin{cases}  
\frac{1}{4r-6}, \quad \text{ if the pair  $(\theta,\theta')$ is admissible}\\
0 \qquad \text{otherwise}.
\end{cases}
\]
\end{df}

\begin{lem}\label{lem:irr-aper}
Let $r\ge 2$. Then:
\begin{enumerate}
\item The finite state Markov chain $\mathcal Y$ is 
irreducible and aperiodic.

\item The uniform distribution $\mu_r$ on $S$ (where
  $\mu_r(\theta)=1/\#(S)=\frac{1}{4r(r-1)}$ for every $\theta\in S$) is
  the unique stationary distribution for $\mathcal Y$.
\end{enumerate}
\end{lem}
\begin{proof}

It is not hard to see from the definitions that for any $\theta,
\theta'\in S$ there exists a finite admissible sequence
$\theta_1,\dots,\theta_n$ such that $\theta_1=\theta$ and $\theta_n=\theta'$ and
that $n\ge 2$. Hence, for any $\theta,\theta'\in S$ there exists $n\ge 1$
such that the transition probability of $\mathcal Y$ to start at
$\theta$ and to end at $\theta'$ after $n$ steps is positive. This means
that the finite state Markov chain $\mathcal Y$ is indeed irreducible, as claimed. Similarly, it is not hard to verify directly that for any $\theta\in S$ there exist admissible sequences $\theta_1,\dots,\theta_n$ and $\theta_1',\dots,\theta_m'$ with $\theta_1=\theta_n=\theta_1'=\theta_m'=\theta$ such that $m,n\ge 2$ and that $gcd(m,n)=1$. E.g. we can take $n=2,m=3$, $\theta_1=\theta_2=\theta$ and $\theta_1'=\theta_2'=\theta_3'=\theta$. This means that  $\mathcal Y$ is aperiodic. Thus (1) is verified.

The fact that $\mathcal Y$ is irreducible and aperidoc implies (see
\cite[Theorem~4.1]{Seneta}) that
there exists a unique $\mathcal Y$-stationary probability distribution on $S$. A direct computation shows that the uniform distribution $\mu_r$ on $S$ is $\mathcal Y$-stationary. Indeed, let $\nu$ be the distribution on $S$ obtained from $\mu_r$ by applying a single step of $\mathcal Y$.
Then for any $\theta'\in S$ we have
\begin{gather*}
\nu(\theta')=\sum_{\theta\in S} \mu_r(\theta) P_\mathcal Y(\theta'|\theta)=\sum_{\theta\in S_-(\theta')} \mu_r(\theta) P_\mathcal Y(\theta'|\theta)=\sum_{\theta\in S_-(\theta')}  \frac{1}{4r(r-1)}\frac{1}{4r-6}=\\
\#(S_-(\theta')) \frac{1}{4r(r-1)}\frac{1}{4r-6}=(4r-6)\frac{1}{4r(r-1)}\frac{1}{4r-6}=\frac{1}{4r(r-1)}=\mu_r(\theta').
\end{gather*}
Thus $\mu_r$ is indeed $\mathcal Y$-stationary, as claimed, and (2) is
verified.
\end{proof}

\begin{df}
Let $r\ge 3$. Denote by $\mathcal W$ the random process given by the Markov chain $\mathcal Y$ corresponding to the initial distribution $\mu_r$ on $S$.
Thus $\mathcal W$ is a sequence of random variables $\mathcal W=W_1,\dots, W_n,\dots $, where each $W_i$ is a random variable with values in $S$, where $W_1$ has distribution $\mu_r$ and where  for any $\theta,\theta'\in S$ and any $n\ge 1$ 
\[
Pr(W_{n+1}=\theta'| W_n=\theta)=P_\mathcal Y(\theta'|\theta):=\begin{cases}  
\frac{1}{4r-6}, \quad \text{ if the pair  $(\theta,\theta')$ is admissible}\\
0 \qquad \text{otherwise}.
\end{cases}
\]
\end{df}

The sample space $\Omega_\mathcal W$ for $\mathcal W$ is the product
space $\Omega_\mathcal W=\underbrace{S\times S\times \dots}_{\mathbb N
  \text{ copies}} =S^\N$. The space $S$ is endowed with the discrete
topology, and the space  $\Omega_\mathcal W$ is given the product
topology, so that it becomes a compact Hausdorff topological space.

The random process $\mathcal W$ determines a probability measure $\mu_\mathcal W$ on $\Omega_{\mathcal W}$. The support $supp(\mu_\mathcal W)$ of $\mu_\mathcal W$ consists of all the sequences $\omega=\theta_1,\theta_2, \dots \in \Omega_\mathcal W$ such that for every $n\ge 1$ the pair $(\theta_n,\theta_{n+1})$ is admissible.

\begin{lem}\label{lem:cycl}
Let $r\ge 2$.  For a random trajectory $\theta_1,\theta_2,\dots $ of $\mathcal W$ we have
\[
\lim_{n\to\infty} Pr(\theta_1,\dots,\theta_n \text{ is a cyclically admissible sequence})=\frac{2r-3}{2r(r-1)}.
\]
\end{lem}
\begin{proof}
Since $\mathcal Y$ is an irreducible aperiodic finite state Markov
chain, the fact that $\mu_r$ is $\mathcal Y$-stationary implies (see, for example Theorem~4.2 on p. 119 in \cite{Seneta}) that
the distribution of $W_n$ on $S$ converges to $\mu_r$ almost surely as $n\to\infty$.
This means that for every $\theta\in S$ we have $\lim_{n\to\infty} Pr(W_n=\theta)=\mu_r(\theta)=\frac{1}{4r(r-1)}$.  Since for each $\theta\in S$ we have $\#(S_-(\theta))=4r-6$, it now follows that  
for a random trajectory $\omega=\theta_1,\theta_2,\dots, \theta_n\dots $ of $\mathcal W$ we have
\[
\lim_{n\to\infty} Pr((\theta_n,\theta_1) \text{ is an admissible pair}) =\frac{4r-6}{4r(r-1)}=\frac{2r-3}{2r(r-1)} 
\]
and therefore, in view of the definition of $\mathcal W$, we have
\[
\lim_{n\to\infty} Pr(\theta_1,\dots,\theta_n \text{ is a cyclically admissible sequence})=\frac{2r-3}{2r(r-1)},
\]
as required.
\end{proof}

\begin{df}[Property $(\mathcal G)$]
Let $r\ge 3$ be an integer. We say that $\phi\in\Out(F_r)$ has \emph{property $(\mathcal G)$} if all of the following hold:
\begin{enumerate}
\item The outer automorphism $\phi$ is ageometric fully irreducible;
\item We have $i(\phi)=\frac{3}{2}-r$ (so that $\gind(T_\phi)=2r-3$), and $\phi$ has single-element index list $\{\frac{3}{2}-r\}$.
\item There exists a train track representative $f \colon R_r\to R_r$ of $\phi$ such that $f$ has no pINPs and such that $f$ has exactly one nondegenerate illegal turn.
\item  The ideal Whitehead graph $\mathcal{IW}(\phi)$ of $\phi$  is the complete graph on $2r-1$ vertices.
\item The axis bundle for $\phi$ in $CV_r$ consists of a single axis. 
\end{enumerate}

\end{df}

Our main result is the following:

\begin{thm}\label{T:main}
Let $r\ge 3$.  For $n\ge 1$ let $E_n$ be the event that for a trajectory $\omega=\theta_1,\theta_2,\dots $ of $\mathcal W$ the sequence $\theta_1,\dots,\theta_n$  is cyclically admissible.
Also, for $n\ge 1$ let $B_n$ be the event that for a trajectory $\omega=\theta_1,\theta_2,\dots $ of $\mathcal W$ the outer automorphism  $\phi_n=\theta_n\dots\theta_1\in\Out(F_r)$ has property $(\mathcal G)$.

Then the following hold:

\begin{enumerate} 
\item For the conditional probability $Pr(B_n|E_n)$ we have \[\lim_{n\to\infty} Pr(B_n|E_n)=1.\]
\item We have $Pr(E_n)\to_{n\to\infty} \frac{2r-3}{2r(r-1)}$ and $\liminf_{n\to\infty} Pr(B_n)\ge  \frac{2r-3}{2r(r-1)}>0$.
\item For $\mu_\mathcal W$-a.e. trajectory $\omega=\theta_1,\theta_2,\dots $ of $\mathcal W$, there exists an $n_\omega\ge 1$ such that for every $n\ge n_\omega$ such that $\mathfrak t_n=\theta_1,\dots, n_n$ is cyclically admissible, we have that the outer automorphism $\phi_n=\theta_n \circ \cdots \circ \theta_1\in\Out(F_r)$ has property $(\mathcal G)$.
\end{enumerate}
\end{thm}
\begin{proof}
Fix a sequence  $\mathfrak s=\theta_1',\dots,\theta_q'$  provided by Proposition~\ref{P:CP}.

We first establish part (1) of the theorem. 
For $n\ge q$ let $B_n'$ be the event that for a trajectory $\omega=\theta_1,\theta_2,\dots $ of $\mathcal W$ $\mathfrak t_n=\theta_1,\dots,\theta_n$  is a cyclically admissible sequence such that for some $1\le i\le n-q+1$ we have $\theta_i=\theta_1', \theta_{i+1}=\theta_2',\dots \theta_{i+q-1}=\theta_q'$. Note that by definition $B_n'\subseteq E_n$.

Since $\mathfrak s$ is an admissible sequence the probability that for a trajectory $\theta_1,\theta_2,\dots, $ of $\mathcal W$ there is $1\le i \le n-q+1$, such that  $\theta_i=\theta_1', \theta_{i+1}=\theta_2',\dots \theta_{i+q-1}=\theta_q'$, tends to $1$ as $n\to\infty$. Since $\lim_{n\to\infty} Pr(E_n)=\frac{2r-3}{2r(r-1)}>0$, it follows that for the conditional probability $Pr(B_n'|E_n)$ we have $\lim_{n\to\infty} Pr(B_n'|E_n)=1$.

Let $\omega=\theta_1,\theta_2, \dots \in B_n'$ be arbitrary.
Then there exists a cyclic permutation $\mathfrak t_n'$ of  $\mathfrak t_n=\theta_1,\dots,\theta_n$ such that $\mathfrak t_n'$ starts with $\mathfrak s$.
Since $\mathfrak t_n$ is cyclically admissible, $\mathfrak t_n'=\theta_1',\dots, \theta_n'$ is also cyclically admissible.
Therefore, Theorem~\ref{T:tech} applies to $\mathfrak t_n'$ and hence the outer automorphism class $\phi_n'\in\Out(F_r)$ of $\Phi_n'=\theta_n'\circ\dots \circ \theta_1'\in\Aut(F_r)$  has property $(\mathcal G)$. Denote by $\phi_n\in\Out(F_r)$ the outer automorphism class of the automorphism $\Phi_n=\theta_n\circ\dots \circ \theta_1\in \Aut(F_r)$. The fact that $\mathfrak t_n'$ is a cyclic permutation of $\mathfrak t_n$ implies that $\phi_n$ is conjugate to $\phi_n'$ in $\Out(F_r)$. Moreover, $g_{\mathfrak t_n'}$ can be used as a topological representative for $\phi_n$, except that we need to change the marking on $R_r$ from the identity map to the map corresponding to the initial segment of $\mathfrak t_n$ that moved to the end to obtain $\mathfrak t_n'$ as a cyclic permutation of $\mathfrak t_n$. Thus $g_{\mathfrak t_n'}$, with a modified marking, is a topological representative of $\phi_n$. Therefore, by Theorem~\ref{T:tech}, $\phi_n$ also has property $(\mathcal G)$.
By definition of $B_n$ this means that $\omega\in B_n$. Since $B_n'\subseteq E_n$, we have $B_n'\subseteq B_n\cap E_n$.

Hence $B_n'\subseteq B_n\cap E_n$ and $\lim_{n\to\infty} Pr(B_n'|E_n)=1$. Therefore $\lim_{n\to\infty} Pr(B_n|E_n)=1$, as required. Thus part (1) of Theorem~\ref{T:main} is verified.

By Lemma~\ref{lem:cycl} we have $\lim_{n\to\infty} Pr(E_n)=\frac{2r-3}{2r(r-1)}>0$.  Thus part (1) of Theorem~\ref{T:main} implies that $\liminf_{n\to\infty} Pr(B_n)\ge  \frac{2r-3}{2r(r-1)}>0$, and part (2) is verified.

The proof of part (3) is similar to that of part (1).  Namely, for $\mu_\mathcal W$-a.e. trajectory $\omega=\theta_1,\theta_2,\dots $ of $\mathcal W$ the sequence $\mathfrak s$ has infinitely many occurrences in $\omega$.
Let $n_\omega\ge 1$ be such that $\mathfrak t_\omega=\theta_1,\dots, \theta_{n_\omega}$ ends in $\mathfrak s$. Then for every $n\ge n_\omega$ such that $\mathfrak t_n$ is cyclically admissible there exists a cyclic permutation $\mathfrak t_n'$ of  $\mathfrak t_n=\theta_1,\dots,\theta_n$ such that $\mathfrak t_n'$ starts with $\mathfrak s$. Then exactly the same argument as in the proof of (1) above shows that Theorem~\ref{T:tech} applies to $\mathfrak t_n'$ and hence the conclusion of part (3) of Theorem~\ref{T:main} holds for $\omega$.
\end{proof}

\begin{rk}\label{rk:left-right}
Traditionally, random walks on groups are ``right'' random walks, since at each step the current group element gets multiplied by a new generator on the right. Thus, let $G$ be a finitely generated group and $X\subseteq G$ is a finite generating set for $G$ with $X=X^{-1}$. The simple random walk on $G$ with respect to $X$ is a sequence of i.i.d. random variables $X_1,X_2,\dots, X_n,\dots, $ where each $X_n$ is an $X$-valued random variable corresponding to the uniform distribution on $X$. To every trajectory $\omega= x_1,x_2,\dots, x_n, \dots$ (where $x_n\in X$) of this sequence of random variables one associates the sequence $g_\omega=g_1,g_2,\dots, g_n,\dots $ of elements of $G$ where $g_n=x_1\dots x_n$. Thus $g_{n+1}=g_nx_{n+1}$ for all $n\ge 1$.

By contrast, when viewed in terms of $\Aut(F_r)$, our random process $\mathcal W$ is a ``left random walk'' on $\Aut(F_r)$ (or on $\Out(F_r)$). Indeed, to a trajectory $\omega=\theta_1,\dots, \theta_n,\dots $ of $\mathcal W$ we associate a sequence $\Phi_1,\Phi_2,\dots, \Phi_n, \dots$ of elements of $\Aut(F_r)$, where $\Phi_n=\theta_n\dots\theta_1$, so that $\Phi_{n+1}=\theta_{n+1}\Phi_n$.
 
It is possible to convert $\mathcal W$ in a ``right random walk'' on $\Aut(F_r)$, although the resulting statement is somewhat awkward. Note that if $\theta=[x\mapsto yx]$ is a standard Nielsen automorphism of $F_r$, then so is $\theta^{-1}$, with $\theta^{-1}=[x\mapsto y^{-1}x]$. We can say that a pair $(\theta,\theta')$ of elements of $S$ is \emph{anti-admissible} if the pair $(\theta^{-1},(\theta')^{-1})$ is admissible. Similarly, a sequence $\theta_1,\dots, \theta_n$ of elements of $S$ is \emph{anti-admissible} if for all $1\le i<n$ the pair $(\theta_i,\theta_{i+1})$ is anti-admissible. We can then define a random process $\mathcal W^-$ in a similar way to $\mathcal W$: 
We have  $\mathcal W^-=W_1^-, W_2^-,\dots $ where each $W_i^-$ is an $S$-valued random variable, with $W_1^-$ having the uniform distribution on $S$ and with the transition probability $P(W_{n+1}^-=\theta'|W_n^-=\theta)=1/(4r-6)$ if the pair $(\theta,\theta')$ is anti-admissible and $P(W_{n+1}^-=\theta'|W_n^-=\theta)=0$ otherwise. 
To a trajectory $\omega=\theta_1,\theta_2,\dots $ of $\mathcal W^-$ we associate a sequence $\Psi_1,\Psi_2,\dots$ of elements of $\Aut(F_r)$ as $\Psi_n=\theta_1\theta_2\dots \theta_n$. Thus, $\Psi_{n+1}=\Psi_n\theta_{n+1}$.

If the sequence $\omega=\theta_1,\dots, \theta_n, \dots$ is anti-admissible then the sequence $\omega'=\theta_1^{-1},\theta_2^{-1},\dots, \theta_n^{-1},\dots $ is admissible. In this case the process $\mathcal W$ associates to $\omega'$ the sequence $\Phi_n=\theta_n^{-1}\circ\dots \circ \theta_1^{-1}\in \Aut(F_r)$ and $\Phi_n=\Psi_n^{-1}$.

Thus, Theorem~\ref{T:main} implies that for a $\mathcal W^-$-random trajectory $\omega=\theta_1,\theta_2,\dots $, conditioning on the event that $\theta_1,\theta_2,\dots\theta_n$ is cyclically anti-admissible, the probability (corresponding to $\mathcal W^-$) that $\Psi_n=\theta_1\theta_2\dots \theta_n$ is an atoroidal fully irreducible whose inverse $\Psi_n^{-1}$ is ageometric, tends to $1$ as $n\to\infty$.
\end{rk}

\begin{rk}
%For $r\ge 3$ we define a directed labelled graph  $\mathbf{TR}_r$ as follows. The vertex set of  $\mathbf{TR}_r$ is the set of all regular graph-maps $g:R_r\to R_r$ such that $g$ is a homotopy equivalence. The directed edges of $\mathbf{TR}_r$ are defined as follows.
%Whenever $g:R_r\to R_r$ is a vertex of $\mathbf{TR}_r$  and $\theta\in S$ is a standard Nielsen automorphism such that $g_\theta\circ g:R_r\to R_r$ is again regular graph-map, we put a directed edge $e_{g,\theta}$ with label $\theta$ from the vertex $g$ to the vertex $g_\theta\circ g$ of  $\mathbf{TR}_r$.  Note that in this case the map $g_{\theta^{-1}}\circ g_\theta \circ g:R_r\to R_r$ is not regular, and so there is no edge labelled by $\theta^{-1}$ from $g_\theta\circ g$ to $g$ in  $\mathbf{TR}_r$.  Note also that for a vertex $g$ of $\mathbf{TR}_r$  and an element $\theta=[x\mapsto yx]\in S$ there is an edge from $g$ to $g_\theta\circ g$ in $\mathbf{TR}_r$ if and only if $\{x,y\}\not\in \mathcal T(g)$.
%Since vertices of $\mathbf{TR}_r$ are homotopy equivalences and since distinct elements of $S$ are not homotopic as maps $R_r\to\R_r$, it follows that for any two vertices $g,g'$ on $\mathbf{TR}_r$ there is at most one directed edge from $g$ to $g'$ in $\mathbf{TR}_r$ and there are no loop-edges in $\mathbf{TR}_r$.
%We now define a random work

Let $\mathbf{TR}_r$ be the set of all graph-maps $g:R_r\to R_r$ such that $g$ is a homotopy equivalence.
We can re-interpret Theorem~\ref{T:main} in terms of a certain type of a ``train track directed random walk'' on the space $\mathbf{TR}_r$. To every sequence $\omega=\theta_1,\theta_2,\dots, \in \Omega_\mathcal W$ we can associate a sequence $g_\omega=g_1,g_2,\dots $ of elements of $\mathbf{TR}_r$ where $g_n=g_{\theta_n}\circ \dots \circ g_{\theta_1}$ for $n=1,2,\dots $. 

The proof of Theorem~\ref{T:main} can be interpreted as saying that, for $\omega=\theta_1,\theta_2,\dots \in \Omega_\mathcal W$ with associated sequence $g_\omega=g_1,g_2,\dots \in  {\mathbf{TR}_r}^\N$, conditioning on the event that the sequence $\theta_1,\dots,\theta_n$ is cyclically admissible, the probability that $g_n:R_r\to R_r$ is a train track map with exactly one nondegenerate illegal turn and no pINPs, representing an ageometric fully irreducible element of $\Out(F_r)$, tends to $1$ as $n\to\infty$. 
\end{rk}

\subsection{Spectral properties of the train track directed random walk}\label{sect:spectral}

We can also get reasonably precise information about the growth of the PF eigenvalues and of the word length in $Out(F_r)$ along random trajectories of our walk.

First we recall the following classic ergodic theoretic fact known as Kingman's Subadditive Ergodic Theorem:

\begin{prop}\label{P:King}~\cite{King}
Let $(\Omega, \mathcal F,\mu)$ be a probability space and let $T:\Omega\to\Omega$ be a measurable and measure-preserving transformation (that is, one such that, for every measurable subset $Y\subseteq \Omega$, we have $\mu(Y)=\mu(T^{-1}Y)$). Let $Z_n:\Omega\to \mathbb R_{\ge 0}$ be a sequence of random variables (where $n=0,1,2,\dots$) such that, for each $\omega\in \Omega$, and for any $m,n\ge 0$, we have $Z_{n+m}(\omega)\le Z_n(\omega)+Z_m(T^n\omega)$. Then there exists a $T$-invariant random variable $\ell:\Omega\to \mathbb R_{\ge 0}$, such that $\mu$-almost surely and in $L^1(\Omega,\mu)$, we have 
\[
\lim_{n\to\infty} \frac{Z_n}{n}=\ell.
\]
In particular, if $T$ is $\mu$-ergodic, then $\ell=const$ on $\Omega$.
\end{prop}

Note that if $\theta\in S$ is a standard Nielsen automorphism of
$F_r$, then the transition matrix $M(g_\theta)$ is an $r\times r$
elementary matrix obtained from the $r\times r$ identity matrix by
changing a single off-diagonal entry from $0$ to $1$. Let $S'$ be the
set of all such $r\times r$ elementary matrices. Note that $S'\subseteq SL(r,\mathbb Z)$ and that $\#(S')=r^2-r$.

We note the following basic fact that will be useful in our arguments:
\begin{lem}\label{L:norm}
Let $r\ge 2$. Then:

(1) For every $M\in S'$ and every $v\in \mathbb R^r$, we have $||Mv||\ge ||v||$.

(2) For any $M_1,\dots, M_n\in S'$, we have $||M_n \cdots M_1||\ge 1$.
\end{lem}
\begin{proof}
Part (1) is obvious from the definition of $S'$.

From part (1), by induction on $n$, we get that, for any $M_1,\dots, M_n\in S'$ and any $v\in \mathbb R^r$, we have $||M_n \cdots M_1 v||\ge ||v||$. Therefore $||M_n \cdots M_1||\ge 1$, and (2) holds.
\end{proof}

On the sample space $\Omega_\mathcal W$ of $\mathcal W$, we define a \emph{shift-map} $T \colon \Omega_\mathcal W\to \Omega_\mathcal W$ by
\[
T: \theta_1,\theta_2,\theta_3 \dots \mapsto \theta_2,\theta_3,\dots
\]
for each $\omega=\theta_1,\theta_2,\dots \in \Omega_\mathcal W$.

Then $T \colon \Omega_\mathcal W\to \Omega_\mathcal W$ is a continuous $\mu_\mathcal W$-measure preserving map. Since $\mathcal Y$ is a finite-state irreducible aperiodic Markov chain, it follows that $T$ is $\mu$-ergodic.

Consider the following functions $X_n: \Omega_\mathcal  W\to \mathbb R_{\ge 0}$ (where $n\ge 1$):
\[
X_n(\theta_1,\theta_2,\theta_3 \dots )= \log ||M(g_{\theta_n}) \cdots  M(g_{\theta_1})||
\]
for every $\omega=\theta_1,\theta_2,\dots \in \Omega_\mathcal W$. We also put $X_0:=0$. 

Note that for $\mathfrak t_n=\theta_1,\dots, \theta_n$ we have $g_{\mathfrak t_n}=g_{\theta_n}\circ \dots \circ g_{\theta_1}$ and therefore
$M(g_{\mathfrak t_n})=M(g_{\theta_n})\dots M(g_{\theta_n})$.

We have:

\begin{prop}\label{P:SET}
Let $r\ge 3$. There exists a number $\ell_1\ge 0$, called the \emph{top Lyapunov exponent}, such that for $\mu_\mathcal W$-a.e. trajectory  $\omega=\theta_1,\theta_2,\dots$ of $\mathcal W$ we have
\[
\lim_{n\to\infty} \frac{1}{n} \log ||M(g_{\mathfrak t_n})|| = \lim_{n\to\infty} \frac{1}{n} \log ||M(g_{\theta_n}) \cdots M(g_{\theta_1})||=\ell_1.
\]
where $\mathfrak t_n=\theta_1,\dots, \theta_n$ for $n=1,2,\dots $.
\end{prop}
\begin{proof}

Let $\omega=\theta_1,\theta_2, \dots \in \Omega_\mathcal W$ be arbitrary.
By Lemma~\ref{L:norm}, for every $n\ge 1$ we have $||M(g_{\theta_n})\cdots M(g_{\theta_1})||\ge 1$ and therefore  $X_n(\omega)\ge 0$. Since $X_0=0$, we also have $X_0(\omega)\ge 0$.
For $m,n\ge 1$ we have $X_m(T^n\omega)=\log ||M(g_{\theta_{n+m}})\cdots M(g_{\theta_{n+1}})||$ and $X_n(\omega)=\log ||M(g_{\theta_n}) \cdots  M(g_{\theta_1})||$. 
Since \[||M(g_{\theta_{n+m}})\cdots M(g_{\theta_1})||\le ||M(g_{\theta_{n+m}})\cdots M(g_{\theta_{n+1}})||\cdot ||M(g_{\theta_n}) \cdots  M(g_{\theta_1})||,\] we also have \[\log ||M(g_{\theta_{n+m}})\cdots M(g_{\theta_1})||\le \log ||M(g_{\theta_{n+m}})\cdots M(g_{\theta_{n+1}})|| + \log ||M(g_{\theta_n}) \cdots  M(g_{\theta_1})||,\] that is 
$X_{n+m}(\omega)\le X_n(\omega)+X_m(T^n\omega)$. It is easy to check that $X_{n+m}(\omega)\le X_n(\omega)+X_m(T^n\omega)$ also holds if at least one of $m,n$ is equal to $0$.

Since $T$ is a $\mu_\mathcal W$-ergodic transformation, Proposition~\ref{P:King} (Kingman's Subadditive Ergodic Theorem) now implies that there 
exists a number $\ell_1\ge 0$ such that for $\mu_\mathcal W$-a.e. trajectory  $\omega=\theta_1,\theta_2,\dots$ of $\mathcal W$ we have
\[
\lim_{n\to\infty} \frac{1}{n} \log ||M(g_{\theta_n}) \cdots M(g_{\theta_1})||=\ell_1.
\]
\end{proof}

\begin{prop}\label{P:pos}
Let $r\ge 3$ and let $\ell_1$ be provided by Proposition~\ref{P:SET}. Then $\ell_1>0$.
\end{prop}
\begin{proof}
It is possible to derive the fact that $\ell_1>0$ from a general result of Guivarc'h~\cite{Gui} on the simplicity of the Lyapunov spectrum in the context of the Multiplicative Ergodic Theorem for matrix-valued Markov chains satisfying some natural ``irreducibility'' and ``contractibility'' conditions (which are satisfied in our case).

We provide a direct and more elementary argument for $\ell_1>0$ here.

By Proposition~\ref{P:CP}, there exists a cyclically admissible sequence $\mathfrak s$ such that $M(g_\mathfrak s)>0$. By replacing $\mathfrak s$ by its positive power, we can further assume that every entry of $M(g_\mathfrak s)$ is $\ge 2$.

For a sequence $\mathfrak t=\theta_1,\dots, \theta_n$ denote by $\langle \mathfrak s,\mathfrak t\rangle$ the number of times $\mathfrak s$ occurs as a sub-block of $\mathfrak t$.

By the Law of Large Numbers applied to $\mathcal Y$, there exists $\alpha>0$ such that, for $\mu_\mathcal W$-a.e. trajectory $\omega=\theta_1,\theta_2,\dots $ of $\mathcal W$, we have
\[
\lim_{n\to\infty}\frac{\langle \mathfrak s, \mathfrak t_n\rangle}{n}=\alpha>0,
\]
where $\mathfrak t_n=\theta_1,\dots,\theta_n$.

Let $\omega=\theta_1,\theta_2,\dots $ be a  $\mu_\mathcal W$-random trajectory of $\mathcal W$. Then for $n>>1$ there are $n\alpha+ o(n)$ occurrences of $\mathfrak s$ in $\mathfrak t_n$. Hence, we can find $n\alpha/k+ (1/k)o(n)$ disjoint occurrences of $\mathfrak s$ in $\mathfrak t_n$, where $k$ is the length of $\mathfrak s$.
Thus, we can subdivide $M(g_{\theta_n})\cdots M(g_{\theta_1})$ as a product
\[
M(g_{\theta_n}) \cdots M(g_{\theta_1})= C_qB \cdots C_1BC_0,
\]
where $B=M(g_\mathfrak s)$, where $q=n\alpha/k+ (1/k)o(n) \ge n\alpha/(2k)$ and where each $C_i$ is a product of several consecutive matrices from the product $M(g_{\theta_n})\cdots M(g_{\theta_1})$.
Recall that every entry in $B$ is $\ge 2$. Thus, for every vector $v\in \mathbb R^r$ we have
\begin{gather*}
||M(g_{\theta_n})\cdots M(g_{\theta_1})v||=||C_qBC_{q-1}   \cdots C_1BC_0v||\ge ||BC_{q-1}  \cdots C_1BC_0v||\ge \\
2||C_{q-1} B  \cdots C_1BC_0v||\ge \dots \ge 2^q||C_0v||\ge 2^q||v||\ge  2^{n\alpha/(2k)}||v||.
\end{gather*}
Therefore, $||M(g_{\theta_n})\cdots M(g_{\theta_1})||\ge 2^{n\alpha/(2k)}$ and $\log ||M(g_{\theta_n})\cdots M(g_{\theta_1})||\ge n\alpha/(2k)\log 2$, so that
\[
\liminf_{n\to\infty} \frac{1}{n}\log ||M(g_{\theta_n})\cdots M(g_{\theta_1})||\ge \alpha/(2k)\log 2.
\]
Hence, $\ell_1\ge \alpha/(2k)\log 2>0$.
\end{proof}

The growth of the spectral radius of $M(g_{\theta_n})\cdots M(g_{\theta_1})$ is more important for our purposes than the growth of $||M(g_{\theta_n})\cdots M(g_{\theta_1})||$.
Luckily, in our situation, these two quantities grow roughly at the same rate, as follows from the following general result due to Terence Tao. The proof of this fact was communicated to us by Tao on MathOverflow. Since the statement of Proposition~\ref{P:Tao} does not seem to be available in the literature, we include Tao's proof here.

\begin{prop}\label{P:Tao}
Let $M=(m_{ij})_{ij=1}^r$ be an $r\times r$ matrix with real coefficients such that all $m_{ij}\ge 1$. Then $\lambda(M)\le ||M||\le r \left(\lambda(M)\right)^2$.
\end{prop}
\begin{proof}
Since all entries of $M$ are $>0$, the spectral radius $\lambda(M)>0$ is the Perron-Frobenius eigenvalue of $M$. There exists a nonzero vector $u\in \mathbb R^r$ such that $Mu=\lambda(M)u$ and hence $||Mu||/||u||=\lambda(M)$.  The inequality $\lambda(M)\le ||M||$ is obvious since $||M||=\max_{v\in \mathbb R^r\setminus\{0\}} \frac{||Mv||}{||v||}$.

Recall that, by the Spectral Theorem, $||\lambda(M)||=\lim_{n\to\infty} \sqrt[n]{||M^n||}$.  
Consider the ``product'' partial ordering $\le$ on $\mathbb R^r$ where
$(x_1,\dots ,x_n)\le (y_1,\dots ,y_n)$ whenever $x_i\le y_i$ for $i=1, \dots, r$.

Observe that if $v,u\in \mathbb R^r$ are  vectors with non-negative coordinates and such that $v\le u$, then $||v||\le ||u||$ and
$Mv\le Mu$.

Also notice that $Me_j\ge m_{ij}e_i$ and $Me_i\ge e_j$.
Hence $M^2e_j\ge m_{ij}Me_i\ge m_{ij}e_j$, so that $M^2e_j\ge m_{ij}e_j$.
Iterating this argument we get $M^{2n}e_j\ge m_{ij}^ne_j$ and hence, by taking the norm of both sides, we get
$m_{ij}^n \le ||M^{2n}e_j||\le ||M^{2n}||$.
By taking the $n$-th root and passing to the limit, by the Spectral Theorem we get
$m_{ij}\le \lambda(M)^2$ for all $1\le i,j \le r$.

Hence, $\max m_{ij}\le  \lambda(M)^2$, and therefore $||M||\le r \lambda(M)^2$. 
\end{proof}

\begin{thm}\label{T:growth}
Let $r\ge 3$ and let $\ell_1$ be provided by Proposition~\ref{P:SET} (so that $\ell_1>0$ by Proposition~\ref{P:pos}).

Then, for $\mu_\mathcal W$-a.e. trajectory $\omega=\theta_1,\theta_2,\dots $ of $\mathcal W$, the following hold:

(1)
\[
0<\ell_1/2\le \liminf_{n\to\infty} \frac{1}{n}\log  \lambda(g_{\mathfrak t_n}) \le \limsup_{n\to\infty} \frac{1}{n}\log \lambda(g_{\mathfrak t_n}) \le \ell_1,
\]
where $\mathfrak t_n=\theta_1,\dots, \theta_n$ for $n\ge 1$.

(2) For any strictly increasing sequence of indices $1\le n_1<n_2<n_3 <\dots $ such that for each $i\ge 1$ $\mathfrak t_{n_i}=\theta_1,\dots, \theta_{n_i}$ is cyclically admissible we have
\[
0<\ell_1/2\le \liminf_{i\to\infty} \frac{1}{n_i}\log  \lambda(\phi_{n_i}) \le \limsup_{i\to\infty} \frac{1}{n_i} \log \lambda(\phi_{n_i}) \le \ell_1,
\]
where $\phi_{n_i}\in\Out(F_r)$ is the outer automorphism represented by $g_{\mathfrak t_{n_i}}$.
\end{thm}
\begin{proof}

Let $\mathfrak s$ be the admissible sequence provided by Proposition~\ref{P:CP}, so that $M(g_\mathfrak s)>0$.

Then, for $\mu_\mathcal W$-a.e. trajectory $\omega=\theta_1,\theta_2,\dots $ of $\mathcal W$ the sequence $\mathfrak s$ has infinitely many occurrences in $\omega$.
Let $n(\omega)\ge 1$ be such that $\mathfrak t_\omega=\theta_1,\dots, \theta_{n(\omega)}$ ends in $\mathfrak s$. Then, for every $n\ge n(\omega)$, we have $M(g_{\mathfrak t_n})>0$. Then, by Proposition~\ref{P:Tao}, for every $n\ge n(\omega)$ we have $\sqrt{r} \sqrt{||M(g_{\mathfrak t_n})||} \le \lambda(g_{\mathfrak t_n}) \le ||M(g_{\mathfrak t_n})||$. Hence, by Proposition~\ref{P:SET} and Proposition~\ref{P:pos},
\[
0<\ell_1/2\le \liminf_{n\to\infty} \frac{1}{n} \log
\lambda(g_{\mathfrak t_n}) \le \limsup_{n\to\infty} \frac{1}{n} \log \lambda(g_{\mathfrak t_n}) \le \ell_1,
\]
as required, so that part (1) of the theorem is verified.  Note further that in this situation for every $n\ge n(\omega)$ such that $\mathfrak t_n$ is cyclically admissible, Theorem~\ref{T:tt} implies that $g_{\mathfrak t_n}$ is an expanding irreducible train track representative of $\phi_n$ with $M(g_{\mathfrak t_n})>0$. Therefore, by Proposition~\ref{P:stretch}, we have  $\lambda(g_{\mathfrak t_n})=\lambda(\phi_n)$ for every such $n$.  Hence part (2) of the theorem holds as well.
\end{proof}

For $\Phi\in \Aut(F_r)$, denote $|\Phi|_A:=\max_{a\in A}
|\Phi(a)|_A$. Let $Q=\{\psi_1,\dots,\psi_m\}$ be any finite generating
set of $\Out(F_r)$ such that $Q=Q^{-1}$. For an element $\phi\in
\Out(F_r)$, denote by $|\phi|_Q$ the geodesic word-length of $\phi$
with respect to the generating set $Q$ of $\Out(F_r)$; that is
$|\phi|_Q$ is the smallest $n$ such that $\phi$ can be written as a
product $\phi=\psi_{i_1}\dots \psi_{i_n}$ where $\psi_{i_j}\in Q$.

For each $\psi_i$, choose an automorphism $\Psi_i\in \Aut(F_r)$ in the outer automorphism class $\psi_i$. Finally, put $|Q|_A:=\max_{i=1}^m |\Psi_i|_A$.  

Recall from Definition~\ref{D:stretch} the definition of the stretch factor $\lambda(\phi)\ge 1$ for a $\phi\in \Out(F_r)$.

\begin{lem}\label{L:ST}
Let $r\ge 2$ and let $Q=Q^{-1}$ be a finite generating set of $\Out(F_r)$. Then the following hold:
\begin{enumerate}
\item  For each $\phi\in\Out(F_r)$ and representative $\Phi\in \Aut(F_r)$ in the outer automorphism class $\phi$, we have $\lambda(\phi)\le |\Phi|_A$.
\item If $\phi=\psi_{i_1}\dots \psi_{i_n}$ is a word of length $n$ over $Q$, then $\lambda(\phi)\le (|Q|_A)^n$.  
\item For any $\phi\in \Out(F_r)$, we have $\lambda(\phi)\le (|Q|_A)^{|\phi|_Q}$.
\end{enumerate}
\end{lem}

\begin{proof}
Note that, for each $1 \ne w \in F_r$ and each $n \ge 1$, we have $||\phi^n(w)||_A \le |\Phi^n(w)|_A \le |w|_A  |\Phi|_A^n$ and, by taking $n$-th roots and passing to the limit, we get $\lambda(\phi,w)\le  |\Phi|_A$. Therefore, by the definition of the stretch factor of an element of $\Out(F_r)$ (Definition~\ref{D:stretch}), we have that $\lambda(\phi)=\sup_{w\in F_r-\{1\}} \lambda(\phi,w) \le |\Phi|_A$. Thus (1) is verified.

Part (2) follows from part (1) since, if $\phi=\psi_{i_1}\dots \psi_{i_n}$ is a word of length $n$ over $Q$, then $|\Phi|_A \le (|Q|_A)^n$, where $\Phi=\Psi_{i_1}\dots \Psi_{i_n}$.

Part (2) directly implies part (3).
\end{proof}

We can now prove that the random walk $\mathcal W$ has a positive linear rate of escape with respect to the word metric on $\Out(F_r)$:

\begin{thm}\label{T:LRE}
Let $r\ge 3$ and let $Q$ be a finite generating set of $\Out(F_r)$ such that $Q=Q^{-1}$. Then there exists a constant $c>0$ such that, for $\mu_\mathcal W$-a.e. trajectory $\omega=\theta_1,\theta_2,\dots $ of $\mathcal W$,
\[
\lim_{n \to \infty} \frac{1}{n} |\theta_n \dots \theta_1|_Q =c.
\]
\end{thm}

\begin{proof}
For $n\ge 1$, define $Z_n \colon \Omega_\mathcal W \to \R_{\ge 0}$ as
\[
Z_n(\theta_1, \theta_2,\dots ):=|\theta_n \dots \theta_1|_Q.
\]
Also put $Z_0=0$. Then, for $m,n\ge 1$ and any $\omega = \theta_1, \theta_2 \dots \in \Omega_\mathcal W$, we have $Z_m(T^n\omega)=|\theta_{n+m} \dots \theta_{n+1}|_Q$ and $Z_n(\omega)=|\theta_n \dots \theta_1|_Q$. 
Since $|\theta_{n+m}\dots \theta_1|_Q\le |\theta_{n+m}\dots \theta_{n+1}|_Q+|\theta_n \dots \theta_1|_Q$, it follows that 
$Z_{n+m}(\omega)\le Z_n(\omega)+Z_m(T^n\omega)$. It is easy to check that $Z_{n+m}(\omega)\le Z_n(\omega)+Z_m(T^n\omega)$ also holds if at least one of $m,n$ is equal to $0$. 

Since $T$ is a $\mu_\mathcal W$-ergodic transformation, Proposition~\ref{P:King} now implies that there 
exists a number $c\ge 0$ such that, for $\mu_\mathcal W$-a.e. trajectory $\omega=\theta_1,\theta_2,\dots$ of $\mathcal W$, we have
\[
\lim_{n\to\infty} \frac{1}{n} |\theta_n\dots \theta_1|_Q=c.
\]
It remains to show that $c>0$, that is, to rule out the possibility $c=0$.
Thus, assume that $c=0$. Then for $\mu_\mathcal W$-a.e. trajectory $\omega=\theta_1,\theta_2,\dots $ of $\mathcal W$, the word-length $|\phi_n|_Q$ grows sub-exponentially in $n$, where 
$\phi_n=\theta_n \dots \theta_1\in \Out(F_r)$.

On the other hand, for $\mu_\mathcal W$-a.e. trajectory $\omega = \theta_1,\theta_2, \dots$, there exist infinitely many indices $1\le n_1<n_2<\dots $ such that, for each $i\ge 1$, we have that $\theta_1,\dots,\theta_{n_i}$ is a cyclically admissible sequence.
Theorem~\ref{T:growth} implies that $\lambda(\phi_{n_i})$ grows exponentially fast in $n_i$. Then part (3) of Lemma~\ref{L:ST}, applied to $\phi_{n_i}=\theta_{n_i}\dots \theta_1$, implies that $|\phi_{n_i}|_Q$ must grow at least linearly fast in $n_i$. This contradicts the fact that $|\phi_n|_Q$ grows sub-exponentially in $n$. Thus, the case $c=0$ is impossible, and hence $c>0$, as required.
\end{proof}

\section{Realizability of powers of train track maps with one illegal turn by admissible compositions}\label{s:New}

In this section we show that for each train track map $g \colon R_r\to R_r$ with exactly one illegal turn, some positive power $g^p$ of $g$ can be represented as the composition of a cyclically admissible sequence, so that $g^p$ is reachable by our walk $\mathcal W$; see Theorem~\ref{t:decomp} below for a precise statement.

We assume some familiarity of the reader with Stallings folds, and only briefly recall the basics related to folds here; we refer the reader to \cite{s83,KM02} for details.

\begin{df}[Stallings folds]\label{d:folds}

Let $g \colon \Gamma\to\Gamma'$ be a regular graph-map. We say that a nondegenerate turn $\tau=\{e_1,e_2\}$ is $g$-\emph{smooth} if $Dg(\tau)$ is a nondegenerate turn in $\Gamma'$. We say that a nondegenerate turn $\tau=\{e_2,e_1\}$ is $g$-\emph{foldable} if $Dg(\tau)$ is a degenerate turn in $\Gamma'$.

Suppose that  $\tau=\{e_1,e_2\}$ is a $g$-foldable turn. Then there exist maximal nontrivial initial segments $e_1',e_2'$ of $e_1,e_2$ accordingly such that $g(e_1')=g(e_2')$ as paths. Note that this automatically means that $g$ send terminal points of $e_1',e_2'$ to a vertex of $\Gamma'$.

We consider the equivalence relation on $\Gamma$ generated by identifying $e_1'$ with $e_2'$ according to the map $g$. The quotient object is a graph $\Gamma_1$ and the quotient map $q \colon \Gamma\to\Gamma_1$ is a graph-map called a \emph{Stallings fold}, or just a \emph{fold}; then $q(e_1')=q(e_2')$ is an edge of $\Gamma_1$. There is also a natural graph-map $g' \colon \Gamma_1\to\Gamma'$ such that $g=g'\circ q \colon \Gamma\to\Gamma'$.

The fold $q$ is said to be
\begin{enumerate}
\item a \emph{complete fold} if $e_1'=e_1$, $e_2'=e_2$
\item a \emph{partial fold} if $e_1'\ne e_1$, $e_2'\ne e_2$
\item a \emph{proper full fold} if either $e_1'=e_1, e_2'\ne e_2$ or $e_1'\ne e_1, e_2'=e_2$ (i.e. if, for some $i,j$ such that $\{i,j\}=\{1,2\}$, $q$ identifies a proper initial segment of $e_i$ with the entire edge $e_j$). 
\end{enumerate}
\end{df}

Note that a fold $q$ determined by a $g$-foldable turn $\{e_1,e_2\}$ as above fails to be a homotopy equivalence if and only if $q$ is a complete fold and $t(e_1)=t(e_2)$ in $\Gamma$. 

The following important result is due to Stallings~\cite{s83}:

\begin{prop}\label{p:Sta}
Let $\Gamma,\Gamma'$ be finite connected graphs without any degree-1 vertices. Let $g \colon \Gamma\to\Gamma'$ be a regular graph-map such that $g$ is a homotopy equivalence.
Then there exists a decomposition of $g$ as a composition
\[
\Gamma = \Gamma_0 \xrightarrow{q_1} \Gamma_1 \xrightarrow{q_2} \cdots \xrightarrow{q_{n-1}} \Gamma_{n-1} \xrightarrow{q_n} \Gamma_n = \Gamma' 
\]
such that $q_i$, with $1 \leq i \leq n-1$, is a fold and $q_n$ is a graph-isomorphism (and in particular $q_n$ is a homeomorphism).
Moreover, for $1\le i<n$ the fold $q_i$ a homotopy equivalence. 
\end{prop}

A Stallings fold decomposition of $g\colon \Gamma  \to \Gamma'$ in Proposition~\ref{p:Sta} can be obtained as follows (the maps $h_i$ are depicted in Figure \ref{Fi:Sta}). Put $\Gamma_0=\Gamma$ and $h_0=g \colon \Gamma_0\to\Gamma'$.
If $g$ is not a graph-automorphism already, choose a $g$-foldable turn $\{e_1,e_2\}$ in $\Gamma_0=\Gamma$.  Then take $q_1 \colon \Gamma_0\to\Gamma_1$ to be the fold determined by this turn, so that we also get a homotopy equivalence $h_1 \colon \Gamma_1\to\Gamma'$ such that $h_0=h_1\circ q_1$. Apply the same procedure to the map $h_1 \colon  \Gamma_1\to\Gamma'$ and, proceeding inductively, construct a sequence of folds $q_k \colon \Gamma_{k-1}\to\Gamma_{k}$ and maps $h_k \colon \Gamma_k\to\Gamma'$, for $k=1,2,\dots $, such that $h_k\circ q_k=h_{k-1}$. Each of $h_k,q_k$ is a homotopy equivalence. The process must terminate in a finite number of steps since, by construction, $\Gamma_k$ has fewer edges then $\Gamma_{k-1}$. If the process terminates with the map $h_n \colon \Gamma_n\to\Gamma'$, then every nondegenerate turn in $\Gamma_n$ is $h_n$-smooth, and the map $h_n \colon \Gamma_n\to\Gamma'$ is a graph-isomorphism. See the illustration of this process in Figure~\ref{Fi:Sta}.

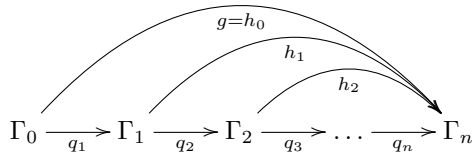
\begin{figure}[here]
\[
\xymatrix{\Gamma_0 \ar[r]_{q_1} \ar@/^4pc/[rrrr]_{g=h_0} & \Gamma_1 \ar[r]_{q_2} \ar@/^3pc/[rrr]_{h_1} & \Gamma_2 \ar[r]_{q_3} \ar@/^2pc/[rr]_{h_2} & \dots \ar[r]_{q_n}  & \Gamma_n \\}
\]
\caption{Constructing a Stallings folds decomposition}\label{Fi:Sta}
\end{figure}

\begin{lem}\label{L:StallingsFoldDecomp}
Let $r\ge 2$ and let $g \colon R_r \to R_r$ be a regular graph map such that $g$ is a homotopy equivalence and such that there is at most one $g$-foldable nondegenerate turn in $R_r$.

Then there exists a decomposition $g=q_n\circ \dots q_1$ such that:

\begin{enumerate}
\item For $i=1,\dots, n$ we have $q_i \colon \Gamma_{i-1}\to \Gamma_i$ is a regular graph map, where $\Gamma_i=\Gamma_{i-1}=R_r$.
\item For $1\le i<n$ the map $q_i$ is a proper full fold on $\Gamma_{i-1}$.
\item  The map $q_n \colon \Gamma_{n-1}\to \Gamma_n$ is a graph-isomorphism (and in particular a homeomorphism).
\end{enumerate} 
\end{lem}

\begin{proof}
Recall that we have an orientation on $R_r$ so that $E_+R_r=\{e_1,\dots, e_r\}$ consists of exactly $r$ edges.
For a regular graph map $f \colon R_r\to R_r$ we define the \emph{complexity} $c(f)$ as $c(f):=\sum_{e\in E_+R_r} |f(e)|$.
Note that $c(f)\ge r$, and, assuming that $f$ is also a homotopy equivalence, we have $c(f)=r$ if and only if $f$ is a graph isomorphism.

We prove the statement of the lemma by induction on $c(g)$. 
If $c(g)=r$, then $g$ is a graph isomorphism and the conclusion of the lemma holds with $n=1$ and $q_1=g$.

Suppose now that $c(g)>r$ and that the statement of the lemma has been established for all smaller values of the complexity.
Since $c(g)>r$, there is exactly one nondegenerate $g$-foldable turn $\tau$ in $R_r$. Without loss of generality we may assume that $\tau=\{e_1,e_2\}$.  
Let
\[
R_r = \Gamma_0 \xrightarrow{q_1} \Gamma_1 \xrightarrow{h_1}  R_r,
\]
where $\Gamma_1$ is obtained from $\Gamma_0=R_r$ by applying a fold $q_1$, which is the fold corresponding to the $g$-foldable turn $\{e_1,e_2\}$. 
Thus $g=h_1\circ q_1$.

Note that $q_1 \colon R_r\to\Gamma_1$ cannot be a complete fold since in that case $q_1$ would not be a homotopy equivalence, and $\Gamma_1$ would be an $(r-1)$-rose. Thus $q_1$ is either a proper full fold or a partial fold. 

Suppose first that $q_1$ is a partial fold. Then $\Gamma_1$ would be as in Figure~\ref{Fi:partial} and $\Gamma_1$ would not be homeomorphic to $R_r$.
On the other hand, by construction, every nondegenerate turn in $\Gamma_1$ would be $h_1$-smooth. 
Indeed, the turn $e_1'',e_2''$ as in  Figure~\ref{Fi:partial}  is $h_1$-smooth since the fold $q_1$ identified maximal initial segments of $e_1,e_2$ with the same $g$-image. The turns $\overline e, e_1''$ and $\overline e, e_1''$ are $h_1$-smooth because by assumption the paths $g(e_1)$ and $g(e_2)$ are tight.  Every other nondegenerate turn in $\Gamma_1$ is already present in $\Gamma_0$ and is $g$-smooth there, and hence it is $h_1$-smooth in $\Gamma_1$. Thus there are no folds applicable to $\Gamma_1$, and yet the map $h_1 \colon \Gamma_1\to \Gamma_n$ is not a graph-isomorphism, yielding a contradiction.

\begin{figure}
\includegraphics[width=1in]{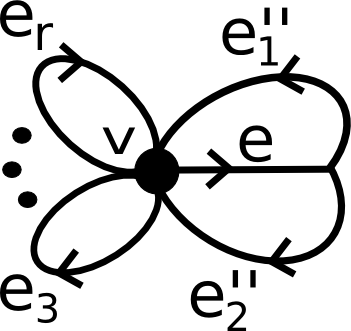}
\caption{Partial fold on the rose}\label{Fi:partial}
\end{figure}

Hence $q_1$ is a proper full fold, so that $\Gamma_1=R_r$.
Let $e_1=e_1'e_1''$, where $e_1'$ is a proper initial segment of $e_1$, and let $q_1$ completely fold $e_1'$ around the edge $e_2$. Thus $\Gamma_1$ is a rose with loop-edges $e_1'',e_2,\dots, e_r$ wedged at a single vertex $v_1$.

Note that by construction any turn formed by any two distinct directions among $\overline{e_1''}, e_2, \overline{e_2},\dots, e_r, \overline{e_r}$ is $h_1$-smooth because these turns were already present in $\Gamma_0$ and they were $h_0$-smooth. Since $h_0(e_1)$ is a tight edge-path and since $e_1'$ has been folded with $e_2$, the turn $\overline{e_2}, e_1$ is $h_1$-smooth. Thus, the only possibility for a nondegenerate $h_1$-foldable turn in $\Gamma_1$ is a turn consisting of $e_1''$ and one of the directions $e_2, e_3,\overline{e_3},\dots, e_r, \overline{e_r}$.
There is at most one among the directions $e_2, e_3,\overline{e_3},\dots, e_r, \overline{e_r}$ which can form a $h_1$-foldable turn together with $e_1''$ since otherwise some two distinct directions among $e_2, e_3,\overline{e_3},\dots, e_r, \overline{e_r}$ would have formed a $g$-foldable turn in $\Gamma_0$, contrary to the assumption that $\{e_1,e_2\}$ was the only nondegenerate $g$-foldable turn in $\Gamma_0$. Thus $\Gamma_1=R_r$ and there is at most one nondegenerate $h_1$-foldable turn in $\Gamma_1$.
Since $c(h_1)<c(g)$, by the inductive hypothesis applied to $h_1$, there exists a decomposition of $h_1$ as $h_1=q_n\circ \dots \circ q_2$
\[
\Gamma_1 \xrightarrow{q_2} \Gamma_2 \xrightarrow{q_3} \cdots \xrightarrow{q_{n-1}} \Gamma_{n-1} \xrightarrow{q_n} \Gamma_n = R_r
\] 
satisfying the requirements of Lemma~\ref{L:StallingsFoldDecomp} for $h_1$.
Then $g=q_n\circ \dots \circ q_2\circ q_1$ is the required decomposition for $g$.
\end{proof}

Recall that $F_r=F(A)$ where $A=\{a_1,\dots, a_r\}$ and that $R_r$ is equipped with the marking identifying $e_i\in E_+R_r$ with $a_i$ for $i=1,\dots, r$. 
We say that $\Psi\in \Aut(F_r)$ is a \emph{permutational} automorphism if there exists a permutation $\sigma\in S_r$ and $\epsilon_1,\dots \epsilon_r\in \{1,-1\}$ such that $\Psi(a_i)=a_{\sigma(i)}^{\epsilon_i}$ for $i=1,\dots, r$.
Recall also that with every $\Phi\in \Aut(F_r)$ we have associated its standard representative $g_\Phi \colon R_r\to R_r$, see Definition~\ref{d:standard-rep}.

The following lemma is an immediate corollary of the definitions:
\begin{lem}\label{lem:perm} Let $r\ge 2$. Then:
\begin{enumerate}
\item A regular graph map $g:R_r\to R_r$ is a graph-isomorphism if and only if $g=g_\Psi$ for some permutational automorphism $\Psi$ of $F_r$.
\item A homotopy equivalence regular graph map $g \colon R_r\to R_r$ is a single proper full fold if and only if $g=g_\theta$ (up to isotopy relative to the vertex of $R_r$) for some elementary Nielsen automorphism $\theta=[x\mapsto yx]$ of $F_r$.
\item If $\Psi$ is a permutational automorphism of $F_r$ and $\theta=[x\mapsto yx]$ is an elementary Nielsen automorphism of $F_r$, then for $\theta'=[\Psi(x)\mapsto \Psi(y)\Psi(x)]$, we have $\Psi\theta=\theta'\Psi$ in $\Aut(F_r)$ and, moreover, $g_\Psi\circ g_\theta=g_{\theta'}\circ g_\Psi$, as maps $R_r\to R_r$.
\end{enumerate}
\end{lem}

\begin{thm}\label{t:decomp}
Let $r\ge 2$ and let $g \colon R_r\to R_r$ be a train track map with exactly one nondegenerate illegal turn representing some $\phi\in \Out(F_r)$. Then there exist $p\ge 1$ and a decomposition
\[
g^p=g_{\theta_n}\circ \dots \circ g_{\theta_1},
\]
where $\theta_1,\dots, \theta_n$ is a cyclically admissible sequence of elementary Nielsen automorphisms of $F_r$.
\end{thm}

\begin{proof}
By Lemma~\ref{L:StallingsFoldDecomp} and Lemma~\ref{lem:perm}, there exist elementary Nielsen automorphisms $\theta_1,\dots, \theta_m$ and a permutational automorphism $\Psi$ of $F_r$ such that
\[
g=g_{\Psi}\circ g_{\theta_m}\circ \dots \circ g_{\theta_1}. 
\]
Let $p$ be the order of $\Psi$ in $\Aut(F_r)$.  Then $g_{\Psi}^p=Id_{R_r}$. 
We have
\[
g^p=(g_{\Psi}\circ g_{\theta_m}\circ \dots \circ g_{\theta_1})\circ \dots \circ (g_{\Psi}\circ g_{\theta_m}\circ \dots \circ g_{\theta_1}),
\]
where the term $g_{\Psi}\circ g_{\theta_m}\circ \dots \circ g_{\theta_1}$ is repeated $p$ times. 
By applying part (3) of Lemma~\ref{lem:perm}, we can move all the occurrences of $g_\Psi$ in the above expression to the right and obtain a decomposition of $g^p$ as
\[
g^p=g_{\theta_{pm}'}\circ \dots \circ g_{\theta_1'} \circ  g_{\Psi}^p=g_{\theta_{pm}'}\circ \dots \circ g_{\theta_1'} 
\]
for some elementary Nielsen automorphisms $\theta_1',\dots \theta_m'$ of $F_r$, where $\theta_i'=[x_i'\mapsto y_i'x_i']$.

We claim that the composition $g_{\theta_{pm}'}\circ \dots \circ g_{\theta_1'}$ is admissible. 
Indeed, $g^p \colon R_r\to R_r$ is a train track map with exactly one non-degenerate illegal turn. Therefore $Dg^p(A^{\pm 1})$ consists of $2r-1$ distinct directions. 

Suppose that the sequence $\theta_1',\dots, \theta_{pm}'$ is not admissible. Let $i\ge 1$ be the smallest index such that the pair $(\theta_i'=[x_i'\mapsto y_i'x_i'],\theta_{i+1}'=[x_{i+1}'\mapsto y_{i+1}'x_{i+1}'])$ is not admissible.
Then for $g_i=g_{\theta_i'}\circ \dots \circ g_{\theta_1'}$, by Lemma~\ref{L:reg}, we have $Dg_i=A^{\pm 1}-\{x_i'\}$, with $Dg_i(x_1')=Dg_i(y_1')$.
The only illegal turn for $g_{\theta_{i+1}'}$ is $\{x_{i+1}', y_{i+1}'\}$, and $Dg_{\theta_{i+1}'}(x_{i+1}') =Dg_{\theta_{i+1}'}(y_{i+1}')$. The fact that the pair $(\theta_i',\theta_{i+1}')$ is not admissible means that $x_{i+1}'\ne x_i'$ and  $y_{i+1}'\ne x_i'$, which means that $Dg_{\theta_{i+1}'}$ identifies two distinct directions in  $A^{\pm 1}-\{x_i'\}$. It follows that $Dg^p(A^{\pm 1})$ consists of $\le 2r-2$ directions, yielding a contradiction.

Since $g^{2p} \colon R_r\to R_r$ is also a train track map with exactly one nondegenerate illegal turn, the same argument implies that the composition $g_{\theta_{pm}'}\circ \dots \circ g_{\theta_1'}\circ g_{\theta_{pm}'}\circ \dots \circ g_{\theta_1'}$ is also admissible. Hence the composition $g^p=g_{\theta_{pm}'}\circ \dots \circ g_{\theta_1'}$ is cyclically admissible, as required.

\end{proof}
Note that, as the above proof shows, the power $p\ge 1$ in the conclusion of Theorem~\ref{t:decomp} can be chosen independent of the choice of $g$. In particular, if $p_0$ is the least common multiple of the orders of all the elements in the symmetric group $S_r$, then $p=2p_0$ works for all $g$ as in Theorem~\ref{t:decomp}, since for every permutational $\Psi\in\Aut(F_r)$ we have $\Psi^p=1$ in $\Aut(F_r)$ and $g_{\Psi}^p=Id_{R_r}$.

\bibliographystyle{amsalpha}

\providecommand{\bysame}{\leavevmode\hbox to3em{\hrulefill}\thinspace}
\providecommand{\MR}{\relax\ifhmode\unskip\space\fi MR }
% \MRhref is called by the amsart/book/proc definition of \MR.
\providecommand{\MRhref}[2]{%
  \href{http://www.ams.org/mathscinet-getitem?mr=#1}{#2}
}

\providecommand{\href}[2]{#2}

\end{document}